\documentclass[leqno,a4paper]{amsart}
\usepackage{geometry}
\usepackage[utf8]{inputenc}

\usepackage[dvipsnames]{xcolor}
\definecolor{ColorOne}{HTML}{5EA4FF}
\newcommand{\colorOne}[1]{\textcolor{ColorOne}{#1}}
\definecolor{ColorTwo}{HTML}{ffa217}
\newcommand{\colorTwo}[1]{\textcolor{ColorTwo}{#1}}
\usepackage{tikzit}

\tikzstyle{black dot}=[fill=black, draw=none, shape=circle, scale=0.4]

\tikzstyle{c0}=[-, line width=0.1em]
\tikzstyle{c0 arrow}=[->, line width=0.1em]
\tikzstyle{c0 bold}=[-, line width=0.15em]
\tikzstyle{c1}=[-, draw={rgb,255: red,94; green,164; blue,255}, line width=0.1em]
\tikzstyle{c1 arrow}=[->, draw={rgb,255: red,94; green,164; blue,255}, line width=0.1em]
\tikzstyle{c2}=[-, draw={rgb,255: red,255; green,162; blue,23}, line width=0.1em]
\tikzstyle{c2 arrow}=[->, draw={rgb,255: red,255; green,162; blue,23}, line width=0.1em]
\tikzstyle{d1}=[-, dashed, line width=0.1em]
\tikzstyle{d1 arrow}=[->, line width=0.1em, dashed]

\usepackage{comment}
\usepackage{epsfig}
\usepackage{graphicx}
\usepackage{amsbsy}
\usepackage{amsmath}
\usepackage{amsfonts}
\usepackage{amssymb}
\usepackage{textcomp}
\usepackage{hyperref}
\usepackage{aliascnt}
\usepackage{bbm}

\usepackage{enumitem}\setlist[enumerate]{label=\textnormal{(\roman*)}}
\renewcommand{\leq}{\leqslant}
\renewcommand{\geq}{\geqslant}
\renewcommand{\le}{\leqslant}
\renewcommand{\ge}{\geqslant}
\newcommand{\pmo}{^{\pm1}}
\DeclareMathOperator{\nequiv}{\not\equiv}

\newcommand{\mcm}[3]{\newcommand{#1}[#2]{{\ensuremath{#3}}}} 

\mcm{\tuple}{1}{\langle #1 \rangle}
\mcm{\name}{1}{\ulcorner #1 \urcorner}
\mcm{\Nbb}{0}{\mathbb{N}}
\mcm{\N}{0}{\mathbb{N}}
\mcm{\Zbb}{0}{\mathbb{Z}}
\mcm{\Rbb}{0}{\mathbb{R}}
\mcm{\Cbb}{0}{\mathbb{C}}
\mcm{\Qbb}{0}{\mathbb{Q}}
\mcm{\Acal}{0}{\cal A}
\mcm{\Bcal}{0}{\mathcal B}
\mcm{\Ccal}{0}{\cal C}
\mcm{\Dcal}{0}{\cal D}
\mcm{\Ecal}{0}{\cal E}
\mcm{\Fcal}{0}{\cal F}
\mcm{\Gcal}{0}{\cal G}
\mcm{\Hcal}{0}{\cal H}
\mcm{\Ical}{0}{\cal I}
\mcm{\Jcal}{0}{\cal J}
\mcm{\Kcal}{0}{\cal K}
\mcm{\Lcal}{0}{\cal L}
\mcm{\Mcal}{0}{\cal M}
\mcm{\Ncal}{0}{\cal N}
\mcm{\Ocal}{0}{{\cal O}}
\mcm{\Pcal}{0}{{\cal P}}
\mcm{\Qcal}{0}{{\cal Q}}
\mcm{\Rcal}{0}{{\cal R}}
\mcm{\Scal}{0}{{\cal S}}
\mcm{\Tcal}{0}{{\cal T}}
\mcm{\Ucal}{0}{{\cal U}}
\mcm{\Vcal}{0}{{\cal V}}
\mcm{\Wcal}{0}{{\cal W}}
\mcm{\Xcal}{0}{{\cal X}}
\mcm{\Ycal}{0}{{\cal Y}}
\mcm{\Zcal}{0}{{\cal Z}}
\mcm{\Mfrak}{0}{\mathfrak M}

\mcm{\restric}{0}{\upharpoonright}
\mcm{\upset}{0}{\uparrow}
\mcm{\onto}{0}{\twoheadrightarrow}
\mcm{\smallNbb}{0}{{\small \mathbb{N}}}
\DeclareMathOperator{\preop}{op}
\mcm{\op}{0}{^{\preop}}

\newcommand{\se}{\subseteq}
\newcommand{\sm}{\setminus}
\newcommand{\unit}{\mathbb{I}}

{\begin{array}{c}
\setlength{\unitlength}{1em}}%
{\end{array}}

\usepackage{amsthm}

\newcommand{\theoremize}[2]{\newaliascnt{#1}{thm} \newtheorem{#1}[#1]{#2} \aliascntresetthe{#1}}

\theoremstyle{plain}
\newtheorem{thm}{Theorem}[section]
\theoremize{lem}{Lemma}

\theoremize{skolem}{Skolem}
\theoremize{fact}{Fact}
\theoremize{setting}{Setting}
\newtheorem{sublem}{Sublemma}[thm]

\theoremize{claim}{Claim}
\theoremize{obs}{Observation}
\theoremize{prop}{Proposition}
\theoremize{cor}{Corollary}
\theoremize{que}{Question}
\theoremize{oque}{Open Question}
\theoremize{opro}{Open Problem}
\theoremize{con}{Conjecture}

\theoremstyle{definition}
\theoremize{dfn}{Definition}
\theoremize{rem}{Remark}
\theoremize{eg}{Example}
\theoremize{exercise}{Exercise}
\theoremstyle{plain}

\newtheorem{main}{Theorem}

\newenvironment{cproof}{\noindent\textit{Proof of Sublemma.}}{\phantom{!}\!\hfill$\diamondsuit$}

\hypersetup{
    draft = false,
    bookmarksopen=true,
    colorlinks,
    linkcolor={red!60!black},
    citecolor={green!60!black},
    urlcolor={blue!60!black}
}

\renewcommand{\i}{^{-1}}
\renewcommand{\r}[1]{\textcolor{red}{\underline{#1}}}
\renewcommand{\b}[1]{\textcolor{blue}{\underline{#1}}}
\newcommand{\g}[1]{\textcolor{teal}{\underline{#1}}}

\newcommand{\grey}[1]{\textcolor{gray}{#1}}

\newcommand{\red}[1]{{\color{red}{#1}}}

\newcommand{\cay}[2]{\textnormal{Cay}(#1,#2)}

\newcommand{\X}{\{\unit,h\}}

\newcommand{\vvv}{v\hspace{-0.2em}v\hspace{-0.2em}v}

\usepackage{caption}

\title{Towards a Stallings-type theorem for finite groups}
\author[Johannes Carmesin]{Johannes Carmesin${}^\clubsuit$}
\author[George Kontogeorgiou]{George Kontogeorgiou${}^\heartsuit$}
\author[Jan Kurkofka]{Jan Kurkofka${}^\clubsuit$}
\author[Will J.\ Turner]{Will J.\ Turner${}^\diamondsuit$}

\thanks{${}^\clubsuit$ University of Birmingham, Birmingham, UK, funded by EPSRC, grant number EP/T016221/1}
\thanks{${}^\heartsuit$ University of Warwick, Coventry, UK, funded by EPSRC Doctoral Training Partnership}
\thanks{${}^\diamondsuit$ University of Birmingham, Birmingham, UK, funded by EPSRC, CDT in Topological Design EP/S02297X/1}


\begin{document}
\thispagestyle{empty}

\topskip0pt
\vspace*{\fill}
\maketitle

\noindent\textsc{Abstract.}
A recent development in graph-minor theory is to study \emph{local separators}, vertex-sets that separate graphs locally but not necessarily globally.
The local separators of a graph roughly correspond to the genuine separators of its \emph{local covering}: a usually infinite graph obtained by keeping all local structure of the original graph while unfolding all other structure as much as possible.

We use local separators and local coverings to discover and prove a low-order Stallings-type result for finite nilpotent groups~$\Gamma$:
the $r$-local covering of some Cayley graph $G$ of $\Gamma$ has $\ge 2$ ends that are separated by $\le 2$ vertices iff $G$ has an $r$-local separator of size $\le 2$ and $\Gamma$ has order~$>r$, iff $\Gamma$ is isomorphic to $C_i\times C_j$ for some $i>r$ and $j\in\{1,2\}$.

\vspace*{\fill}

\thispagestyle{empty}

\newpage
\setcounter{page}{1}


\section{Introduction}

Stallings theorem, usually regarded as a founding result of geometric group theory and of Bass-Serre theory more specifically, states that a finitely generated group has at least two ends if and only if it decomposes as a non-trivial amalgamated free product or HNN-extension over a finite subgroup~\cite{stallings1968torsion,stallings1972group}.
It is not known how to extend Stallings theorem to finite groups.
The purpose of this paper is to suggest a new approach via local separations and local coverings, two recently developed concepts from the theory of graph-minors, and to demonstrate that this approach is feasible by discovering and proving a low-order Stallings-type theorem for finite nilpotent groups.

There are two major challenges when searching for a Stallings-type theorem for finite groups.
The first challenge is that the notion of ends is meaningful only for infinite groups and graphs.
The second challenge arises from a key step in finding the decomposition of a given multi-ended group~$\Gamma$.
Here one considers a Cayley graph $G$ of $\Gamma$ and discovers a non-empty set $N(G)$ of pairwise non-crossing separations\footnote{A \emph{separation} of $G$ is a pair $\{A,B\}$ such that $A\cup B=V(G)$, there is no edge in $G$ between $A\sm B$ and $B\sm A$, and both $A\sm B$ and $B\sm A$ are nonempty. Its \emph{separator} is $A\cap B$, and its \emph{order} is $|A\cap B|$. In the group-theoretic literature, edge-cuts are often used instead of separations, but both perspectives work here.} of~$G$ such that the action of $\Gamma$ on $G$ induces an action of $\Gamma$ on~$N(G)$ that is transitive.
But as soon as $\Gamma$ is finite, it no longer induces a well-defined action on any set of pairwise non-crossing separations of~$G$.

A recent development in graph-minor theory is the study of \emph{$r$-local separations}: separations that separate a graph locally in a ball of a previously chosen radius $r/2>0$, but which need not separate the entire graph, roughly speaking~\cite{carmesin}.
Sets $N$ of pairwise non-crossing $r$-local separations can be found in
Cayley graphs $G$ of finite groups $\Gamma$ such that the action of $\Gamma$ on $G$ induces an action of $\Gamma$ on~$N$ that is even transitive; see \autoref{fig:loc-sepr}.
The $r$-local separations of $G$ correspond\footnote{The correspondence works for order at most two, but requires mild assumptions for higher order.} to genuine separations of the \emph{$r$-local covering} $G_r$ of $G$, which is another graph obtained from $G$ by preserving all balls of radius $r/2$ while unfolding all other structure as much as possible~\cite{GraphDec}.
The $r$-local covering $G_r$ is usally infinite.

\begin{figure}[ht]
    \centering
    \includegraphics[height=12\baselineskip]{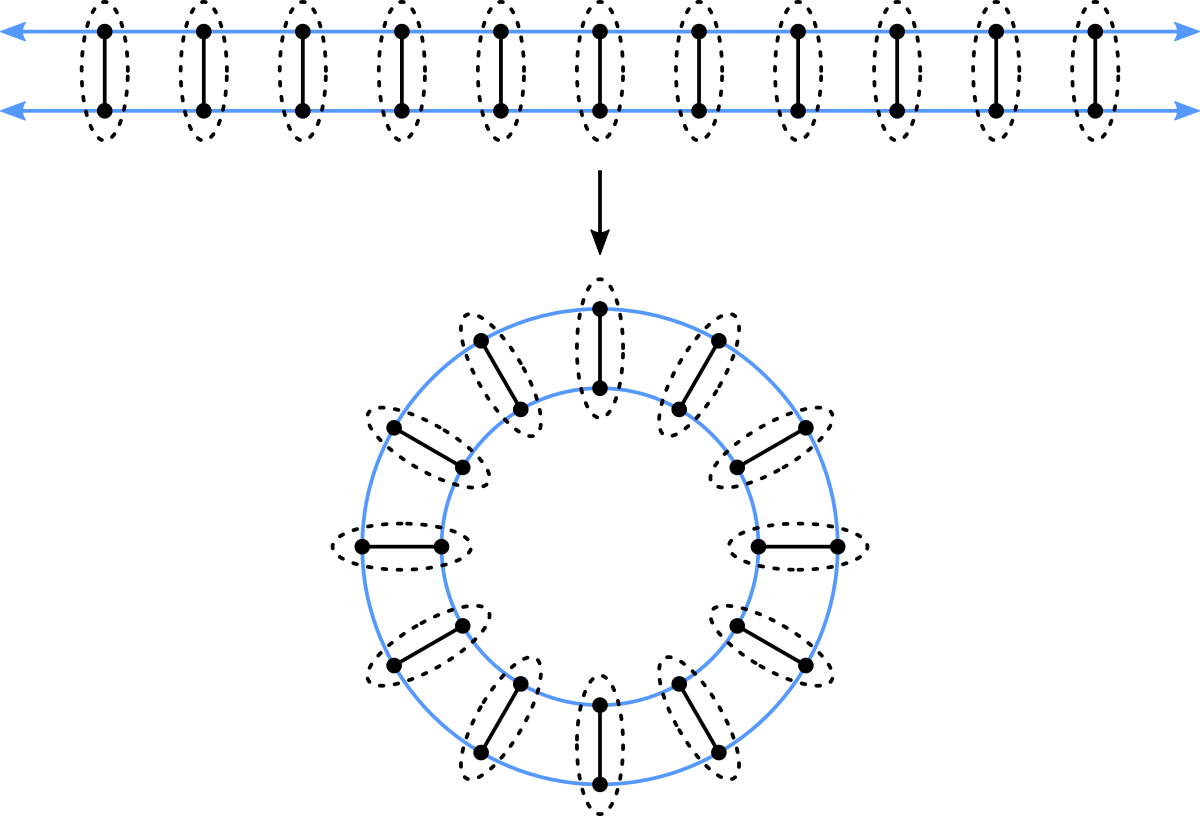}
    \caption{Bottom: The Cayley graph $G$ of $C_{12}\times C_2=:\Gamma$. 
    The twelve circled vertex-sets are the separators of a set $N$ of pairwise nested 11-local 2-separations of~$G$.
    The group $\Gamma$ acts transitively on~$N$.
    Every element of $N$ lifts to a 2-separation of the 11-local covering of $G$ (top) where it separates two ends.}
    \label{fig:loc-sepr}
\end{figure}

To overcome the two challenges outlined above, we suggest considering the ends of the $r$-local coverings of Cayley graphs of finite groups~$\Gamma$ on the one hand, and using $r$-local separations instead of genuine separations on the other hand.
As our main result, we show that this approach works for $r$-local separations of low order in conjunction with the assumption of nilpotency, yielding the following Stallings-type theorem:

\begin{main}\label{3-con}
    Let $\Gamma$ be a finite group that is nilpotent of class~$\le n$.
    Let $r\ge \max\{4^{n+1},20\}$.
    Then the following assertions are equivalent:
    \begin{enumerate}
        \item\label{StallingsCovering} The $r$-local covering of some Cayley graph of $\Gamma$ has $\ge 2$ ends that are separated by $\le 2$ vertices.
        \item\label{StallingsLocalSeparators} Some Cayley graph of $\Gamma$ has an $r$-local separator of size $\le 2$ and $\Gamma$ has order~$>r$.
        \item\label{StallingsProduct} $\Gamma$ is isomorphic to $C_i\times C_j$ for some $i>r$ and $j\in\{1,2\}$.
    \end{enumerate}
\end{main}

\autoref{3-con} does not extend to all finite groups, as witnessed by simple finite groups; see \autoref{fig:cutv_ex}.
However, we believe that extending \autoref{3-con} to solvable groups is in reach:

\begin{opro}
    Extend \autoref{3-con} to solvable groups.
\end{opro}

\begin{figure}[ht]
    \centering
    \includegraphics[height=8\baselineskip,angle=180,origin=c]{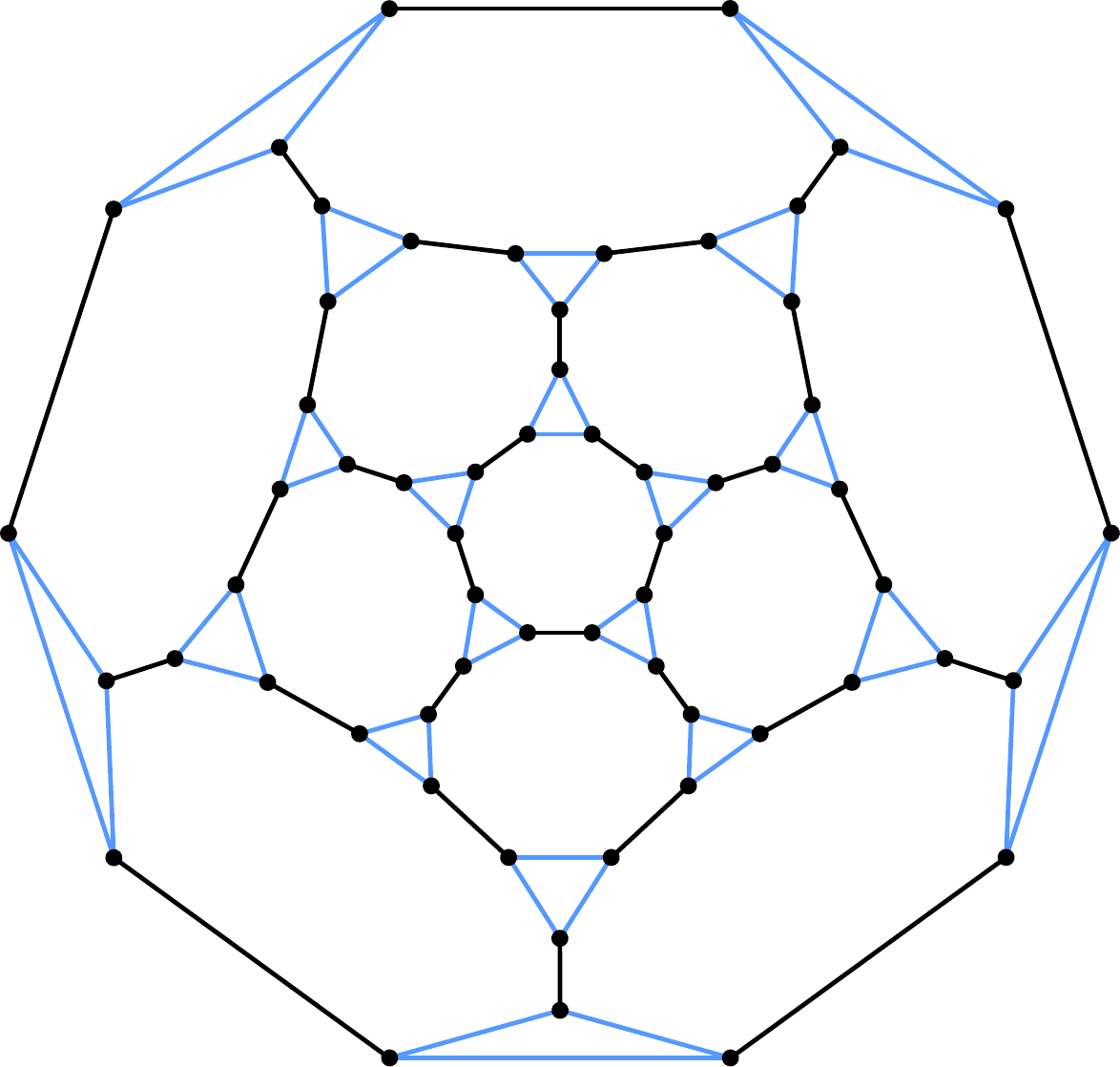}
\caption{Every vertex of the Cayley graph of $A_5\cong\langle a,b:a^3,b^2,(ab)^5\rangle$ is a 9-local cutvertex. However, these local cutvertices cannot give rise to interesting product structure, as alternating groups $A_n$ are simple for $n\ge 5$.
More generally, for every $r>0$ there is $n$ such that every vertex of some Cayley graph of $A_n$ is an $r$-local cutvertex, see \autoref{UnboundedLocalCutvertex}.}
\label{fig:cutv_ex}
\end{figure}

The main challenge when fully extending Stallings theorem to finite groups via local separations and local coverings is to deal with local separations of higher order.
Indeed, a key ingredient in the proof of \autoref{3-con} is a version for local separations of Tutte's theorem for 2-connectivity~\cite{carmesin}.
A Tutte-type theorem for 3-connectivity has been found and proved only recently~\cite{Tridecomp}, but it is an open problem to extend it to the setting of local separations.
It is not known how to find Tutte-type theorems for higher connectivity, let alone extending them to the local setting.

\begin{opro}
    Extend \autoref{3-con} to $r$-local separators of size $\ge 3$.
\end{opro}

\subsection*{Related work}

Stallings theorem builds on the work of Freudenthal~\cite{Freudenthal}, Hopf~\cite{Hopf} and Wall~\cite{Wall}.
It has since been used in various works on highly symmetric graphs, for example by Georgakopoulos, Hamann, Lehner, Miraftab, M\"oller~\cite{OneEndedTDCs,Geo_plane_cubic,Accessibility2018Hamann,hamann2018stallings} and by Woess~\cite{WoessAmenable,WoessStallings}.
Recent work on local separations and local coverings includes \cite{albrechtsen2023structural,albrechtsen2023induced,bumpus2023structured,carmesin2024structure,carmesin2023apply,GraphDec}.

\subsection*{Overview of the proof}

We focus on the hard implication of \autoref{3-con}, which is \ref{StallingsLocalSeparators}$\to$\ref{StallingsProduct}.
For this, let $G$ be the Cayley graph of a finite group~$\Gamma$ that is nilpotent of class~$\le n$ with generating set~$S$.
Let $X$ be an $r$-local $(\le 2)$-separator of~$G$ with $r\gg n$, and assume that $\Gamma$ has order~$>r$.
The aim is to show that $\Gamma$ is isomorphic to $C_i\times C_j$ for some $i>r$ and $j\in\{1,2\}$.

To harness the separating property of~$X$, we will introduce tools with which we can find cycles of length $\le r$ that traverse a fixed vertex through a pair of prescribed edges.
More specifically, the $n$th iterated commutator $[g,h]_n$ labels the edges of a closed walk in~$G$, and it is straightforward to find a subwalk $O$ that runs once around a cycle of length~$\le r$.
If $O$ has length at least four, then $O$ and its inverse together contain nearly all 2-letter-combinations of $g$ or $g^{-1}$ with $h$ or $h^{-1}$.
We will call this property \emph{magic} and prove the above fact in the \nameref{magic_cor}~(\ref{magic_cor}).
Since $G$ is vertex-transitive, we can move $O$ around in~$G$ so that a prescribed vertex lies in the middle of a desired 2-letter-combination.

In the case where $X$ is a local cutvertex, we will use such tools to show that short cycles link up the neighbours of~$X$ so that we can deduce that $S$ is trivial.
Hence the group $\Gamma$ will be cyclic; see \autoref{submainCutvertex}.

In the other case, $X$ is a local $2$-separator.
The main difficulty is that no iterated commutator seems to be useful for constructing short cycles that could help us harness the separating property of~$X$.
Our plan is to circumvent this problem entirely by extending $S$ so that $X$ is spanned by an edge, labelled $h$ (say).
Since $h$ will also span every local $2$-separator in the $\Gamma$-orbit of~$X$, it is not clear that $X$ will be a local $2$-separator after adding~$h$ to~$S$.
To resolve this, we will use the local Tutte-decomposition~\cite{carmesin} to replace $X$ with a local $2$-separator that is nested with its orbit under the action of~$\Gamma$ on~$G$.
Adding $h$ will also incur a toll on~$r$, as $X$ might become an $r'$-local $2$-separator for a smaller $r'\le r$, but the toll will be manageable.
This will be the \nameref{long_edge}~(\ref{long_edge}).
We then distinguish the following two cases.

Case 1: $h$ is not an involution.
In this case, we will investigate short cycles from iterated commutators of the form $[g,h]_n$ to show that $S$ must be a subset of $\{h^{\pm 1},h^{\pm 2}\}$.
Then $\Gamma$ will be cyclic.
We will deal with this case in \autoref{sec:noninvo}, where the key result will be \autoref{lem:involution_free}.

Case 2: $h$ is an involution.
Here we will show that the subgroup $\langle h\rangle$ of~$\Gamma$ generated by $h$ is normal.
Taking the quotient $\Gamma/\langle h\rangle$ roughly corresponds to contracting all the edges in $G$ labelled by~$h$.
Since $X$ is a local $2$-separator of~$G$, it defines a local cutvertex of the resulting contraction minor, so the quotient $\Gamma/\langle h\rangle$ will be cyclic by the local-cutvertex case (\autoref{submainCutvertex}).
Then it will follow that $\Gamma$ is the direct product of a cyclic group with~$C_2$.
We will deal with this case in \autoref{sec:invo}, where the key result will be \autoref{lem:involution_case}.

\subsection{Organisation of the paper}
We introduce the terminology that we use in \autoref{sec:graphs}.
Next we show in \autoref{sec:commutators} that iterated commutators have the magic property.
We use the magic property in \autoref{sec:cutvertices} to deal with local cutvertices.
Local $2$-separators are introduced in detail in \autoref{sec:2sep} where we also prove the \nameref{long_edge}~(\ref{long_edge}).
In \autoref{sec:proof} we explain in great detail how we plan to deal with local $2$-separators.
The case where the given local $2$-separator is spanned by a non-involution will be addressed in \autoref{sec:noninvo}, whereas the involution case will be solved in \autoref{sec:invo}.

\subsection{Acknowledgements}

We are grateful to David Hume for a valuable discussion.

\section{Terminology}
\label{sec:graphs}

For graph-theoretic terminology we follow Diestel \cite{diestelBook}.
    A \emph{walk} (of \emph{length}~$k$) in a graph~$G$ is a non-empty sequence 
    \[W=v_0 e_0 v_1\ldots v_{k-1} e_{k-1} v_k\] 
    alternating between vertices and edges in~$G$ such that $e_i$ has endvertices $v_i$ and $v_{i+1}$ for all~$i<k$.
    A walk is \emph{closed} if $v_0=v_k$. We refer to $k$ as the \emph{length} of $W$. 
    A walk is \emph{trivial} if it has length~0, i.e. if it consists of a single vertex. 
    A walk $W'$ is a \emph{subwalk} of~$W$ if $W'=v_i e_i v_{i+1}\ldots v_{j-1} e_{j-1} v_j$ for some indices $i,j$ with $0\le i\le j\le k$.
Given a graph $G$ with a vertex $v$ and $r\in\N$, the \emph{ball around~$v$ of diameter~$r$} is the subgraph of~$G$ whose vertices and edges are the ones that lie on closed walks in $G$ of length at most~$r$ that pass through~$v$. We denote it by $B_r(v,G)$, and simply by $B_r(v)$ if $G$ is implicitly given by the context. 
We say that $v$ is an \emph{$r$-local cutvertex} of $G$ if it is a cutvertex of the ball $B_r(v)$; that is, if the punctured ball $B_r(v)-v$ is disconnected.
    An $r$-local cutvertex is also an $r'$-local cutvertex for every $r'$ satisfying $0<r'<r$.

We use the standard group-theoretic terminology as found in Hall \cite{Hall}. 
In what follows, let $\Gamma$ be a finite group with generating set~$S$. We follow the convention that a \emph{generating set} $S$ does not contain the neutral element and is closed under taking inverses.
We denote the neutral element of~$\Gamma$ by~$\unit$.
In this paper, we will carefully distinguish between \textit{words in~$S$} (i.e.\ finite sequences of elements of $S$), \textit{reduced words in $S$} (i.e.\ elements of the free group $F(S)$), and elements of the group $\Gamma$ (which can be represented by words, reduced or not, in a not necessarily unique way). 
    We say that two group elements $a$ and $b$ are \emph{equivalent}, and we write~$a\equiv b$, if $a=b$ or $a=b\i$.
    The \emph{Cayley graph} $\cay{\Gamma}{S}$ is the simple graph with vertex set $\Gamma$ and edge set $\{\,\{g,gs\}\mid g\in \Gamma, s\in S\,\}$, where each edge between vertices $x$ and $y$ has two \emph{orientations}, one corresponding to $s$ such that $y=xs$ and the other to $s^{-1}$, which satisfies $x=y s^{-1}$. 
Right-multiplication of an element $g$ by a generator $s$ corresponds to traversing the edge $\{g,gs\}$ in the direction from $g$ to $gs$ in the Cayley graph~$G$.
The group~$\Gamma$ acts on its Cayley graph $G$ by left-multiplication, so every $h\in\Gamma$ maps edges $\{g,gs\}$ to edges $\{(h g),(h g)s\}$.
Each word $w$ in $S$ determines a walk in $G$ that starts at the neutral element $\unit$ and traverses edges of $G$ labelled by the generators appearing in $w$ in order from left to right. 

We will use the same symbols to denote group elements, edge-labels and vertices. For example, given a group $\Gamma$, a generating set $S$ and the corresponding Cayley graph $G=\cay{\Gamma}{S}$, a group element $g\in\Gamma$ will label a vertex $g\in V(G)$. 
If $g\in S$, then every vertex $x\in V(G)$ will be incident with two edges labelled by~$g$, one from $xg\i$ to $x$ and one from $x$ to~$xg$, except when $g$ is an involution, in which case the two edges coincide. In the text we will compensate this compressed notation by writing out the role of $g$, such as `group element', `vertex' and `edge labelled by'.

Given a group $\Gamma$ and a subset $X\subseteq\Gamma$,
an \emph{identity word} of~$\Gamma$ in~$X$ is a nonempty word $x_1\ldots x_n$ with $x_i\in X$ for all $i\in [n]$ such that $x_1\ldots x_n=\unit$ in~$\Gamma$.
An identity word $x_1\ldots x_n$ is \emph{trivial} if $n=1$.

\begin{dfn}[Morpheme]
    Let $\Gamma$ be a group.
    A \emph{morpheme} of $\Gamma$ in~$S$ is an identity word of~$\Gamma$ in~$S$ (in particular, distinct from the word~$\unit$) such that no nonempty proper subword of it is an identity word.
\end{dfn}

\begin{eg}\label{cyclesGiveMorphemes}
Let $G:=\cay{\Gamma}{S}$ be a Cayley graph.
From every cycle in~$G$ we can obtain morphemes by choosing a startvertex and then reading the edge-labels around one of its two directions.
In fact, all morphemes of~$\Gamma$ in~$S$ can be obtained in this way, except for the ones of the form $h^2$ where $h$ is an involution.
\end{eg}

\begin{lem}[Folklore]
\label{lem:cycle_finder}
    Every nontrivial non-backtracking \footnote{A closed walk $v_0e_0\ldots v_k$ is \emph{backtracking} if there is $i\in \mathbb{Z}_k$ such that $v_i=v_{i+2}$.} closed walk contains a cycle as a subwalk.
\end{lem}
\begin{proof}
Pick a nontrivial closed subwalk that is of minimum length.  
\end{proof}

\begin{lem}[Morpheme Finding Lemma]\label{commutatorContainsMorpheme}
Let $\Gamma$ be a group, let $g_1,\ldots,g_k\in \Gamma$ with $g_i\nequiv g_j$ for all distinct $i,j\in [k]$, and let $u$ be a nontrivial reduced word in $g_1\pmo,\ldots,g_k\pmo$.

If $u=\unit$ in~$\Gamma$, then $u$ contains a morpheme~$m$ of $\Gamma$ in $g_1\pmo,\ldots,g_k\pmo$ as a subword, and $m$ has length at least three or is of the form $g_i^{\pm 2}$ for some~$i\in[k]$.
\end{lem}
\begin{proof}
    If $u$ contains $g_i^{\pm 2}$ as a subword for some $i\in [k]$, then $m:=g_i^2$ or $m:=g_i^{-2}$ is the desired morpheme.
    Otherwise we let $S:=\{g_1,\ldots ,g_k\}$ and consider the walk in $\cay{\Gamma}{S}$ that starts at $\unit$ and is determined by~$u$.
    This walk is not backtracking\footnote{A walk $v_0 e_0\ldots v_k$ is \emph{backtracking} if there is $i\in \{0,\ldots,k-2\}$ such that $v_i=v_{i+2}$.} since $u$ does not contain $g_i^{\pm 2}$ as a subword for any~$i\in [k]$.
    By applying \autoref{lem:cycle_finder} to the walk, we find a cycle in it as a subwalk, from which we read the desired morpheme~$m$.
\end{proof}

\section{Iterated commutators} \label{sec:commutators}

\begin{dfn}[Iterated commutator word]
    Let $g,h\in\Gamma$. The $n$-th \emph{iterated commutator word} is the word defined recursively by
    \begin{align*}
        [g,h]_1 &:=gh\i g\i h, \text{ and}\\
        [g,h]_n &:=[g,[g,h]_{n-1}]_1 \text{ after reduction}.
    \end{align*}
\end{dfn}

We remark that the notion of iterated commutators that we use slightly differs from notions commonly used in the literature.
We work with the above notion as it will be more convenient for our proofs.
However, all notions characterise the same nilpotency classes, and in this sense are equivalent.
We recall that a group $\Gamma$ is \emph{nilpotent of class $\le n$} for an integer $n\ge 1$ if $[g,h]_n = \unit$ for all inequivalent $g,h\in\Gamma$.
In the context of two given group-elements named $g$ and~$h$ with $g\nequiv h$ we abbreviate $[g,h]_n$ by~$u_n$. 

\begin{eg}\label{eg1}
The first three iterated commutator words in $g$ and $h$ are:
\begin{align*}
    u_1&=gh\i g\i h,\\
    u_2&=gh\i ghg\i h\i g\i h\text{, and}\\
    u_3&=gh\i ghgh\i g\i hg\i h\i g hg\i h\i g\i h.
\end{align*}
\end{eg}

\begin{dfn}[Truncation]
    Let $u=a_1\ldots a_n$ be a word with~$n\ge 2$ letters. 
    The \emph{left truncation} and \emph{right truncation} of~$u$ are the subwords ${}^-u := a_2\ldots a_n$ and $u^- := a_1\ldots a_{n-1}$, respectively.
\end{dfn}

\begin{lem}
\label{lem:what_words_look_like}
    Let $g\nequiv h$ be letters.
    Then the word $u_n$ can be written as $u_n=g(u_{n-1}\i )({}^-u_{n-1})$ for all $n\ge 2$.
    For $n\ge 4$ we obtain the following more detailed description of this form:
    \begin{equation}\label{eq:unRecursiveForm}
        u_n =g\,\underset{u_{n-1}\i}{\underbrace{(h\i ghgh\i g\i h\ldots ghg\i h\i g\i h g\i)}}\;\underset{{}^{-}u_{n-1}}{\underbrace{(h\i g h g h\i g\i \ldots h\i gh g\i h\i g\i h)}}
    \end{equation}
\end{lem}
\begin{proof}
    We proceed by induction on~$n$. 
    The base cases $n=2$ and $n=3$ are true by \autoref{eg1}, so let~$n>3$.
    By definition, the word $u_n$ is obtained from $u':=gu_{n-1}\i g\i u_{n-1}$ by reduction. By the induction hypothesis (or using the explicit form of $u_3$ provided by \autoref{eg1} if $n-1=3$) we have
    \begin{equation*}
        u'=g(h\i ghgh\i g\i h\ldots ghg\i h\i g\i h g\i)\red{g\i} (\red{g}h\i g h g h\i g\i\ldots h\i gh g\i h\i g\i h)
    \end{equation*}
    where the terms in brackets are already in reduced form.
    Therefore, only the subword formed by the two red letters is reduced to obtain $u_n$ from~$u'$, which yields $u_n=g(u_{n-1}\i)( {}^-u_{n-1})$ and (\ref{eq:unRecursiveForm}) as desired.
\end{proof}

\begin{cor}\label{commutatorLength}
    $|u_n|=2^{n+1}$ for all~$n\ge 1$.\qed
\end{cor}

\begin{dfn}[Cyclic permutation]
A \emph{cyclic permutation} of a word $w=a_1\ldots a_n$ is a word of the form $w'=a_k\ldots a_n a_1\ldots a_{k-1}$ for some $k\in [n]$; here we follow the convention that $a_i\ldots a_i=a_i$ and $a_1\ldots a_0$ is the empty word.
\end{dfn}

\begin{dfn}[Cyclic subword]
Let $u$ and $w$ be two words.
We say that $u$ is a \emph{cyclic subword} of~$w$ if $u$ or its inverse~$u\i$ is a subword of a cyclic permutation of~$w$. 
\end{dfn}

If $u$ is a subword of~$w$, we also say that $u$ is a \emph{linear subword} of~$w$.

\begin{eg}
Given letters $a,b,c$, the word $ca$ is a cyclic subword of $abc$.
The word $a\i c\i$ is a cyclic subword of $abc$ because its inverse $ca$ is a subword of the cyclic permutation $cab$ of $abc$.
The words $ac$ and $c\i a\i$ are not cyclic subwords of $abc$.
\end{eg}

\begin{dfn}[Square words]
     Given letters $g\nequiv h$, the \emph{square words in $g$ and $h$} are the words that can be obtained by cyclic permutations or inversions of $ghgh$ and $gh\i gh\i$.
We omit `in $g$ and $h$' when $g$ and $h$ are implicitly given by the context.
\end{dfn}

\begin{eg}\label{square_words}
    The square words in $g$ and~$h$ are:
    \[
    \begin{array}{c|c|c||c|c}
        & & \text{cyclic permutation} & & \text{cyclic permutation}\\ \hline
        &ghgh &hghg &gh\i gh\i &h\i gh\i g \\ \hline
        \text{inversion} &h\i g\i h\i g\i &g\i h\i g\i h\i &hg\i hg\i &g\i hg\i h.
    \end{array}
    \]
    
\end{eg}

\begin{lem}[No-Squares Lemma]\label{lem:nosquares}
    For all $n\ge1$, the iterated commutator word $u_n:=[g,h]_n$ contains no square words as cyclic subwords.
\end{lem}

\begin{proof}
    We are required to show that no cyclic permutation of $u_n$ contains any of the words from \autoref{square_words}.
    We proceed by induction on~$n$.
    For $n\le 2$, we can see this by inspecting the words in \autoref{eg1}.
    So let $n\ge 3$.
    By \autoref{lem:what_words_look_like}~(\ref{eq:unRecursiveForm}) and using the explicit form of $u_3$ provided by \autoref{eg1}, we have
    \begin{equation}\label{eq:NoSquares}
        u_n=g\,\underset{u_{n-1}\i}{\underbrace{(h\i gh\ldots g\i h g\i)}}\;\underset{{}^{-}u_{n-1}}{\underbrace{(h\i g h \ldots h\i g\i h)}}.
    \end{equation}
    Words that are entirely contained in either of $u_{n-1}\i$ or ${}^{-}u_{n-1}$ are not square by the induction hypothesis.
    Beyond those, there are seven subwords of length four of cyclic permutations of $u_n$ that we need to check; we find these either as subwords of $u_n$ that are contained in neither $u_{n-1}\i$ nor ${}^{-}u_{n-1}$, and we also find some of them in the same way but after cyclically permuting~$u_n$.
    Explicitly, by inspecting (\ref{eq:NoSquares}) from left to right we find that these are the seven words $gh\i gh$, $g\i hg\i h\i$, $hg\i h\i g$, $g\i h\i gh$, $h\i g\i hg$, $g\i hgh\i$, and $hgh\i g$. None of the listed words are square words, which completes the proof.
\end{proof}

\begin{dfn}[Magic]
    Given letters $g\nequiv h$, a word $u$ is \emph{magic} (\emph{in $g$ and $h$}) if $u$ and $u\i$ together contain at least three of the following as linear subwords:
    \[
        gh,\quad hg,\quad g\i h,\quad hg\i
    \]
\end{dfn}

\begin{lem}\label{lem:magic}
    Given letters $g\nequiv h$, a word alternating between $g\pmo$ and $h\pmo$ of length at least four either is magic or contains a square word as a linear subword.
\end{lem}

\begin{proof}
    It suffices to prove the above for words of length exactly four.
    So let $u$ be a word alternating between $g\pmo$ and $h\pmo$ of length exactly four.
    Being magic and being a square word are invariant under applying any number of the following interchanges to the entire word:  
    $g\pmo\leftrightarrow g^{\mp 1}, h\pmo\leftrightarrow h^{\mp 1},g\pmo\leftrightarrow h\pmo$.
    Hence we may assume without loss of generality that $u$ starts with $g$ and ends with~$h$. So $u$ is one of the four words:
    \[
    ghgh,\quad g hg\i h,\quad gh\i gh,\quad gh\i g\i h
    \]
    The first word is square and the other three are magic.
\end{proof}

\begin{lem}[Magic Lemma]\label{magic_cor}
    All cyclic subwords of $u_n=[g,h]_n$ of length at least four are magic, for all $n\ge 1$.
\end{lem}

\begin{proof}
    This follows from \autoref{lem:nosquares} and \autoref{lem:magic}. 
\end{proof}

\begin{cor}\label{magic2}
        Let $\Gamma$ be a group that is nilpotent of class~$\le n$ with a generating set $S\subseteq\Gamma$ containing two inequivalent elements $g$ and $h$. 
    There exists a morpheme of $\Gamma$ in $\{g\pmo,h\pmo\}$ of length at most $r$ that is magic in $g$ and $h$, where $r=2^{n+1}$. 
\end{cor}

\begin{proof}
    Consider the iterated commutator $u_n$ in $g$ and~$h$, which has length at most~$r$ by \autoref{commutatorLength}. 
By the \nameref{commutatorContainsMorpheme} (\ref{commutatorContainsMorpheme}), we find a morpheme $u$ in~$u_n$ as a subword.
Since $u$ alternates in $g\pmo$ and $h\pmo$, the morpheme $u$ has length at least three.
If $u$ has length at least four, we are done by \autoref{magic_cor}. 
So assume that $u$ has length exactly three. 
Then one of $g$ or $h$ is a power of the other. 
Hence $g$ and $h$ commute. 
Thus the commutator $ghg^{-1}h^{-1}$ is a morpheme that is magic in $g$ and $h$, and the commutator has length at most $r=2^{n+1}$ since $n\ge 1$.
\end{proof}

In the context of two given elements of a generating set~$S$, named $g$ and~$h$, we abbreviate $w_n:=~[g,h^2]_n$.

\begin{eg}
The words $w_1$ and $w_2$ are
    $w_1=gh^{-2}g\i h^2$ and $w_2=gh^{-2}gh^2g\i h^{-2}g\i h^2$.
\end{eg}

\begin{obs}\label{obs:commlen2}
We have $|w_n|=\frac{3}{2}\cdot2^{n+1}$ for all $n\ge 1$.
\end{obs}
\begin{proof}[Proof:] implied by \autoref{commutatorLength}.
\end{proof}

\begin{dfn}
    Given a word $u$ and a letter $a$, an \emph{$a$-segment} of $u$ is a maximal linear subword of $u$ in the letter $a$. A \emph{segment} of $u$ is an $a$-segment of $u$ in some letter $a$.
\end{dfn}

\begin{lem}[Ramsey Lemma]\label{Ramsey}
    Let $g\nequiv h$ be letters. If a cyclic subword $w$ of the iterated commutator $w_n=[g,h^2]$ has at least four segments, then $w$ either contains at least one of $gh^2$ or $h^2g$ as a cyclic subword or $w=g\i h^2g\i h\i$ up to cyclic permutation or inversion.
\end{lem}

\begin{proof}
    Suppose that $w$ has a linear subword $w'$ that consists of four consecutive `full'  segments; that is, the $h$-segments of $w'$ have length two.
    By substituting $h\pmo$ for $h^{\pm2}$ in \autoref{magic_cor}, at least one of $gh^2$ or $h^2g$ is a linear subword of $w'$ or of the inverse of~$w'$ and so of $w$ or the inverse of~$w$, and so we are done here. Thus assume that $w$ has exactly four segments and they are not all full.
    
    By inverting or cyclically permuting $w$ if necessary, assume that $w$ has the form $w=g\pmo h^{2}g\pmo h\pmo$ where the `$\pm$' are independent of each other. Since otherwise we find $gh^2$ or $h^2g$ as a linear subword of $w$, the letters before and after the $h^2$ must both be $g\i$. If $w=g\i h^2g\i h$, then the iterated commutator contains the square $g\i h^2g\i h^2$, which is impossible by the
    \nameref{lem:nosquares} (\ref{lem:nosquares}) (applied with the substitution). Thus $w=g\i h^2g\i h\i$.
    This completes our proof.
\end{proof}

\section{Local cutvertices}
\label{sec:cutvertices}

In this section, we prepare to prove \autoref{3-con} in the local cutvertex case, by proving the following theorem:

\begin{thm}\label{submainCutvertex}
   Let $G$ be a Cayley graph of a group~$\Gamma$ that is nilpotent of class~$\le n$. Suppose that $r\geq 2^{n+1}$.
   If $G$ has an $r$-local cutvertex, then $G$ is a cycle longer than~$r$ and $\Gamma$ is cyclic.
\end{thm}

We will also show that the assumption of nilpotency is necessary, see \autoref{UnboundedLocalCutvertex}.
But first, we will prove \autoref{submainCutvertex}.

\begin{lem}
\label{lem:short_cycle_forall}
    Let $\Gamma$ be a group that is nilpotent of class~$\le n$, let $G:=\cay{\Gamma}{S}$ be a Cayley graph for a generating set $S$ containing two inequivalent elements, and let~$r\ge 2^{n+1}$. 
    Suppose that, for all $g,h\in S$ with $g\nequiv h$, there exists a morpheme of $\Gamma$ in $S$ of length at most $r$ that is magic in $g$ and $h$. Then $G$ has no $r$-local cutvertex.
\end{lem}

\begin{figure}[ht]
    \centering
\begin{tikzpicture}
	\begin{pgfonlayer}{nodelayer}
		\node [style=black dot] (4) at (0, 3) {};
		\node [style=black dot] (6) at (0, -3) {};
		\node [style=black dot] (7) at (3, 0) {};
		\node [style=black dot] (8) at (0, 0) {};
		\node [style=black dot] (9) at (-3, 0) {};
		\node [style=none] (10) at (-1.5, 0) {};
		\node [style=none] (11) at (0, 1.5) {};
		\node [style=none] (12) at (1.5, 0) {};
		\node [style=none] (13) at (0, -1.5) {};
		\node [style=none] (14) at (0.5, 1.5) {$h$};
		\node [style=none] (15) at (1.5, -0.5) {$\colorOne{g}$};
		\node [style=none] (16) at (0.5, -1.5) {$h$};
		\node [style=none] (17) at (-1.5, -0.5) {$\colorOne{g}$};
		\node [style=none] (18) at (-3.5, 3) {$P(gh)$};
		\node [style=none] (19) at (3.75, 3) {$P(g\i h)$};
		\node [style=none] (20) at (3.5, -3) {$P(hg)$};
		\node [style=none] (21) at (-3.75, -3) {$P(gh\i)$};
		\node [style=none] (22) at (0.5, -0.5) {$\unit$};
		\node [style=none] (23) at (0, 3.5) {$h$};
		\node [style=none] (24) at (0, -3.5) {$h\i$};
		\node [style=none] (25) at (3.5, 0) {$g$};
		\node [style=none] (26) at (-3.75, 0) {$g\i$};
	\end{pgfonlayer}
	\begin{pgfonlayer}{edgelayer}
		\draw [style=c1 arrow] (9) to (10.center);
		\draw [style=c1 arrow] (8) to (12.center);
		\draw [style=c1] (10.center) to (8);
		\draw [style=c1] (12.center) to (7);
		\draw [style=c2, bend left=60, looseness=1.25] (4) to (7);
		\draw [style=c2, bend left=60, looseness=1.25] (7) to (6);
		\draw [style=c2, bend left=60, looseness=1.25] (6) to (9);
		\draw [style=c2, bend left=60, looseness=1.25] (9) to (4);
		\draw [style=c0] (13.center) to (8);
		\draw [style=c0] (11.center) to (4);
		\draw [style=c0 arrow] (6) to (13.center);
		\draw [style=c0 arrow] (8) to (11.center);
	\end{pgfonlayer}
\end{tikzpicture}
\caption{In the proof of \autoref{lem:short_cycle_forall} we construct a walk by concatenating three orange paths joining four neighbours of $\unit$.}
\label{fig:cutvertexCycles}
\end{figure}
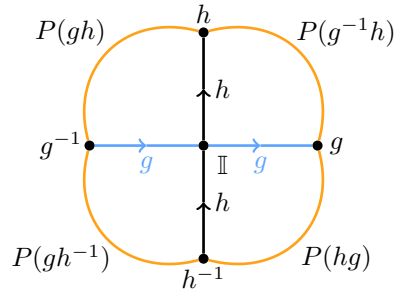

\begin{proof}
    We have to show that $B_r(v)-v$ is connected for every vertex $v\in G$.
    Since Cayley graphs are vertex-transitive, it is enough to show that $B_{r}(\unit)-\unit$ is connected.
Let $g$ be a neighbour of $\unit$. 
We shall show that for every element $h$ from the generating set $S$ with $g\not\equiv h$, there is a walk included in $B_r(\unit)-\unit$ containing the vertices $g$, $h$ and $g^{-1}$. 
By assumption, there is a morpheme $m$ of $\Gamma$ in $S$ of length at most~$r$ that is magic in $g$ and~$h$.
The morpheme $m$ determines a closed walk in $B_r(\unit)$ that starts and ends at~$\unit$, and that runs once around a cycle~$O$ because magic morphemes cannot have length~$\le 2$.
For each $ab\in \{gh,hg, g^{-1}h,hg^{-1}\}$ for which $m$ or $m\i$ contains $ab$ as a linear subword, we use that $G$ is vertex-transitive to obtain from~$O$ a cycle $O(ab)$ in $G$ such that the path $a\i\,\unit\; b$ is a subpath of~$O(ab)$, and we let $P(ab):=O(ab)-\unit$.
Since $O(ab)$ has length~$\le r$, the path $P(ab)$ links $a\i$ to~$b$ in $B_r(\unit)-\unit$.
By the magic property of~$m$, we find the paths $P(ab)$ for at least three of the pairs $ab\in\{gh,hg, g^{-1}h,hg^{-1}\}$, see \autoref{fig:cutvertexCycles}.
The concatenation of these three paths is the desired walk in $B_r(\unit)-\unit$ containing the vertices $g$, $h$ and~$g^{-1}$.
\end{proof}

\begin{proof}[Proof of \autoref{submainCutvertex}]
    Let $\Gamma$ be a group that is nilpotent of class~$\le n$, and let $G=\cay{\Gamma}{S}$ be a Cayley graph of~$\Gamma$.
    Suppose that $G$ has an $r$-local cutvertex for some $r\ge 2^{n+1}$.
    Our aim is to show that $G$ is a cycle longer than $r$ and $\Gamma$ is cyclic.

By \autoref{magic2}, for any two inequivalent elements $g$ and $h$ from $S$, there is a morpheme of $\Gamma$ in~$\{g\pmo,h\pmo\}$ that is magic in $g$ and $h$ and that has length at most~$r$. Since $G$ has an $r$-local cutvertex, the contrapositive of \autoref{lem:short_cycle_forall} implies that the generating set $S$ only contains equivalent elements. Thus $\Gamma$ is cyclic, and the Cayley graph $G=\cay{\Gamma}{S}$ is a cycle which must have length more than~$r$.
\end{proof}

Next, we will show that an assumption like nilpotency in \autoref{submainCutvertex} is necessary, see \autoref{UnboundedLocalCutvertex}. 
This essentially follows from the following result in the literature.

\begin{thm}[Dixon, Pyber, Seress and Shalev~{\cite[Theorem~3]{connectedgirth}}]\label{HighGirthRandom}\,\\
Let $\Gamma$ be a finite simple group and let $w$ be a nontrivial element of the free group on two elements $a,b$.
Let $\gamma_w(x,y)$ be the element of $\Gamma$ that is obtained from the word $w$ by first replacing $a$ and $b$ with $x$ and~$y$, respectively, and then evaluating the resulting word in the group~$\Gamma$.
Let $x,y\in\Gamma$ be chosen independently and uniformly at random.
Then a.a.s.\ $\langle x\pmo,y\pmo\rangle=\Gamma$ and $\gamma_w(x,y)\neq\unit$ as $|\Gamma|\to\infty$.
\end{thm}

\begin{lem}\label{UnboundedLocalCutvertex}
For every $r>0$ there exists $n>0$ such that every vertex of some Cayley graph of the alternating group $A_n$ is an $r$-local cutvertex.
\end{lem}
\begin{proof}
    Let $r>0$.
    Let $w_1,\ldots,w_k$ be all the nonempty words of length at most~$r$ in the two letters $a\pmo,b\pmo$.
    By \autoref{HighGirthRandom}, there is $n>0$ and a generating set $S=\{x\pmo,y\pmo\}$ of $\Gamma:=A_n$, so that $\gamma_{w_i}(x,y)\neq\unit$ for all $i\in [k]$, where $\gamma_{w_i}(x,y)$ is defined as in the statement of the theorem.
    Then the Cayley graph $G:=\cay{\Gamma}{S}$ has no cycles of length at most~$r$. 
    Hence every vertex of $G$ is an $r$-local cutvertex.
\end{proof}

\section{Local 2-separations}
\label{sec:2sep}

\subsection{Local components and separations}
We begin this section by defining $r$-local $2$-separators. Here we work with the definition of \cite{carmesin2023apply}; in that paper it is proved that it is equivalent to the original definition of \cite{carmesin}. We denote the neighbourhood of a set $X$ of vertices of a graph by $N(X)$.

\begin{dfn}
 Given a positive integer $r\ge 2$ and two vertices $v_0$ and $v_1$ in a graph~$G$, the \emph{connectivity graph} at $\{v_0,v_1\}$ of \emph{locality~$r$} has vertex set $N(\{v_0,v_1\})$ and two of its vertices $a$ and $b$ are adjacent if there is an $i\in \mathbb{F}_2$ such that $a$ and $b$ are in the same component of $B_r(v_i)-v_i-v_{i+1}$. 
 We denote the connectivity graph by $C_r(v_0,v_1,G)$ or just $C_r(v_0,v_1)$ if $G$ is clear from context.
 
 We say that $\{v_0,v_1\}$ is an \emph{$r$-local $2$-separator} of~$G$ if $C_r(v_0,v_1)$ is disconnected and the distance between $v_0$ and $v_1$ in $G$ is at most~$r/2$. 
 An \emph{$r$-local component} at $\{v_0,v_1\}$ is a connected component of $C_r(v_0,v_1)$. 
 An \emph{$r$-local $2$-separation} of~$G$ is a pair $\{A,B\}$ such that $A\cap B$ is an $r$-local $2$-separator $X$ and the sets $A\sm X$ and $B\sm X$ are nonempty and partition the vertex set of $C_r(v_0,v_1)$ so that each component of $C_r(v_0,v_1)$ has its vertex set contained in exactly one of $A\sm X$ and~$B\sm X$.
\end{dfn}

In this context, we refer to $A$ and~$B$ as the \emph{sides} of the $r$-local $2$-separation.
We say that the sides $A,B$ are \emph{opposite} of each other.
Two vertices $u,v$ are said to \emph{lie on opposite sides} of~$\{A,B\}$ if we have $u\in A\sm B$ and $v\in B\sm A$ or vice versa.
Two vertices $u,v$ are said to \emph{lie on the same side} of~$\{A,B\}$ if they both lie in $A\sm B$ or if they both lie in $B\sm A$.
We refer to $A\cap B$ as the \emph{separator} of~$\{A,B\}$.

\begin{lem}\label{locSepSeps}
Let $G$ be a graph, $r\ge 4$, and $X\se V(G)$ of size two.
Then $G$ has no edge between any two $r$-local components at~$X$.
\end{lem}
\begin{proof}
Let $x\in X$ and $u$ be a neighbour of $x$ in $N(X)$. Let $w$ be a neighbour of $u$ in $N(X)$. Then $u$ and $w$ are in the same component of $B_4(x)\sm X$ and thus in the same local component of $X$.
\end{proof}

\subsection{Traversals}

For a depiction of the following definition, see \autoref{fig:walk_traversal}.

\begin{dfn}[Cyclic subwalk]
Let $W=v_0 e_0 v_1 \ldots v_{n-1} e_{n-1} v_n$ be a closed walk, so~$v_0=v_n$.
A \emph{cyclic permutation} of~$W$ is either the walk~$W$ itself or a walk of the form $v_k e_k \ldots e_{n-1} v_n e_0 v_1\ldots v_k$ for some~$k$ with $0<k<n$.
A walk~$W'$ is a \emph{cyclic subwalk} of a closed walk~$W$ if $W'$ or the reverse of~$W'$ is a subwalk of a cyclic permutation of~$W$.
\end{dfn}

\begin{dfn}[Traversal]
    Let $G$ be a graph, $r\ge 2$ an integer, and $X$ an $r$-local $2$-separator of~$G$.
    A~walk in~$G$ is a \emph{traversal} of~$X$ if its ends lie in distinct $r$-local components at~$X$ and all its internal vertices lie in~$X$.
    Suppose now that $\{A,B\}$ is an $r$-local $2$-separation of~$G$ with separator~$X$.
    A~walk in~$G$ is a \emph{traversal} of~$\{A,B\}$ if it is a traversal of~$X$ and, additionally, its ends lie on opposite sides of~$\{A,B\}$.

    A traversal of either kind is \emph{weak} if it has precisely one internal vertex (so it has a total length of two), and \emph{strong} if it has precisely two internal vertices.

    Let $W=v_0 e_0 v_1\ldots v_{n-1} e_{n-1} v_n$ be a (closed) walk in~$G$, and let $T$ denote the set of those indices $i\in \mathbb{Z}_n$ such that the cyclic subwalk $v_ie_i \ldots v_{i+k}$ is a traversal of~$X$.
    We say that
    \begin{enumerate}
        \item $W$ \emph{traverses} $X$ if $T$ is nonempty;
        \item $W$ traverses $X$ \emph{evenly} or \emph{oddly} if $|T|$ is even or odd, respectively;\\
        (if $|T|=0$ we still say that $W$ traverses~$X$ evenly, by a slight abuse of notation)
        \item $W$ traverses $X$ \emph{weakly} or \emph{strongly} if at least one of the traversals that contribute to~$T$ contains a weak or strong traversal, respectively.
    \end{enumerate}
    If $W$ traverses~$X$ weakly, and the internal vertex of the traversal witnessing this is~$x$, then we say that $W$ \emph{weakly traverses $X$ at~$x$}.
    Similarly, we define when $W$ traverses $\{A,B\}$ evenly/oddly/weakly/strongly/weakly at~$x\in X$.
\end{dfn}

\begin{eg}
A~walk may traverse both weakly and strongly, see \autoref{fig:walk_traversal}.
However, paths and cycles cannot traverse local $2$-separators or local $2$-separations both weakly and strongly.
\end{eg}

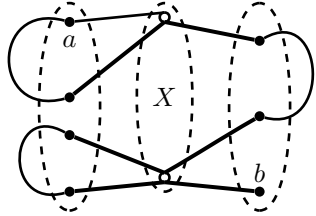
\begin{figure}[ht]
    \begin{tikzpicture}
	\begin{pgfonlayer}{nodelayer}
		\node [style=none] (0) at (0, 1) {};
		\node [style=none] (1) at (0, 1.25) {};
		\node [style=none] (2) at (0, -3) {};
		\node [style=none] (3) at (0, -3.25) {};
		\node [style=none] (4) at (0, 1.5) {};
		\node [style=none] (5) at (0, -3.5) {};
		\node [style=black dot] (6) at (-2.5, 1) {};
		\node [style=black dot] (7) at (-2.5, -1) {};
		\node [style=black dot] (8) at (-2.5, -2) {};
		\node [style=black dot] (9) at (-2.5, -3.5) {};
		\node [style=black dot] (10) at (2.5, -3.5) {};
		\node [style=black dot] (11) at (2.5, -1.5) {};
		\node [style=black dot] (12) at (2.5, 0.5) {};
		\node [style=none] (13) at (0, -1) {};
		\node [style=none] (14) at (0, -1) {};
		\node [style=none] (15) at (0, -1) {$X$};
		\node [style=none] (16) at (-2.5, 0.5) {$a$};
		\node [style=none] (17) at (2.5, -3) {$b$};
		\node [style=none] (18) at (-2.5, 1.5) {};
		\node [style=none] (19) at (-2.5, -4) {};
		\node [style=none] (20) at (2.5, 1.5) {};
		\node [style=none] (21) at (2.5, -4) {};
	\end{pgfonlayer}
	\begin{pgfonlayer}{edgelayer}
		\draw [style=c0, bend left=90, looseness=1.75] (1.center) to (0.center);
		\draw [style=c0, bend right=90, looseness=1.75] (2.center) to (3.center);
		\draw [style=c0, bend right=90, looseness=1.75] (3.center) to (2.center);
		\draw [style=c0, bend right=270, looseness=1.75] (0.center) to (1.center);
		\draw [style=d1] (5.center)
			 to [bend right=90, looseness=0.50] (4.center)
			 to [bend left=270, looseness=0.50] cycle;
		\draw [style=c0] (6) to (1.center);
		\draw [style=c0 bold] (7) to (0.center);
		\draw [style=c0 bold] (0.center) to (12);
		\draw [style=c0 bold, in=30, out=-150] (11) to (2.center);
		\draw [style=c0 bold] (2.center) to (8);
		\draw [style=c0 bold] (9) to (3.center);
		\draw [style=c0 bold] (3.center) to (10);
		\draw [style=c0, in=165, out=-165, looseness=2.50] (7) to (6);
		\draw [style=c0, in=-15, out=30, looseness=2.25] (12) to (11);
		\draw [style=c0, in=-165, out=150, looseness=2.75] (8) to (9);
		\draw [style=d1] (19.center)
			 to [bend right=90, looseness=0.50] (18.center)
			 to [bend left=270, looseness=0.50] cycle;
		\draw [style=d1, bend right=90, looseness=0.50] (21.center) to (20.center);
		\draw [style=d1, bend left=270, looseness=0.50] (20.center) to (21.center);
	\end{pgfonlayer}
\end{tikzpicture}
    \caption{An $a$--$b$ walk that traverses a local $2$-separator~$X$ oddly. The traversals are bold.}
    \label{fig:walk_traversal}
\end{figure}

\begin{lem}\label{traversal}
Let $G$ be a graph, $r\ge 2$ an integer, and $\{A,B\}$ an $r$-local $2$-separation of~$G$ with separator~$X$.
Let $a,b\in N(X)$ and let $W$ be an $a$--$b$ walk in~$G$ that is included in $B_r(x)$ for some~$x\in X$.
Then $a$ and $b$ lie on opposite sides of~$\{A,B\}$ if and only if $W$ traverses $\{A,B\}$ oddly.
\end{lem}
\begin{proof}
Let $v_0e_0v_1\ldots v_{n-1}e_{n-1}v_n:=W$, so $v_0=a$ and~$v_n=b$.
We write the walk $W$ as a concatenation of walks $W_0,\ldots,W_k$ with even~$k$ such that
\begin{itemize}
    \item the walks $W_i$ with odd index~$i$ have their ends in $N(X)$ and meet $X$ with all their internal vertices,
    \item while the walks $W_i$ with even index~$i$ avoid $X$ and link two vertices in~$N(X)$; these walks may have length~0.
\end{itemize}
For every odd~$i$, the ends of $W_i$ lie on opposite sides of~$\{A,B\}$ if and only if $W_i$ is a traversal of~$\{A,B\}$.
For every even~$i$, the walks $W_i$ witness that the ends of~$W_i$ lie on the same side of~$\{A,B\}$.
So the ends $a$ and $b$ of~$W$ lie on opposite sides of~$W$ if and only if the number of walks $W_i$ with odd~$i$ that are traversals of~$\{A,B\}$ is odd.
Since every traversal of~$\{A,B\}$ contained in~$W$ occurs as a~$W_i$, we are done.
\end{proof}

\begin{cor}\label{CycleWeakTraversal}
    Let $G$ be a graph, $r\ge 2$ an integer, and $\{A,B\}$ an $r$-local $2$-separation of~$G$ with separator~$X$.
    If a cycle $O$ in~$G$ of length at most~$r$ weakly traverses $\{A,B\}$ at some vertex of~$X$, then $O$ weakly traverses $\{A,B\}$ at both vertices of~$X$.\qed
\end{cor}

\begin{dfn}[Traversal of words and generators]
    Let $G:=\cay{\Gamma}{S}$ be a Cayley graph and~$r\ge 2$.
    Let $X$ be an $r$-local $2$-separator of~$G$, and let $\{A,B\}$ be an $r$-local $2$-separation of~$G$.

    We say that a word $w$ in $S$ \emph{traverses} $X$ if there exists a walk in~$G$ labelled by~$w$ that traverses~$X$.
    We say that $w$ traverses $X$ \emph{weakly} or \emph{strongly} if some walk labelled by $w$ traverses $X$ weakly or strongly, respectively.
    We say that an element $g\in S$ \emph{traverses} $X$ if the word $g^2$ weakly traverses~$X$.
    We say that $g$ traverses $X$ \emph{at} some $x\in X$ if some traversal of~$X$ labelled by $g^2$ uses the vertices $xg\i,x,xg$ in this order.
    Similarly, we define all of the above with~$\{A,B\}$ in place of~$X$.
\end{dfn}
\begin{rem}
   If $X=\X$, then no traversing element of~$S$ is equal to $h$ or~$h\i$.
   If $g$ traverses $X$, then it does so at some $x\in X$.
\end{rem}

\begin{lem}[Strong Traversal Lemma]\label{traverser}
Let $G:=\cay{\Gamma}{S}$ be a Cayley graph, $r\ge 2$, and $X$ an $r$-local $2$-separator of~$G$.
Then no cycle in~$G$ of length at most~$r$ strongly traverses~$X$.
In particular, if a word $w$ strongly traverses $X$, then $w$ cannot be a cyclic subword of a morpheme of~$\Gamma$ in~$S$ of length at most~$r$.
\end{lem}
\begin{proof}
Suppose for a contradiction that there is a cycle $O$ in $G$ of length at most~$r$ that strongly traverses~$X$.
Then $O$ is contained in the ball $B_r(x)$, where $x\in X$ is arbitrary, and the ends of the path $O\sm X$ lie in distinct $r$-local components at~$X$ because $O$ strongly traverses~$X$.
But the path $O\sm X$ is contained in $B_r(x)\sm X$, so its ends lie in the same $r$-local component at~$X$, a contradiction.
 \end{proof}

\begin{figure}[ht]
    \begin{tikzpicture}
	\begin{pgfonlayer}{nodelayer}
		\node [style=none] (82) at (-1.5, 3) {$A$};
		\node [style=none] (83) at (1.5, 3) {$B$};
		\node [style=black dot] (95) at (0, 0) {};
		\node [style=black dot] (97) at (0, 2) {};
		\node [style=black dot] (99) at (-2, 0) {};
		\node [style=none] (100) at (0.5, 0.5) {};
		\node [style=none] (101) at (-0.5, 0.5) {};
		\node [style=none] (102) at (-0.5, -0.5) {};
		\node [style=none] (105) at (-0.5, 2.5) {};
		\node [style=none] (107) at (-0.5, -2.5) {};
		\node [style=none] (109) at (2.5, 0.5) {};
		\node [style=none] (110) at (-2.5, 0.5) {};
		\node [style=none] (112) at (-2.5, 0) {$a$};
		\node [style=none] (113) at (0, 2.5) {$a'$};
		\node [style=none] (118) at (1.5, -1.5) {$\varnothing$};
		\node [style=none] (119) at (1.5, 0) {$\varnothing$};
		\node [style=none] (120) at (0, -1.5) {$\varnothing$};
		\node [style=none] (124) at (-3, -1.5) {$B'$};
		\node [style=none] (125) at (-3, 1.5) {$A'$};
		\node [style=none] (193) at (0.5, -0.5) {$v$};
		\node [style=none] (194) at (0.5, -2.5) {};
		\node [style=none] (195) at (0.5, -0.75) {};
		\node [style=none] (196) at (0.5, -0.25) {};
		\node [style=none] (197) at (0.5, 2.5) {};
		\node [style=none] (198) at (2.5, -0.5) {};
		\node [style=none] (199) at (0.75, -0.5) {};
		\node [style=none] (200) at (-2.5, -0.5) {};
		\node [style=none] (201) at (0.25, -0.5) {};
	\end{pgfonlayer}
	\begin{pgfonlayer}{edgelayer}
		\draw [style=d1] (107.center) to (102.center);
		\draw [style=d1] (194.center) to (195.center);
		\draw [style=d1] (109.center) to (100.center);
		\draw [style=d1] (198.center) to (199.center);
		\draw [style=d1] (200.center) to (102.center);
		\draw [style=d1] (110.center) to (101.center);
		\draw [style=d1] (105.center) to (101.center);
		\draw [style=d1] (197.center) to (196.center);
		\draw [style=d1] (101.center) to (100.center);
		\draw [style=d1] (101.center) to (102.center);
		\draw [style=d1] (102.center) to (201.center);
		\draw [style=c0] (99) to (95);
		\draw [style=c0] (95) to (97);
	\end{pgfonlayer}
\end{tikzpicture}
    \caption{Two crossing $r$-local separations as in \autoref{empty_corner}, where $v$ is not an $r$-local cutvertex}
    \label{fig:empty_corner}
\end{figure}
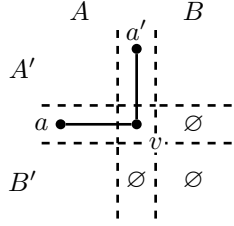

The following lemma is supported by \autoref{fig:empty_corner}.

\begin{lem}\label{empty_corner}
Let $G$ be a graph and $r\ge 4$.
Let $\{A,B\}$ be an $r$-local $2$-separation of~$G$ with separator~$\{a',v\}$ such that $a'v$ is an edge in~$G$.
Let $\{A',B'\}$ be an $r$-local $2$-separation of~$G$ with separator~$\{a,v\}$ such that $av$ is an edge in~$G$ and~$a\neq a'$.
Suppose that $a\in A$ and $a'\in A'$, and that $v$ is not an $r$-local cutvertex of~$G$.
Then $B\cap B'=\{v\}$.
\end{lem}

\begin{proof}
Suppose for a contradiction that $B\cap B'$ contains another vertex~$x$ besides~$v$.
Then $x$ has neighbours in both separators $\{a',v\}$ and $\{a,v\}$.
Since $x$ lies in $B\sm A$ and $a$ lies $A\sm B$, there is no edge between $x$ and $a$ by \autoref{locSepSeps}.
Similarly, there is no edge between $x$ and~$a'$.
Hence $v$ must be a neighbour of~$x$.
Since $v$ is not an $r$-local cutvertex of~$G$, 
there is an $x$--$\{a,a'\}$ path $P$ in the ball~$B_r(v)$; say its endvertex is~$a$. 
Then $a$ and $x$ are in the same component of $B_r(v)-v-a'$, contradicting that $a$ and~$x$ lie on opposite sides of~$\{A,B\}$.
\end{proof}

\subsection{The Edge Insertion Lemma}

In this section, we will prove the \nameref{long_edge} (\ref{long_edge}), a key lemma, which will allow us to assume in the proof of \autoref{3-con} that the given local $2$-separator spans an edge.

\begin{lem}[Edge Insertion Lemma]\label{long_edge}
    Let $4\le r^-\le r_0\le r^+$ be integers such that $\tfrac{1}{2}r_0\cdot r^-\le r^+$.
    Let $G:=\cay{\Gamma}{S}$ be a Cayley graph such that $G$ has no $r^-$-local cutvertex and $G$ is not a cycle of length~$\le r^+$, but such that $G$ contains a cycle of length at most~$r^+$.
    We further assume that $G$ has an $r^+$-local $2$-separator.
    Then there is some $h\in\Gamma$ such that the Cayley graph $G':=\cay{\Gamma}{S\cup\{h^{\pm 1}\}}$ has no $r_0$-local cutvertex and such that $\X$ is an $r_0$-local $2$-separator of~$G'$.
    \end{lem}

\begin{dfn}
    Let $G$ be a graph and $r\ge 2$ an integer.
    We say that a pair $\{u,v\}$ of two distinct vertices of~$G$ \emph{crosses} an $r$-local $2$-separation $\{A,B\}$ of~$G$ with separator~$X$ if the two vertex sets $\{u,v\}$ and $X$ are disjoint and if there is a $u$--$v$ walk $W$ in~$G$ of length at most~$r/2$ such that $W$ traverses $\{A,B\}$ oddly, see \autoref{fig:crossing}.
    We say that $\{u,v\}$ \emph{crosses} an $r$-local $2$-separator $X$ of~$G$ if $\{u,v\}$ crosses some $r$-local $2$-separation of~$G$ with separator equal to~$X$.
\end{dfn}

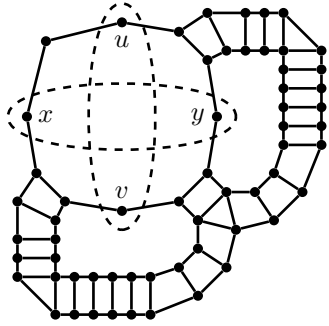
\begin{figure}[ht]
    \centering
\begin{tikzpicture}
	\begin{pgfonlayer}{nodelayer}
		\node [style=none] (4) at (0, 2) {};
		\node [style=none] (5) at (0, -4) {};
		\node [style=none] (22) at (-3, -1) {};
		\node [style=none] (23) at (3, -1) {};
		\node [style=black dot] (24) at (0, -3.5) {};
		\node [style=black dot] (25) at (2.5, -1) {};
		\node [style=black dot] (26) at (0, 1.5) {};
		\node [style=black dot] (27) at (-2.5, -1) {};
		\node [style=black dot] (28) at (1.5, 1.25) {};
		\node [style=black dot] (29) at (2.25, 0.5) {};
		\node [style=black dot] (30) at (-2, 1) {};
		\node [style=black dot] (32) at (1.5, -3.25) {};
		\node [style=black dot] (33) at (2.25, -2.5) {};
		\node [style=black dot] (36) at (2.75, -3) {};
		\node [style=black dot] (44) at (2, -3.75) {};
		\node [style=black dot] (60) at (3, -4) {};
		\node [style=black dot] (61) at (2, -3.75) {};
		\node [style=none] (74) at (-2, -1) {$x$};
		\node [style=none] (75) at (2, -1) {$y$};
		\node [style=none] (76) at (0, 1) {$u$};
		\node [style=none] (77) at (0, -3) {$v$};
		\node [style=black dot] (78) at (-2.25, -2.5) {};
		\node [style=black dot] (79) at (-1.5, -3.25) {};
		\node [style=black dot] (80) at (4, -2.5) {};
		\node [style=black dot] (81) at (4.75, -3) {};
		\node [style=black dot] (82) at (4.25, -1.75) {};
		\node [style=black dot] (83) at (5.25, -1.75) {};
		\node [style=black dot] (84) at (4.25, -1.25) {};
		\node [style=black dot] (85) at (5.25, -1.25) {};
		\node [style=black dot] (86) at (4.25, -0.75) {};
		\node [style=black dot] (87) at (5.25, -0.75) {};
		\node [style=black dot] (90) at (3.25, 0.75) {};
		\node [style=black dot] (91) at (3.25, 1.75) {};
		\node [style=black dot] (92) at (3.75, 0.75) {};
		\node [style=black dot] (93) at (3.75, 1.75) {};
		\node [style=black dot] (94) at (4.25, -0.25) {};
		\node [style=black dot] (95) at (5.25, -0.25) {};
		\node [style=black dot] (96) at (4.25, 0.75) {};
		\node [style=black dot] (97) at (3.5, -3) {};
		\node [style=black dot] (98) at (4, -3.75) {};
		\node [style=black dot] (99) at (4.25, 1.75) {};
		\node [style=black dot] (100) at (5.25, 0.25) {};
		\node [style=black dot] (101) at (4.25, 0.25) {};
		\node [style=black dot] (102) at (5.25, 0.75) {};
		\node [style=black dot] (103) at (2.75, 0.75) {};
		\node [style=black dot] (104) at (2.25, 1.75) {};
		\node [style=black dot] (105) at (2, -4.5) {};
		\node [style=black dot] (106) at (2.75, -5) {};
		\node [style=black dot] (107) at (2, -5.75) {};
		\node [style=black dot] (108) at (1.5, -5) {};
		\node [style=black dot] (109) at (0.75, -6.25) {};
		\node [style=black dot] (110) at (0.75, -5.25) {};
		\node [style=black dot] (111) at (0.25, -5.25) {};
		\node [style=black dot] (112) at (0.25, -6.25) {};
		\node [style=black dot] (113) at (-0.25, -6.25) {};
		\node [style=black dot] (114) at (-0.25, -5.25) {};
		\node [style=black dot] (115) at (-0.75, -5.25) {};
		\node [style=black dot] (116) at (-0.75, -6.25) {};
		\node [style=black dot] (117) at (-1.25, -6.25) {};
		\node [style=black dot] (118) at (-1.25, -5.25) {};
		\node [style=black dot] (119) at (-1.75, -5.25) {};
		\node [style=black dot] (120) at (-1.75, -6.25) {};
		\node [style=black dot] (121) at (-2.75, -5.25) {};
		\node [style=black dot] (122) at (-2.75, -4.75) {};
		\node [style=black dot] (123) at (-1.75, -4.75) {};
		\node [style=black dot] (124) at (-1.75, -4.25) {};
		\node [style=black dot] (125) at (-2.75, -4.25) {};
		\node [style=black dot] (126) at (-1.75, -3.75) {};
		\node [style=black dot] (127) at (-2.75, -3.25) {};
	\end{pgfonlayer}
	\begin{pgfonlayer}{edgelayer}
		\draw [style=d1] (5.center)
			 to [bend right=90, looseness=0.50] (4.center)
			 to [bend left=270, looseness=0.50] cycle;
		\draw [style=d1] (23.center)
			 to [bend right=90, looseness=0.50] (22.center)
			 to [bend left=270, looseness=0.50] cycle;
		\draw [style=c0] (30) to (26);
		\draw [style=c0] (26) to (28);
		\draw [style=c0] (28) to (29);
		\draw [style=c0] (29) to (25);
		\draw [style=c0] (25) to (33);
		\draw [style=c0] (33) to (32);
		\draw [style=c0] (32) to (24);
		\draw [style=c0] (27) to (30);
		\draw [style=c0] (44) to (36);
		\draw [style=c0] (36) to (33);
		\draw [style=c0] (32) to (44);
		\draw [style=c0] (60) to (61);
		\draw [style=c0] (36) to (60);
		\draw [style=c0] (27) to (78);
		\draw [style=c0] (78) to (79);
		\draw [style=c0] (79) to (24);
		\draw [style=c0] (100) to (83);
		\draw [style=c0] (83) to (81);
		\draw [style=c0] (81) to (98);
		\draw [style=c0] (98) to (60);
		\draw [style=c0] (36) to (97);
		\draw [style=c0] (97) to (80);
		\draw [style=c0] (80) to (82);
		\draw [style=c0] (82) to (96);
		\draw [style=c0] (91) to (90);
		\draw [style=c0] (93) to (92);
		\draw [style=c0] (99) to (96);
		\draw [style=c0] (102) to (100);
		\draw [style=c0] (102) to (99);
		\draw [style=c0] (96) to (102);
		\draw [style=c0] (101) to (100);
		\draw [style=c0] (94) to (95);
		\draw [style=c0] (86) to (87);
		\draw [style=c0] (84) to (85);
		\draw [style=c0] (82) to (83);
		\draw [style=c0] (80) to (81);
		\draw [style=c0] (97) to (98);
		\draw [style=c0] (28) to (104);
		\draw [style=c0] (104) to (99);
		\draw [style=c0] (96) to (103);
		\draw [style=c0] (103) to (29);
		\draw [style=c0] (103) to (104);
		\draw [style=c0] (125) to (121);
		\draw [style=c0] (121) to (120);
		\draw [style=c0] (120) to (109);
		\draw [style=c0] (109) to (107);
		\draw [style=c0] (107) to (106);
		\draw [style=c0] (106) to (60);
		\draw [style=c0] (61) to (105);
		\draw [style=c0] (105) to (108);
		\draw [style=c0] (119) to (124);
		\draw [style=c0] (125) to (124);
		\draw [style=c0] (122) to (123);
		\draw [style=c0] (121) to (119);
		\draw [style=c0] (120) to (119);
		\draw [style=c0] (118) to (117);
		\draw [style=c0] (116) to (115);
		\draw [style=c0] (114) to (113);
		\draw [style=c0] (112) to (111);
		\draw [style=c0] (110) to (109);
		\draw [style=c0] (108) to (107);
		\draw [style=c0] (106) to (105);
		\draw [style=c0] (119) to (110);
		\draw [style=c0] (110) to (108);
		\draw [style=c0] (125) to (127);
		\draw [style=c0] (127) to (78);
		\draw [style=c0] (79) to (126);
		\draw [style=c0] (126) to (124);
		\draw [style=c0] (126) to (127);
	\end{pgfonlayer}
\end{tikzpicture}
\caption{A pair $\{u,v\}$ that crosses a local separation with separator $\{x,y\}$.}
\label{fig:crossing}
\end{figure}

\begin{lem}
{\cite[Corollary 6.6]{carmesin}}
\label{crossingIsSymmetric}
The relation defined on the $r$-local $2$-separators of~$G$ by `crossing' is symmetric.
\end{lem}

By \autoref{crossingIsSymmetric}, we can and will use `crossing' as a symmetric relation on $r$-local $2$-separators.

\begin{thm}
\cite[Theorem~1.3]{carmesin}
\label{AngryTheorem2}
Let $G$ be a connected graph and $r\ge 2$ such that $G$ has no $r$-local cutvertex and contains a cycle of length at most~$r$.
If $G$ has an $r$-local $2$-separator, and if every $r$-local $2$-separator of~$G$ is crossed, then $G$ is a cycle of length at most~$r$.
\end{thm}

\begin{lem}\label{LocalCompsAtXseeX}
Let $X$ be an $r$-local $2$-separator of a connected graph~$G$ for some $r\ge 2$.
If an $r$-local component at~$X$ avoids $N(y)$ for some $y\in X$, then the other vertex $x\in X-y$ is an $r$-local cutvertex of~$G$.
\end{lem}
\begin{proof}
This follows as a consequence of the Local 2-Connectivity Lemma \cite[Lemma~3.10]{carmesin}.\footnote{The statement of \cite[Lemma~3.10]{carmesin} requires the presence of a cycle in $G$ of length~$\le r$ (hidden in the term `$r$-locally 2-connected'), but the proof in \cite{carmesin} never uses this, so we can indeed apply the lemma here.}
\end{proof}

\begin{lem}\label{localCutvertexFromLocal2sep}
Let $\{A,B\}$ be an $r$-local $2$-separation of a connected graph~$G$ for some $r\ge 2$.
Let $X=\{x,y\}$ be the separator of $\{A,B\}$ and denote the distance between $x$ and $y$ in~$G$ by~$d$.
We assume that $G$ has no $r$-local cutvertex.
Then $x$ and $y$ are $r'$-local cutvertices of~$G$ for all integers $r'$ satisfying $1\le r'/2<d$.
\end{lem}

\begin{proof}
By symmetry it suffices to show that $B_{r'}(x)-x$ is disconnected for $r'$ satisfying $1\le r'/2<d$.
By \autoref{LocalCompsAtXseeX} the vertex $x$ has two neighbours, call them $u$ and $v$, in $N(X)$ in distinct $r$-local components at~$X$. 
So $u$ and $v$ are not connected by a path in the connectivity graph $C_r(x,y)$; that is, $u$ and $v$ lie in distinct components of~$B_r(x)\sm X$. 
As $X$ is an $r$-local $2$-separator, we have $d\le r/2$, so $r'<2d\le r$ gives $B_{r'}(x)\subseteq B_r(x)$.
Since $r'<2d$, the vertex~$y$ is not contained in $B_{r'}(x)$.
Hence $B_{r'}(x)-x$ is included in $B_r(x)\sm X$.
Thus $u$ and $v$ lie in distinct components of $B_{r'}(x)-x$.
\end{proof}

A \emph{separation} of $G$ is a pair $\{A,B\}$ such that $A\cup B=V(G)$, both $A\sm B$ and $B\sm A$ are nonempty, and~$G$ contains no edge between $A\sm B$ and~$B\sm A$.
We refer to $A\cap B$ as the \emph{separator} of~$\{A,B\}$ and say that $\{A,B\}$ is a \emph{$k$-separation} for $k:=|A\cap B|$.

Let $\{A,B\}$ be an $r$-local $2$-separation of~$G$ with separator~$X$. Assume that no vertex in $X$ is an $r$-local cutvertex.
For each $x\in X$ let $A_x$ be obtained from $X$ by adding all the vertices in components of $B_r(x)\sm X$ intersecting~$A\sm B$.
Analogously, let $B_x$ be obtained from $X$ by adding all the vertices in components of $B_r(x)\sm X$ intersecting~$B\sm A$.
Then $A_x\sm B_x$ and $B_x\sm A_x$ are nonempty by \autoref{LocalCompsAtXseeX}.
Hence $\{A_x,B_x\}$ is a $2$-separation of~$B_r(x)$ with separator~$X$, which we call the $2$-separation of $B_r(x)$ \emph{induced} by~$\{A,B\}$.

Let $X=\{u,v\}$ be a set of two vertices of~$G$, let $r\ge 2$, and for each $x\in X$ let $\{A_x,B_x\}$ be a $2$-separation of $B_r(x)$ with separator~$X$.
We say that $\{A_u,B_u\}$ and $\{A_v,B_v\}$ are \emph{compatible} if $(A_u\cup A_v)\cap N(X)$ is disjoint from $(B_u\cup B_v)\cap N(X)$.

\begin{eg}\label{LocalSepnCompatibleI}
If $\{A,B\}$ is an $r$-local $2$-separation of~$G$ for $r\ge 2$ with separator~$X$ such that $X$ contains no $r$-local cutvertices, then its induced $2$-separations of the balls $B_r(x)$ for $x\in X$ are compatible.\qed
\end{eg}

\begin{lem}\label{LocalSepnCompatibleII}
Let $G$ be a graph, $r\ge 2$, and $X$ a set of two vertices $u,v$ of~$G$.
If for each $x\in X$ there is a $2$-separation $\{A_x,B_x\}$ of $B_r(x)$ with separator~$X$ so that these $2$-separations are compatible, then $\{A,B\}$ is an $r$-local $2$-separation of $G$ with separator~$X$ for
    \[
        A:=(A_u\cup A_v)\cap (N(X)\cup X)\quad \text{and}\quad B:=(B_u\cup B_v)\cap (N(X)\cup X).
    \]
\end{lem}
\begin{proof}
We have to show that $\{A,B\}$ is an $r$-local $2$-separation.
Since $X$ is included in $B_r(x)$ for $x\in X$ as the separator of $\{A_x,B_x\}$, the distance between $u$ and $v$ is at most~$r/2$.
Since $\{A_u,B_u\}$ and $\{A_v,B_v\}$ are compatible, we have $A\cap B=X$.
And since $r\ge 2$, we have $(A\sm B)\cup (B\sm A)=N(X)$.
Furthermore, $A\sm B$ and $B\sm A$ are nonempty since they include the nonempty sets $(A_u\sm B_u)\cap N(X)$ and $(B_u\sm A_u)\cap N(X)$, respectively.

Hence it remains to show that no edge of the connectivity graph $C_r(X)$ joins a vertex in $A\sm B$ to a vertex in $B\sm A$.
So let $e$ be an edge of $C_r(X)$.
Then the ends of $e$ lie in the same component $C$ of $B_r(u)\sm X$, say.
By assumption, $\{A_u,B_u\}$ is a $2$-separation of $B_r(u)$, so the vertex set of $C$ is included in $A_u\sm B_u$ or in $B_u\sm A_u$, say in $A_u\sm B_u$.
Then the ends of $e$ lie in $A_u\cap N(X)\se A\sm B$.
\end{proof}

To gain insights on the global structure of a group from a local 2-separation in one of its Cayley graphs, we sometimes found it helpful to slightly increase the generating set of the Cayley graph.
\autoref{CompatibleSeparationsAddingEdges} below in combination with \autoref{LocalSepnCompatibleII} above shows that slightly increasing the generating set essentially preserves the local 2-separation, provided that the added generators respect the local 2-separation.

Let $\{A,B\}$ be an $r$-local 2-separation of~$G$.
A set $F\se [V(G)]^2$ of edges \emph{respects} $\{A,B\}$ if for each edge $e\in F$ there exists a walk $W_e$ in $G$ of length~$\le d$ that links the ends of~$e$ and traverses~$\{A,B\}$ evenly.

\begin{lem}\label{CompatibleSeparationsAddingEdges}
Let $r,r'\ge 2$ and $d\ge 1$ with $r'\cdot d\le r$.
Let $\{A,B\}$ be an $r$-local $2$-separation of~$G$ whose separator~$X$ contains no $r$-local cutvertices.
For $x\in X$ let $\{A_x,B_x\}$ be the $2$-separation of $B_r(x)$ with separator~$X$ induced by~$\{A,B\}$.
Let $G'$ be obtained from $G$ by adding a set $F$ of new edges that respects~$\{A,B\}$.
Assume that the two vertices in $X$ have distance $\le r'/2$ in~$G'$.
For $x\in X$ let $A'_x$ and $B'_x$ be obtained from $A_x$ and $B_x$, respectively, by taking the intersection with $V(B_{r'}(x,G'))$.
Then $\{A'_x,B'_x\}$ is a $2$-separation of $B_{r'}(x,G')$ with separator~$X$ for both $x\in X$, and these $2$-separations are compatible.
\end{lem}
\begin{proof}
Let $x\in X$ be arbitrary.
Since the two vertices in $X$ have distance at most $r'/2$ in~$G'$, both vertices lie in~$B_{r'}(x,G')$.
Hence $A'_x\cap B'_x=X$.
The vertex set of $B_{r'}(x,G')$ is included in $B_r(x,G)$ since $r'\cdot d\le r$.
Thus $A'_x\cup B'_x$ is equal to the vertex set of $B_{r'}(x,G')$.
As $X$ contains no $r$-local cutvertices by assumption, \autoref{LocalCompsAtXseeX} tells us that both $A_x\sm B_x$ and $B_x\sm A_x$ contain neighbours of~$x$, and since $r'\ge 2$ these neighbours witness that $A'_x\sm B'_x$ and $B'_x\sm A'_x$ are nonempty.

To show that $\{A'_x,B'_x\}$ is a $2$-separation of $B_{r'}(x,G')$ with separator~$X$, it remains to show that there is no edge in $B_{r'}(x,G')$ between $A'_x\sm B'_x$ and $B'_x\sm A'_x$.
Suppose for a contradiction that there are $a\in A_x'\sm B_x'$ and $b\in B_x'\sm A_x'$ such that $ab$ is an edge in~$B_{r'}(x,G')$.
If $ab$ is present in~$G$, then $ab$ is an edge in $B_{r}(x,G)$ (since $r'\cdot d\le r$) with $a\in A_x\sm B_x$ and $b\in B_x\sm A_x$, contradicting that $\{A_x,B_x\}$ is a $2$-separation of $B_{r}(x,G)$.
So we may assume that $ab$ is in~$F$. 
Since $ab$ is present in the ball $B_{r'}(x,G')$, there is a closed walk $W'$ in~$G'$ of length~$\le r'$ that uses the edge~$ab$ and the vertex~$x$.
Let $W$ be the closed walk in $G$ that is obtained from~$W'$ by replacing each edge $e\in F$ used by $W'$ with the walk~$W_e$.
Then $W$ contains $W_{ab}$ as a subwalk, uses the vertex~$x$ and has length~$\le r$ (since $r'\cdot d\le r$).
Consider a minimal subwalk $W^*$ of~$W$ that includes $W_{ab}$ and has both ends $a^*,b^*$ in~$N_G(X)$.
Then $W^*$ traverses~$\{A,B\}$ as often as $W_{ab}$ does, which is an even number by assumption.
So $a^*,b^*\in A\sm B$ say, by \autoref{traversal}.
Hence $a^*,b^*\in A_x\sm B_x$.
Since the subwalks $a^* W^* a$ and $b W^* b^*$ of~$W^*$ are walks in $B_{r}(x,G)$ that avoid the separator~$X$ of~$\{A_x,B_x\}$, also the two subwalks are included in $A_x\sm B_x$.
In particular, $b\in A_x\sm B_x$.
This contradicts the fact that $b\in B_x'\sm A_x'\se B_x\sm A_x$.

Since the induced $2$-separations $\{A_x,B_x\}$ are compatible by \autoref{LocalSepnCompatibleI}, the $2$-separations $\{A'_x,B'_x\}$ are compatible as well.
\end{proof}

An $r$-local $2$-separator of~$G$ is \emph{totally nested} if it is not crossed by another $r$-local $2$-separator of~$G$.

\begin{proof}[Proof of the \nameref{long_edge} (\ref{long_edge})]
Let $2\le r^-\le r_0\le r^+$ be integers such that $\tfrac{1}{2}r_0\cdot r^-\le r^+$.
Let $G:=\cay{\Gamma}{S}$ be a Cayley graph such that $G$ has no $r^-$-local cutvertex and $G$ is not a cycle of length~$\le r^+$, but such that $G$ contains a cycle of length~$\le r^+$.
Suppose that $G$ has an $r^+$-local $2$-separator.
We have to find an $h\in\Gamma$ such that the Cayley graph $\cay{\Gamma}{S\cup\{h^{\pm 1}\}}$ has no $r_0$-local cutvertex and such that $\X$ is an $r_0$-local $2$-separator of this Cayley graph.

Since $r^+\ge r^-\ge 2$, the assumption that $G$ has no $r^-$-local cutvertex implies that $G$ has no $r^+$-local cutvertex either.
Hence, by \autoref{AngryTheorem2}, $G$~has a totally-nested $r^+$-local $2$-separator.
Using vertex-transitivity, we choose this $r^+$-local $2$-separator so that it contains the vertex~$\unit$.
Then the $r^+$-local $2$-separator has the form $\X$ for some $h\in\Gamma$, and we abbreviate $X:=\X$.

Let us denote the distance between the vertices $\unit$ and~$h$ in~$G$ by~$d$.
Since $X$ is an $r^+$-local $2$-separator of~$G$, we have $d\le r^+/2$ by definition, but we can get a better bound as follows:
Since $X$ is an $r^+$-local $2$-separator of~$G$ but $G$ has no $r^-$-local cutvertex, \autoref{localCutvertexFromLocal2sep} implies~$d\le r^-/2$.
Consider the Cayley graph $G':=\cay{\Gamma}{S\cup\{h^{\pm 1}\}}$.

\begin{sublem}\label{local connectivity}
    $G'$ has no $r$-local cutvertex for $r\ge r^-$.
\end{sublem}

\begin{cproof}
Let $x$ be any vertex of~$G'$.
Since $r^-\ge 2$, the ball $B_{r^-}(x,G')$ contains all neighbours of~$x$ in~$G'$.
Hence to show that $x$ is not an $r$-local cutvertex of~$G'$ for~$r\ge r^-$, it suffices to show that $x$ is not an $r^-$-local cutvertex of~$G'$.

Since $G\se G'$, we have the inclusion $B_{r^-}(x,G)\se B_{r^-}(x,G')$.
As $B_{r^-}(x,G)-x$ is connected by assumption,
and since $B_{r^-}(x,G')$ is connected by the definition of a ball,
to show that $B_{r^-}(x,G')-x$ is connected it suffices to show that $B_{r^-}(x,G)$ contains the entire neighbourhood~$N_{G'}(x)$ of~$x$ in~$G'$.
And indeed, $N_{G'}(x)=N_G(x)\cup\{xh,xh\i\}$, and both vertices $xh,xh\i$ have distance~$d\le r^-/2$ from~$x$ in~$G$, while $N_G(x)$ is included in $B_{r^-}(x,G)$ since~$r^-\ge 2$.
\end{cproof}\medskip

By \autoref{local connectivity}, $G'$ has no $r_0$-local cutvertex.
It remains to show that $X$ is an $r_0$-local $2$-separator of~$G'$.
Since $G$ has no $r^-$-local cutvertex by assumption, it also has no $r^+$-local cutvertex as $r^+\ge r^-$.
For each $x\in X$, let $\{A_x,B_x\}$ denote the $2$-separation of $B_{r^+}(x,G)$ induced by~$\{A,B\}$, and let $A'_x,B'_x$ be obtained from $A_x,B_x$ by taking the intersection with the vertex set of $B_{r_0}(x,G')$.
Then $\{A'_x,B'_x\}$ is a $2$-separation of $B_{r_0}(x,G')$, and these $2$-separations are compatible, by \autoref{CompatibleSeparationsAddingEdges} (applied with $r:=r^+$, $d:=d$, $F$ the set of $h$-labelled edges in~$G'$, and $r':=r_0$, using $r'\cdot d\le r_0\cdot (r^-/2)\le r^+=r$).
Hence the $2$-separations $\{A'_x,B'_x\}$ make $X$ into an $r_0$-local $2$-separator of~$G'$ by \autoref{LocalSepnCompatibleII}.
\end{proof}

\section{Plan for the proof of \autoref{3-con}} \label{sec:proof}

In this section, we state \autoref{lem:involution_free} and \autoref{lem:involution_case}, and we show how they -- combined with the \nameref{long_edge}~(\ref{long_edge}) -- imply \autoref{3-con}. 
To avoid stating every lemma and theorem with the same long list of assumptions, we summarise them in the following setting and refer to it instead:

\begin{setting}\label{set}
Let $\Gamma$ be a group that is nilpotent of class~$\le n$, and $S\subseteq \Gamma\sm\{\unit\}$ a generating set, and $G:=\cay{\Gamma}{S}$.
Let $r\ge 3$ be an integer.
Let $h\in S$ be such that $\X$ is an $r$-local $2$-separator of~$G$ and suppose that $G$ has no $r$-local cutvertex.
\end{setting}

\begin{thm}
\label{lem:involution_free}
    Assume \autoref{set} with $r\ge 2^{n+2}$.
    If $h$ is not an involution, then $S\subseteq\{h^{\pm1},h^{\pm2}\}$.
\end{thm}

\begin{thm}
\label{lem:involution_case}
    Assume \autoref{set} with $r\ge \max\{2^{n+2},10\}$.
    If $h$ is an involution, then the subgroup of~$\Gamma$ generated by~$h$ is normal.
\end{thm}

We prepare to prove \autoref{3-con}.
If $G=\cay{\Gamma}{S}$ is a Cayley graph and $h\in S$ is an involution that generates a normal subgroup $H:=\langle h\rangle=\{\unit,h\}$ in~$\Gamma$, then we denote by $G/h$ the graph obtained from $G$ by contracting the perfect matching formed by all edges labelled by~$h=h\i$ and identifying all other edges $\{g_1,g_2\}$ with $\{g_1',g_2'\}$ whenever $g_i H=g_i' H$ for both $i=1,2$, see \autoref{fig:collapse_subgroup}.
The vertices of $G/h$ are precisely the equivalence classes $[g]=gH=\{g,gh\}$ of $\Gamma/H$.

\begin{figure}[ht]
    \centering
\begin{tikzpicture}
	\begin{pgfonlayer}{nodelayer}
		\node [style=none] (4) at (-1.75, 1.25) {};
		\node [style=none] (5) at (-1.75, -3.25) {};
		\node [style=black dot] (16) at (0, 0.25) {};
		\node [style=black dot] (17) at (0, -2.25) {};
		\node [style=none] (18) at (0.5, -1) {$h$};
		\node [style=none] (19) at (0, -3) {$\unit$};
		\node [style=none] (20) at (1.75, 1.25) {};
		\node [style=none] (21) at (1.75, -3.25) {};
		\node [style=none] (22) at (-1.75, 0.75) {};
		\node [style=none] (23) at (-1.75, -0.25) {};
		\node [style=none] (24) at (-1.75, -2.75) {};
		\node [style=none] (25) at (-1.75, -1.75) {};
		\node [style=none] (26) at (1.75, -2.75) {};
		\node [style=none] (27) at (1.75, -1.75) {};
		\node [style=none] (28) at (1.75, 0.75) {};
		\node [style=none] (29) at (1.75, -0.25) {};
		\node [style=none] (30) at (-3.25, -1) {$G$};
		\node [style=none] (37) at (5.25, 1.25) {};
		\node [style=none] (38) at (5.25, -3.25) {};
		\node [style=black dot] (39) at (7, -1) {};
		\node [style=black dot] (40) at (7, -1) {};
		\node [style=none] (42) at (7, -1.75) {$[\unit]$};
		\node [style=none] (43) at (8.75, 1.25) {};
		\node [style=none] (44) at (8.75, -3.25) {};
		\node [style=none] (45) at (5.25, 0.75) {};
		\node [style=none] (46) at (5.25, -0.25) {};
		\node [style=none] (47) at (5.25, -2.75) {};
		\node [style=none] (48) at (5.25, -1.75) {};
		\node [style=none] (49) at (8.75, -2.75) {};
		\node [style=none] (50) at (8.75, -1.75) {};
		\node [style=none] (51) at (8.75, 0.75) {};
		\node [style=none] (52) at (8.75, -0.25) {};
		\node [style=none] (53) at (10.5, -1) {$G/h$};
		\node [style=none] (54) at (3.5, -1) {$\mapsto$};
	\end{pgfonlayer}
	\begin{pgfonlayer}{edgelayer}
		\draw [style=d1] (5.center)
			 to [bend right=90, looseness=0.50] (4.center)
			 to [bend right=90, looseness=0.50] cycle;
		\draw [style=c0] (16) to (17);
		\draw [style=d1] (21.center)
			 to [bend right=90, looseness=0.50] (20.center)
			 to [bend right=90, looseness=0.50] cycle;
		\draw [style=c1] (16) to (22.center);
		\draw [style=c1] (23.center) to (16);
		\draw [style=c1] (16) to (28.center);
		\draw [style=c1] (29.center) to (16);
		\draw [style=c1] (25.center) to (17);
		\draw [style=c1] (17) to (24.center);
		\draw [style=c1] (17) to (27.center);
		\draw [style=c1] (26.center) to (17);
		\draw [style=d1] (38.center)
			 to [bend right=90, looseness=0.50] (37.center)
			 to [bend right=90, looseness=0.50] cycle;
		\draw [style=d1, bend right=90, looseness=0.50] (44.center) to (43.center);
		\draw [style=d1, bend right=90, looseness=0.50] (43.center) to (44.center);
		\draw [style=c1] (39) to (45.center);
		\draw [style=c1] (46.center) to (39);
		\draw [style=c1] (39) to (51.center);
		\draw [style=c1] (52.center) to (39);
		\draw [style=c1] (48.center) to (40);
		\draw [style=c1] (40) to (47.center);
		\draw [style=c1] (40) to (50.center);
		\draw [style=c1] (49.center) to (40);
	\end{pgfonlayer}
\end{tikzpicture}
\caption{The graph on the right is obtained from the graph on the left by factoring out a normal subgroup generated by an involution $h$.}
\label{fig:collapse_subgroup}
\end{figure}
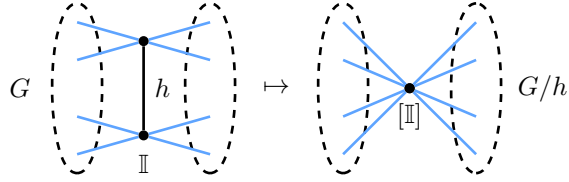

\begin{lem}\label{lem:quotient_is_contraction}
    Let $G=\cay{\Gamma}{S}$ be a Cayley graph and $h\in S$ an involution. 
    If the subgroup $H:=\langle h\rangle$ of~$\Gamma$ generated by~$h$ is normal in~$\Gamma$, then $\Gamma/H$ is a group with $\cay{\Gamma/H}{S/H}=G/h$.
\end{lem}
\begin{proof}
     It remains to prove that there is a suitable bijection between the edges of $G/h$ and those of $\cay{\Gamma/H}{S/H}$.
     \begin{enumerate}
         \item[] There is an edge joining $[g]$ to $[g']$ in $G/H$
         \item[$\Leftrightarrow$] at least one of the edges $\{g,g'\}$, $\{g,g'h\}$, $\{gh,g'\}$ or $\{gh,g'h\}$ exists in~$G$
         \item[$\Leftrightarrow$] $[g]$ and $[g']$ are joined by an edge in $\cay{\Gamma/H}{S/H}$.\qedhere
     \end{enumerate}
\end{proof}

\begin{lem}\label{Loc2sepToLocCutvx}
    Let $G=\cay{\Gamma}{S}$ be a Cayley graph. Suppose that $h\in S$ is an involution and that $\langle h\rangle$ is normal in $\Gamma$. If $\X$ is an $(r+2)$-local $2$-separator of~$G$ for some~$r\ge 2$, then $\X$ is an $r$-local cutvertex of~$G/h$.
\end{lem}
\begin{proof}
Let $\tilde G:=G/h$.
Let $q$ map every vertex $v\in V(G)$ and every edge $e\in E(G)$ to the vertex $\tilde v\in V(\tilde G)$ or edge $\tilde e\in E(\tilde G)$ of the contraction minor $\tilde G$ to which it corresponds.

\begin{sublem}\label{walkLifting}
    Let $\tilde W=u_0 f_0 u_1\ldots u_{k-1} f_{k-1} u_k$ be a walk in~$\tilde G$.
    Let $j$ be an index of this walk, and let $p$ be a vertex or edge of $G$ with $\tilde p=u_j$ or $\tilde p=f_j$.
    Then there is a walk $W=v_0 e_0 v_1\ldots v_{k-1} e_{k-1} v_k$ in~$G$ with $\tilde v_i=u_i$ and $\tilde e_i=f_i$ for all~$i$ and with $v_j=p$ or $e_j=p$.
\end{sublem}
\begin{cproof}
    The edges labelled by $h$ from a perfect matching in~$G$, and $u_i$ and $u_{i+1}$ viewed as 2-sets in $V(G)$ contain a matching between them in~$G$ (as $e_i$ exists and $\langle h\rangle$ is normal).
    Hence we can find $W$ by starting at $p$ and then greedily extending it in two directions.
\end{cproof}

\begin{sublem}\label{walkLiftBall}
    $q\i (B_r(\tilde v,\tilde G))\se B_{r+2}(v,G)$ for all $v,\tilde v$ with $v\in\tilde v\in V(\tilde G)$.
\end{sublem}
\begin{cproof}
    Suppose first that $u\in q\i (B_r(\tilde v,\tilde G))$ is a vertex.
    As $\tilde u\in B_r(\tilde v,\tilde G)$, there is a closed walk $\tilde W$ in $\tilde G$ of length~$\le r$ that uses both $\tilde u$ and $\tilde v$.
    By \autoref{walkLifting}, there is a walk $W$ in $G$ of length~$\le r$ that uses $u$ and has both ends contained in the branch set~$\tilde v\se V(G)$.
    The walk $W$ is closed and witnesses $u\in B_{r+2}(v,G)$ -- unless the ends of $W$ in $\tilde v$ are distinct or $W$ is closed but its ends are equal to the other vertex in the branch set~$\tilde v$.
    In either of the remaining two cases, we can amend $W$ by adding a subwalk of length $\le 2$ that only uses the $h$-labelled edge that spans the branch set~$\tilde v$.
\end{cproof}\medskip

Let $X:=\X$ and let us write $x$ instead of $X$ whenever we consider $X$ as a vertex of~$\tilde G$.
We assume for a contradiction that $B_r(x,\tilde G)-x$ is connected.
Let $v\in X$ be arbitrary.
To obtain a contradiction, it suffices to show that $B_{r+2}(v,G)\sm X$ is connected.
For this, let $a,b\in N(X)$; we have to show that there is an $a$--$b$ walk in $B_{r+2}(v,G)\sm X$.
By assumption, there is an $\tilde a$--$\tilde b$ walk $\tilde W$ in $B_r(x,\tilde G)-x$.
\autoref{walkLifting} obtains a walk $W$ with ends $a$ and $b'\in\tilde b$ from this, where $W$ avoids $X$.
Moreover, $W\subseteq B_{r+2}(v,G)$ by \autoref{walkLiftBall}.
If $b'=b$ we are done; otherwise we append the edge $b'b$ that is labelled by~$h$.
\end{proof}

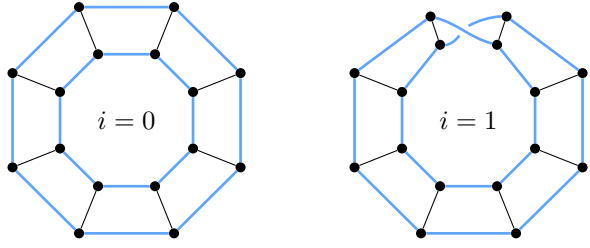
\begin{figure}[ht]
    \centering
\begin{tikzpicture}
	\begin{pgfonlayer}{nodelayer}
		\node [style=black dot] (0) at (-0.75, 1.75) {};
		\node [style=black dot] (1) at (0.75, 1.75) {};
		\node [style=black dot] (2) at (-0.75, -1.75) {};
		\node [style=black dot] (3) at (0.75, -1.75) {};
		\node [style=black dot] (4) at (-1.75, 0.75) {};
		\node [style=black dot] (5) at (-1.75, -0.75) {};
		\node [style=black dot] (6) at (1.75, -0.75) {};
		\node [style=black dot] (7) at (1.75, 0.75) {};
		\node [style=black dot] (12) at (3, -1.25) {};
		\node [style=black dot] (13) at (3, 1.25) {};
		\node [style=black dot] (14) at (-3, 1.25) {};
		\node [style=black dot] (15) at (-3, -1.25) {};
		\node [style=black dot] (16) at (-1.25, 3) {};
		\node [style=black dot] (17) at (1.25, 3) {};
		\node [style=black dot] (18) at (-1.25, -3) {};
		\node [style=black dot] (19) at (1.25, -3) {};
		\node [style=black dot] (20) at (8.25, 2) {};
		\node [style=black dot] (21) at (9.75, 2) {};
		\node [style=black dot] (22) at (8.25, -1.75) {};
		\node [style=black dot] (23) at (9.75, -1.75) {};
		\node [style=black dot] (24) at (7.25, 0.75) {};
		\node [style=black dot] (25) at (7.25, -0.75) {};
		\node [style=black dot] (26) at (10.75, -0.75) {};
		\node [style=black dot] (27) at (10.75, 0.75) {};
		\node [style=black dot] (28) at (12, -1.25) {};
		\node [style=black dot] (29) at (12, 1.25) {};
		\node [style=black dot] (30) at (6, 1.25) {};
		\node [style=black dot] (31) at (6, -1.25) {};
		\node [style=black dot] (32) at (8, 2.75) {};
		\node [style=black dot] (33) at (10, 2.75) {};
		\node [style=black dot] (34) at (7.75, -3) {};
		\node [style=black dot] (35) at (10.25, -3) {};
		\node [style=none] (36) at (0, 0) {$i=0$};
		\node [style=none] (37) at (9, 0) {$i=1$};
		\node [style=none] (38) at (8.75, 2.25) {};
		\node [style=none] (39) at (9, 2.5) {};
	\end{pgfonlayer}
	\begin{pgfonlayer}{edgelayer}
		\draw (14) to (4);
		\draw (0) to (16);
		\draw (17) to (1);
		\draw (7) to (13);
		\draw (6) to (12);
		\draw (19) to (3);
		\draw (2) to (18);
		\draw (15) to (5);
		\draw [style=c1] (15) to (18);
		\draw [style=c1] (18) to (19);
		\draw [style=c1] (19) to (12);
		\draw [style=c1] (12) to (13);
		\draw [style=c1] (13) to (17);
		\draw [style=c1] (17) to (16);
		\draw [style=c1] (16) to (14);
		\draw [style=c1] (14) to (15);
		\draw [style=c1] (5) to (4);
		\draw [style=c1] (4) to (0);
		\draw [style=c1] (0) to (1);
		\draw [style=c1] (1) to (7);
		\draw [style=c1] (7) to (6);
		\draw [style=c1] (6) to (3);
		\draw [style=c1] (3) to (2);
		\draw [style=c1] (2) to (5);
		\draw (30) to (24);
		\draw (20) to (32);
		\draw (33) to (21);
		\draw (27) to (29);
		\draw (26) to (28);
		\draw (35) to (23);
		\draw (22) to (34);
		\draw (31) to (25);
		\draw [style=c1] (31) to (34);
		\draw [style=c1] (34) to (35);
		\draw [style=c1] (35) to (28);
		\draw [style=c1] (28) to (29);
		\draw [style=c1] (29) to (33);
		\draw [style=c1] (32) to (30);
		\draw [style=c1] (30) to (31);
		\draw [style=c1] (25) to (24);
		\draw [style=c1] (24) to (20);
		\draw [style=c1] (21) to (27);
		\draw [style=c1] (27) to (26);
		\draw [style=c1] (26) to (23);
		\draw [style=c1] (23) to (22);
		\draw [style=c1] (22) to (25);
		\draw [style=c1, in=-15, out=165] (21) to (32);
		\draw [style=c1, bend right, looseness=0.75] (20) to (38.center);
		\draw [style=c1, bend left=15, looseness=0.75] (39.center) to (33);
	\end{pgfonlayer}
\end{tikzpicture}
\caption{The two outcomes in the proof of \autoref{3-con} for $i=0,1$.}
\label{fig:main_proof_diag}
\end{figure}

In the proof of \autoref{3-con}, we will use $r$-local coverings.
However, we will use them only briefly, and \cite[Lemma~4.3]{GraphDec} will do most of the work for us.
Hence we do not introduce $r$-local coverings here, but instead refer to \cite[§4 up to Lemma~4.3]{GraphDec} for an introduction.

\begin{proof}[Proof of \autoref{3-con} assuming \autoref{lem:involution_free} and \autoref{lem:involution_case}]
We are given a finite group $\Gamma$ that is nilpotent of class~$\le n$, and we are given $r\ge\max\{4^{n+1},20\}$.

\indent \ref{StallingsCovering}$\to$\ref{StallingsLocalSeparators}.
By assumption, $\Gamma$ has a Cayley graph $G$ such that the $r$-local covering $G_r$ of $G$ has $\ge 2$ ends that are separated by a set $\hat X$ of $\le 2$ vertices.
Then there is a $(\le 2)$-separation of $G_r$ with separator~$\hat X$.
Restricting the sides of this separation to $\hat X$ and its neighbourhood in $G_r$ yields an $r$-local $(\le 2)$-separation $\{\hat A,\hat B\}$ of $G_r$ with separator~$\hat X$.
Then the $r$-local covering $p_r\colon G_r\to G$ projects $\{\hat A,\hat B\}$ to an $r$-local $(\le 2)$-separation $\{A,B\}$ of $G$ with separator $p_r(\hat X)$, by \cite[Lemma~4.3]{GraphDec}.
If $\Gamma$ has order~$\le r$, then every cycle in $G$ has length~$\le r$, so $G_r=G$ as cycles homotopically generate the fundamental group \cite[Lemma~4.1]{GraphDec}, contradicting that $G_r$ is infinite by assumption.

\ref{StallingsLocalSeparators}$\to$\ref{StallingsProduct}.
    By assumption, some Cayley graph $G=\cay{\Gamma}{S}$ has
    an $r$-local separator of size at most two.
    Let $\dot{r}:=\lfloor{\sqrt{r}\rfloor}$. 
    Our choice of $r$ ensures that $\dot{r}\ge2^{n+1}$. 
    If $G$ contains an $\dot{r}$-local cutvertex, then by \autoref{submainCutvertex} we conclude that $\Gamma$ is cyclic and of order $>r$, in which case we are done.
    So assume that $G$ contains no $\dot{r}$-local cutvertex.
    Let $r':=\max\left\{2\dot{r},10\right\}$. 
    Our choice of $r$ ensures that $r'\ge\max\{2^{n+2},10\}$.  
    
    We suppose that $S$ contains more than just one letter and its inverse, as otherwise $\Gamma$ would be cyclic and we would be done.
    We claim that $G$ contains a cycle of length at most~$r$. 
    Indeed, since the group $\Gamma$ is nilpotent of class~$\le n$ and~$S$ contains more than one letter and its inverse, the $n$th iterated commutator word in two inequivalent elements of~$S$ gives a closed walk of length at most $2^{n+1}\le r$ by \autoref{commutatorLength}. 
    By \autoref{lem:cycle_finder}, $G$ includes a cycle of length at most~$r$. 
    By \ref{StallingsLocalSeparators}, this cycle does not span~$G$.

    Now we invoke the \nameref{long_edge} (\ref{long_edge}) with $r^+=r$, $r^-=\dot{r}$, and $r_0=r'$ to obtain a new Cayley graph $G'=\cay{\Gamma}{S'}$ with $\X$ an $r'$-local $2$-separator and no $r'$-local cutvertices, where $S':=S\cup\{h^{\pm1}\}$.
    If $h$ is not an involution, then we apply \autoref{lem:involution_free} (using $r'\ge 2^{n+2}$) to deduce that $S'\subseteq\{h^{\pm1},h^{\pm2}\}$, so $\Gamma$ is cyclic and we are done.
    Hence we may assume now that $h$ is an involution.
    
    We apply \autoref{lem:involution_case}, using that $r'\ge\max\{2^{n+2},10\}$. We obtain that the subgroup $H:=\langle h\rangle$ generated by $h$ is normal in~$\Gamma$.
    Consider the graph $G'/h$ obtained by contracting the edges labelled by~$h$ and identifying edges that are in parallel after performing the contractions. 
    By \autoref{lem:quotient_is_contraction}, $G'/h=\cay{\Gamma/H}{S/H}$.
    By \autoref{Loc2sepToLocCutvx}, the $r'$-local $2$-separator $\X$ of $G'$ appears in $G'/h$ as an $(r'-2)$-local cutvertex. Since $r'\geq 2^{n+2}> 2^{n+1}+2$, we may apply \autoref{submainCutvertex} to obtain that $G'/h$ is a cycle of length~$\ge r'-1$ and $\Gamma/H$ is cyclic. 
    Hence there is some $g\in\Gamma$ such that $\Gamma/H=\langle [g]\rangle$. 
    We conclude that $g$ and $h$ together generate $\Gamma$. Since $H$ is normal in~$\Gamma$, the elements $g$ and $h$ commute.
    So, relators in $\Gamma$ aside from $h^2$ and the commutator can be simplified to a single relator:
    \[
        \Gamma \quad =\quad \langle\;\; g,h \;:\; h^2,\;ghg\i h,\;g^kh^i \;\;\rangle
    \]
    for some $k\in\N$, $i\in\{0,1\}$, see \autoref{fig:main_proof_diag}.
    We choose $k,i$ with $k$ as small as possible.
    If $i=1$, then $h=g^k$ and so $\Gamma\cong C_{2k}$, with $2k>r$ as $\Gamma$ has order~$>r$ by~\ref{StallingsLocalSeparators}.
    Otherwise $i=0$, and then $\Gamma\cong C_k\times C_2$. 
    Since the $C_2$-factor stems from~$H=\langle h\rangle$, we find that the Cayley graph $G$ looks like the one depicted in \autoref{fig:main_proof_diag} (on the left-hand side), where the edges labelled with~$h$ are black and the edges labelled with~$g$ are blue.
    So both blue cycles weakly traverse the $r$-local $2$-separator~$\X$ of~$G$, but only at one vertex.
    Hence the blue cycles have length~$>r$ by \autoref{CycleWeakTraversal}, and so $k>r$.

    \ref{StallingsProduct}$\to$\ref{StallingsCovering}.
    If $\Gamma$ is cyclic, then we consider a cycle as Cayley graph~$G$ of $\Gamma$.
    Since this cycle has length $>r$, its $r$-local covering is a two-way infinite path, which has two ends separated by a cutvertex.
    If $\Gamma$ is isomorphic to $C_i\times C_2$ for some $i>r$, then we let $G$ be the Cayley graph depicted in \autoref{fig:loc-sepr}.
    The $r$-local covering of $G$, also depicted in \autoref{fig:loc-sepr}, has two ends that are separated by two vertices.
\end{proof}

It remains to prove \autoref{lem:involution_free} and \autoref{lem:involution_case}, which we do in \autoref{sec:noninvo} and \autoref{sec:invo}, respectively.

\section{Non-involution case}\label{sec:noninvo}

In this section we prove \autoref{lem:involution_free} via the following two propositions.

\begin{prop}\label{lem:problem_morphemes}
    Assume \autoref{set} with $r\ge 2^{n+2}$, and suppose that $h$ is not an involution. 
    Then there is no $g\in S$ such that $g$ traverses~$\X$ at $\unit$, the word $gh^2$ traverses $\X$ strongly, and both $g\i h^2g\i h\i$ and $ghgh$ are morphemes.
\end{prop}

\begin{prop}\label{non-involution}
Assume \autoref{set} with $r\ge 2^{n+2}$, and suppose that $h$ is not an involution.
Then $S\se\{h\pmo,h^{\pm 2}\}$, or else there is some $g\in S$ such that $g$ traverses~$\X$ at $\unit$, the word $gh^2$ strongly traverses~$\X$, and both $g^{-1}h^2g^{-1}h^{-1}$ and $ghgh$ are morphemes.
\end{prop}

\begin{proof}[Proof of \autoref{lem:involution_free} assuming \autoref{lem:problem_morphemes} and \autoref{non-involution}:]
we combine them.
\end{proof}

We prove \autoref{lem:problem_morphemes} in \autoref{subsec:FirstProp}, and we prove \autoref{non-involution} in \autoref{subsec:SecondProp}.

\subsection{Proof of \autoref{lem:problem_morphemes}}\label{subsec:FirstProp}

\begin{lem}\label{lem:add_z}
    Assume \autoref{set} with~$r\ge4$.
    Let $\{A,B\}$ be an $r$-local $2$-separation of~$G$ with separator~$\X$.
    Suppose that $g\in S$ traverses~$\X$ at $\unit$ so that $g\i\in A$ and $g\in B$, and further suppose that $gh^2$ traverses~$\X$ strongly.
    Assume that $z:=gh$ is an involution.
    Then $\{A,B\cup\{z\}\}$ is an $\frac{r}{2}$-local $2$-separation of $G':=\cay{\Gamma}{S\cup\{z\}}$ with separator~$\X$.
\end{lem}
\begin{proof}
    If $z$ is in~$S$, there is nothing to show, so we may assume that this is not the case.
    We claim that 
    \begin{equation}\label{eq:NGdash}
        N_{G'}(\X)=N_G(\X)\cup\{z\}.
    \end{equation}
    Indeed, since $z$ is an involution, both vertices $\unit$ and $h$ get one (possibly new) neighbour in~$G'$, namely $z$ and~$hz$, respectively. 
    While $z\notin S$ gives $z\notin N_G(\X)$, using that $z$ is an involution we have
    \begin{equation}\label{eq:hz}
        hz=h(gh)=h(h\i g\i)=g\i\in S,
    \end{equation}
    so $hz$ is not a new neighbour; only~$z$ is.
    We will require the following sublemma:
    \begin{sublem}\label{NOghTraversal}
    The word $gh$ does not traverse~$\{A,B\}$.
    \end{sublem}
    \begin{cproof}
    Since $gh$ has length two, it suffices to show that $gh$ does not weakly traverse~$\{A,B\}$.
    Since the separator of $\{A,B\}$ contains the vertex~$h$, the word $gh$ does not weakly traverse $\{A,B\}$ at~$\unit$. 
    It remains to show that the word $gh$ does not weakly traverse~$\{A,B\}$ at~$h$. 
    For this, it suffices to show that both vertices $hg\i$ and $h^2$ lie in $B$.
    Since $g\i\in A$ and $gh^2$ strongly traverses~$\{A,B\}$, we have $h^2\in B$, so it remains to show that $hg\i$ is also in~$B$.
    For this, we consider the cycle $O$ in~$G$ that is labelled by the morpheme $hghg$ and that uses the vertices $h$, $h^2$, $h^2g$ and $h^2gh=hg\i$ in this order; see \autoref{fig:triangle}.
    Let us show that $O$ avoids the vertex~$\unit$, and suppose for a contradiction that it doesn't.
    Then one of the words $h$, $h^2$, $h^2g$, $hg\i$ evaluates to~$\unit$ in~$\Gamma$.
    It cannot be~$h$, and it cannot be $h^2$ since $gh^2$ strongly traverses~$\{A,B\}$.
    It cannot be $h^2 g$ since $h(hg)=h(g\i h\i)=hg\i h\i\neq\unit$.
    But it also cannot be $hg\i$, since $g\nequiv h$, a contradiction. So $\unit\notin O$.
    Hence $O$ cannot strongly traverse~$\{A,B\}$.
    If~$O$ would traverse $\{A,B\}$ weakly, then it would traverse $\{A,B\}$ weakly at both $\unit$ and~$h$ by \autoref{CycleWeakTraversal} (using $r\ge 4$), contradicting that $\unit$ is not on~$O$.
    So $O$ does not traverse $\{A,B\}$ at all.
    In combination with $\unit\notin O$, it follows that the two neighbours of~$h$ on~$O$, namely $hg\i$ and $h^2$, lie on the same side of~$\{A,B\}$.
    Since $h^2$ lies in~$B$ as shown earlier, this means that $hg\i$ lies in~$B$ as well. This completes our proof.
    \end{cproof}\medskip

Since $G$ has no $r$-local cutvertex by assumption, $\{A,B\}$ induces $2$-separations $\{A_x,B_x\}$ of $B_r(x,G)$ with separator~$X$ for both $x\in X$.
Let $A'_x,B'_x$ be obtained from $A_x,B_x$ by taking the intersection with the vertex set of $B_{r/2}(x,G')$.
By \autoref{NOghTraversal}, the ends of every $z$-labelled edge are linked by walk in $G$ of length two that does not traverse~$\{A,B\}$.
This allows us to apply \autoref{CompatibleSeparationsAddingEdges} with $r:=r$, $d:=2$, $F$~the set of $z$-labelled edges, and $r':=r/2$, which yields that the sets $\{A'_x,B'_x\}$ are compatible $2$-separations of $B_{r/2}(x,G')$.
Hence they make $X$ into the separator of an $\frac{r}{2}$-local $2$-separation $\{A',B'\}$ of~$G'$ with $A\sm B\se A'\sm B'$ and $B\sm A\se B'\sm A'$, by \autoref{LocalSepnCompatibleII}.
By (\ref{eq:NGdash}), it remains to show $z\in B'\sm A'$.

Since $g\in B\sm A$ by assumption, since $z=gh$, and since the vertex $z$ does not lie in~$X$ by (\ref{eq:NGdash}), we have $z\in B_{\unit}\sm A_{\unit}$ (using $r\ge 4$).
Hence $z\in B'_{\unit}\sm A'_{\unit}\se B'\sm A'$ as desired.
\end{proof}

\begin{figure}[ht]
    \centering
\begin{tikzpicture}
	\begin{pgfonlayer}{nodelayer}
		\node [style=black dot] (0) at (-2, -1.5) {};
		\node [style=black dot] (1) at (2, -1.5) {};
		\node [style=black dot] (2) at (0, 2) {};
		\node [style=none] (3) at (-1, 0.25) {};
		\node [style=none] (4) at (1, 0.25) {};
		\node [style=none] (5) at (0, -1.5) {};
		\node [style=black dot] (9) at (5.5, 0.5) {};
		\node [style=black dot] (10) at (3.5, 4) {};
		\node [style=none] (12) at (1.75, 3) {};
		\node [style=none] (13) at (0, -5) {};
		\node [style=none] (14) at (-4.5, 2.25) {};
		\node [style=black dot] (15) at (-3.5, 4) {};
		\node [style=black dot] (16) at (-5.5, 0.5) {};
		\node [style=none] (19) at (4.5, 2.25) {};
		\node [style=none] (20) at (-1.75, 3) {};
		\node [style=none] (21) at (-3.75, -0.5) {};
		\node [style=none] (22) at (3.75, -0.5) {};
		\node [style=black dot] (23) at (-2, -5) {};
		\node [style=black dot] (24) at (2, -5) {};
		\node [style=none] (25) at (2, -3.25) {};
		\node [style=none] (26) at (-2, -3.25) {};
		\node [style=none] (27) at (-2.5, -2) {$\unit$};
		\node [style=none] (28) at (0, 2.75) {$h$};
		\node [style=none] (29) at (3.75, -2) {$h^2=h\i$};
		\node [style=none] (30) at (8.25, 0.25) {$hg\i h\i = h^2g$};
		\node [style=none] (31) at (4.5, 4.25) {$hg\i$};
		\node [style=none] (32) at (-2.75, 0.75) {$\colorTwo{z}$};
		\node [style=none] (33) at (3.25, 1.25) {$\colorTwo{z}$};
		\node [style=none] (34) at (0.25, -2.75) {$\colorTwo{z}$};
		\node [style=none] (35) at (2, 2.5) {};
		\node [style=none] (36) at (3.5, 0) {};
		\node [style=none] (37) at (1.75, 1.5) {$O$};
		\node [style=none] (38) at (-2.5, -3.25) {$\colorOne{g}$};
		\node [style=none] (40) at (-4, -1) {$\colorOne{g}$};
	\end{pgfonlayer}
	\begin{pgfonlayer}{edgelayer}
		\draw [style=c0 arrow] (0) to (3.center);
		\draw [style=c0 arrow] (1) to (5.center);
		\draw [style=c0 arrow] (2) to (4.center);
		\draw [style=c0] (4.center) to (1);
		\draw [style=c0] (5.center) to (0);
		\draw [style=c0] (3.center) to (2);
		\draw [style=c0 arrow] (15) to (14.center);
		\draw [style=c0 arrow] (9) to (19.center);
		\draw [style=c0 arrow] (23) to (13.center);
		\draw [style=c0] (13.center) to (24);
		\draw [style=c0] (19.center) to (10);
		\draw [style=c0] (14.center) to (16);
		\draw [style=c1 arrow] (2) to (20.center);
		\draw [style=c1 arrow] (16) to (21.center);
		\draw [style=c1 arrow] (0) to (26.center);
		\draw [style=c1 arrow] (24) to (25.center);
		\draw [style=c1 arrow] (1) to (22.center);
		\draw [style=c1 arrow] (10) to (12.center);
		\draw [style=c1] (12.center) to (2);
		\draw [style=c1] (22.center) to (9);
		\draw [style=c1] (1) to (25.center);
		\draw [style=c1] (23) to (26.center);
		\draw [style=c1] (0) to (21.center);
		\draw [style=c1] (15) to (20.center);
		\draw [style=c2] (2) to (16);
		\draw [style=c2] (0) to (24);
		\draw [style=c2] (1) to (10);
		\draw [style=d1] (36.center)
			 to [bend right=90, looseness=1.75] (35.center)
			 to [bend right=90, looseness=1.75] cycle;
	\end{pgfonlayer}
\end{tikzpicture}
\caption{The situation in the proofs of \autoref{lem:add_z} and \autoref{lem:problem_morphemes}}
\label{fig:triangle}
\end{figure}
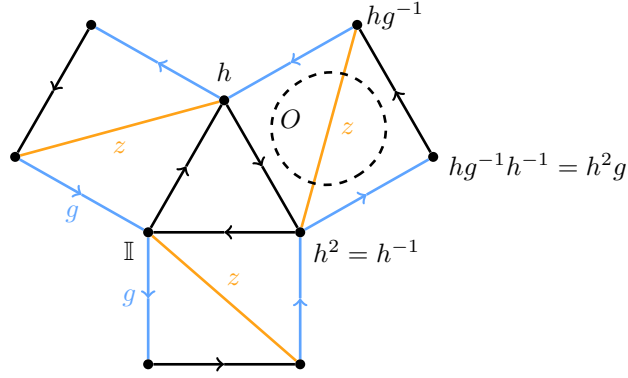

\begin{proof}[Proof of \autoref{lem:problem_morphemes}]
    Suppose for a contradiction that there exists some $g\in S$ such that $g$ traverses~$\X$ at $\unit$, the word $gh^2$ traverses~$\X$ strongly, and both $g\i h^2g\i h\i$ and $ghgh$ are morphemes.

    First, we combine $g\i h^2g\i h\i=\unit$ with $ghgh=\unit$ to obtain~$h^3=\unit$, as follows.
    From $g\i h^2g\i h\i=\unit$ we obtain $h^2 g\i h\i g\i=\unit$.
    From $ghgh=\unit$ we obtain $h\i g\i=gh$.
    By combining the two new terms, we get $\unit=h^2 g\i gh=h^3$.
    See \autoref{fig:triangle} for a depiction.
    
    Since $\X$ is an $r$-local $2$-separator of~$G$ and $g$ traverses~$\X$ at~$\unit$ by assumption, 
    we find an $r$-local $2$-separation~$\{A,B\}$ with separator~$\X$ such that $g\i\in A$ and~$g\in B$.
    Let $z:=gh$.
    Note that $z$ is an involution since $ghgh=\unit$,
    and let $G':=\cay{\Gamma}{S\cup\{z\}}$.
    Then $\{A,B\cup\{z\}\}$ is an $\frac{r}{2}$-local $2$-separation of~$G'$ with separator $\X$ by \autoref{lem:add_z}. 
    
    We claim that the word $zhz$ strongly traverses the $\frac{r}{2}$-local $2$-separation~$\{A,B\cup\{z\}\}$ in~$G'$.
    Indeed, on the one hand, the vertex $z\i=z$ is adjacent to~$\unit$ and lies in the side~$B\cup\{z\}$.
    And on the other hand, the vertex $hz=g\i$, see~(\ref{eq:hz}), lies in~$A$ by assumption.

\begin{sublem}\label{zh}
$z\nequiv h$.
\end{sublem}
\begin{cproof}
If $z=h$, then $gh=h$ yields $g=\unit$, which is not possible.
If $z=h\i$, then $gh=h\i$ in combination with $h^3=\unit$ yields $g=h^{-2}=h$, which is not possible as $g$ traverses $\X$ at~$\unit$.
\end{cproof}\medskip

Since $z\nequiv h$ by \autoref{zh}, we may use \autoref{commutatorContainsMorpheme} to find a morpheme $m$ in the iterated commutator word $[z,h]_n$ as a subword, and $m$ has length at least three since the iterated commutator word alternates in $z$ and~$h\pmo$.
Hence $m$ contains $zhz$ as a cylic subword.
By \autoref{commutatorLength}, $m$ has length at most~$\frac{r}{2}$.
So we can apply the \nameref{traverser} (\ref{traverser}) to obtain the desired contradiction. 
\end{proof}

\subsection{Proof of \autoref{non-involution}}\label{subsec:SecondProp}

    Let $h\nequiv g, s$ be letters and define $w_n:=[g,h^2]_n$ and $\vvv_n:=[s,h^3]_n$ (read $\vvv_n$ as \lq triple-u sub n\rq).

\begin{obs}\label{obs:commlen3}
    $|w_n|=3\cdot 2^n$ and $|\vvv_n|=2^{n+2}$.
\end{obs}
\begin{proof}[Proof:] apply \autoref{commutatorLength}.
\end{proof}

\begin{lem}\label{finish-hcube}
Assume \autoref{set} with $r\ge2^{n+2}$. If $h^3$ traverses $\X$ strongly, then $S\subseteq\{h\pmo, h^{\pm 2}\}$.
\end{lem}
\begin{proof}
    Let us assume for a contradiction that $h^3$ traverses $\X$ strongly, but $S\not\se\{h\pmo, h^{\pm 2}\}$.
    Pick $s\in S\sm\{h\pmo, h^{\pm 2}\}$ arbitrarily.
    Since $s\nequiv h$, by the \nameref{commutatorContainsMorpheme} (\ref{commutatorContainsMorpheme}) we find a morpheme $\vvv$ that is a linear subword of the iterated commutator $\vvv_n=[s,h^3]_n$. 
    By \autoref{obs:commlen3}, we have $|\vvv_n|\leq r$, so $|\vvv|\le r$. 
    Hence by the \nameref{traverser} (\ref{traverser}), $\vvv$ does not contain $h^3$ as a cyclic subword.
    Since $\vvv$ has no full $h$-segment, it has at most one $s$-segment. First note that  $\vvv$ cannot be equal to $s\pmo$ and also cannot consist of a single $h$-segment as $h^3$ traverses strongly. 
    So, possibly after cyclic permutation, $\vvv$ has only one $h$-segment $h^{\pm 1}$ or $h^{\pm 2}$ followed by the letter~$s^{\pm 1}$, contradicting $s\nequiv h,h^2$. 
\end{proof}

\begin{lem}\label{ghpseudomagic}
Let $g\nequiv h$ be elements of a group~$\Gamma$ and suppose that $h$ is not an involution.
Let $w$ be a subword of $w_n=[g,h^2]_n$ that is a morpheme.
Then $w$ contains $gh^2$ or $g\i h^2$ as a cyclic subword.
\end{lem}
\begin{proof}
As $h$ is not an involution and $g \not\equiv h$, the morpheme $w$ is of length at least three, so it contains, perhaps after cyclic permutation, an $h$-segment of length two and a $g$-segment.

If $w$ has exactly two segments, then we cyclically permute $w$ so that $w=g^i h^j$ where $|i|\ge 1$ and $|j|\ge 2$.
If $j$ is positive, then $w$ contains $gh^2$ or $g\i h^2$ as a linear subword.
If $j$ is negative, then $gh^2$ or $g\i h^2$ is a linear subword of $g^{-i}h^{-j}$ by the previous argument.
Since $g^{-i}h^{-j}$ is obtained from $w$ by inversion and cyclic permutation, it follows that $gh^2$ and $g\i h$ are cyclic subwords of~$w$.

Suppose now that $w$ has exactly three segments. 
Then $w=g^i h^j g^k$ or $w=h^i g^j h^k$.
Since $w$ is a morpheme, the exponents $i$ and $k$ have the same sign, so we can cyclically permute $w$ so that it has exactly two segments -- a case which we solved above.

Finally, suppose that $w$ has at least four segments.
Here we cyclically permute $w$ so that it contains a word of the form $g^i h^j g^k$ as a linear subword, where $|i|,|k|\ge 1$ and $|j|\ge 2$.
If $j$ is positive, then we find $gh^2$ or $g\i h^2$ as a linear subword and are done.
If $j$ is negative, then we find $h^{-2}g\i$ or $h^{-2} g$ as a linear subword, and are done as these are the inverses of $gh^2$ and $g\i h^2$, respectively.
\end{proof}

\begin{lem}[Anchor Lemma]\label{anchor}
    Assume \autoref{set} with $r\ge3\cdot2^n$. Suppose that $h$ is not an involution. For every $g\in S\sm\{h\pmo\}$, at least one of the vertices $g$ or $g\i$ lies in the same local component at $\X$ as the vertex $h^2$.
\end{lem}
\begin{proof}
    Using the \nameref{commutatorContainsMorpheme} (\ref{commutatorContainsMorpheme}), take a morpheme $w$ that is a subword of the iterated commutator word $w_n=[g,h^2]_n$. 
    Then $w$ must contain at least one of $gh^2$ or $g\i h^2$ as a cyclic subword by \autoref{ghpseudomagic}. By \autoref{obs:commlen3} and the \nameref{traverser} (\ref{traverser}) applied to $w_n$ and $w$, not both of $gh^2$ and $g\i h^2$ can traverse strongly. Hence at least one of the vertices $g$ or $g\i$ is in the same local component at $\X$ as the vertex~$h^2$.
\end{proof}

\begin{lem}\label{trav_exists}
    Assume \autoref{set} with $r\ge2^{n+2}$.
    Suppose that $h$ is not an involution.
    Then there is some $g\in S$ that traverses $\X$ at~$\unit$, or else $S\subseteq\{h\pmo, h^{\pm 2}\}$.
\end{lem}
\begin{proof}
    Since $\X$ is an $r$-local $2$-separator and neither $\unit$ nor $h$ are $r$-local cutvertices by an assumption in \autoref{set}, it follows by \autoref{LocalCompsAtXseeX} that there is a neighbour $g$ of $\unit$ such that $g$ and $h^2$ lie in distinct local components at $\X$. 
    If $g\neq h^{-1}$, then $g^{-1}$ lies in the same local component as $h^2$ by the \nameref{anchor}~(\ref{anchor}); thus $g$ traverses $\{\unit, h\}$ at~$\unit$.
    Otherwise $g=h\i$. Then $h^3$ strongly traverses $\{\unit, h\}$. Thus $S\subseteq\{h\pmo, h^{\pm 2}\}$ by \autoref{finish-hcube}.
\end{proof}

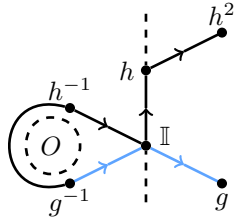
\begin{figure}[ht]
    \centering
\begin{tikzpicture}
	\begin{pgfonlayer}{nodelayer}
		\node [style=black dot] (16) at (0, 0) {};
		\node [style=black dot] (17) at (0, -2) {};
		\node [style=none] (18) at (0, -1) {};
		\node [style=none] (19) at (2, 1) {};
		\node [style=none] (20) at (2, -3) {};
		\node [style=none] (21) at (-2, -3) {};
		\node [style=none] (22) at (1, 0.5) {};
		\node [style=none] (23) at (1, -2.5) {};
		\node [style=none] (24) at (-1, -2.5) {};
		\node [style=none] (25) at (0, 1.5) {};
		\node [style=none] (26) at (0, -3.5) {};
		\node [style=black dot] (27) at (2, 1) {};
		\node [style=black dot] (28) at (2, -3) {};
		\node [style=black dot] (29) at (-2, -3) {};
		\node [style=none] (30) at (-2, -3.5) {$g\i$};
		\node [style=none] (31) at (2, -3.5) {$g$};
		\node [style=none] (32) at (2, 1.5) {$h^2$};
		\node [style=none] (35) at (-0.5, 0) {$h$};
		\node [style=black dot] (36) at (-2, -1) {};
		\node [style=none] (37) at (-1, -1.5) {};
		\node [style=none] (38) at (-1, -1.5) {};
		\node [style=none] (39) at (-2, -0.5) {$h\i$};
		\node [style=none] (40) at (-2.5, -2) {$O$};
		\node [style=none] (41) at (0.5, -1.75) {$\unit$};
		\node [style=none] (42) at (-2.5, -1.25) {};
		\node [style=none] (43) at (-2.5, -2.75) {};
	\end{pgfonlayer}
	\begin{pgfonlayer}{edgelayer}
		\draw [style=c0 arrow] (17) to (18.center);
		\draw [style=d1] (25.center) to (16);
		\draw [style=d1] (26.center) to (17);
		\draw [style=c0 arrow] (16) to (22.center);
		\draw [style=c0] (19.center) to (22.center);
		\draw [style=c0] (18.center) to (16);
		\draw [style=c1 arrow] (17) to (23.center);
		\draw [style=c1 arrow] (21.center) to (24.center);
		\draw [style=c1] (24.center) to (17);
		\draw [style=c1] (23.center) to (20.center);
		\draw [style=c0 arrow] (36) to (38.center);
		\draw [style=c0] (38.center) to (17);
		\draw [style=c0, bend right=105, looseness=2.50] (36) to (29);
		\draw [style=d1] (43.center)
			 to [bend right=270, looseness=1.50] (42.center)
			 to [bend left=90, looseness=1.75] cycle;
	\end{pgfonlayer}
\end{tikzpicture}
\caption{The situation in the proof \autoref{power_to_traverse}}
\label{hcubepic}
\end{figure}

\begin{lem}\label{power_to_traverse}
Assume \autoref{set} with $r\ge 3\cdot 2^n$. Suppose that $h$ is not an involution. If $g\in S$ traverses $\X$ at $\unit$ and $g\equiv h^k$ for an integer~$k$ with $2\leq |k|\leq r-1$, then the word $h^3$ traverses $\X$ strongly.
\end{lem}
\begin{proof}
    Since $h$ is not an involution, $h^2$ lies in a local component at~$\X$.
    By \autoref{anchor}, either $g$ or $g\i$ lies in the same local component at~$\X$ as~$h^2$. 
    Replacing $g$ by $g^{-1}$ if necessary, we may assume that $g$ lies in the same local component as~$h^2$.
    Let $k\in\N$ be minimal such that $g=h^{-k}$ or~$g=h^k$, and note that $k\ge 2$ since $g$ traverses~$\X$.
    Then $gh^k$ or $gh^{-k}$ is a morpheme; let $O$ denote the cycle in~$G$ labelled by it such that $O$ contains the edge joining $g\i$ to~$\unit$, see \autoref{hcubepic}.
    Note that $O$ has length at most~$r$.
    
    We claim that $O$ is labelled by~$gh^{-k}$.
    Indeed, otherwise $gh^k$ is the morpheme labelling~$O$.
    But then it traverses $\X$ strongly, contradicting the \nameref{traverser} (\ref{traverser}).
    So $O$ is labelled by~$gh^{-k}$.
    
    Let us assume for a contradiction that $h^3$ does not traverse~$\X$ strongly.
    Then the vertex $h\i$ lies in the same local component as~$h^2$, which does not contain~$g\i$.
    Hence $gh\i$, and in particular~$O$, weakly traverse~$\X$ at~$\unit$.
    By \autoref{CycleWeakTraversal}, $O$ must weakly traverse~$\X$ at~$h$ as well.
    But then the vertex~$h$ has $h$-degree greater than two (indeed, the edge joining~$\unit$ to~$h$ is not on $O$ and the vertex $h$ cannot be an endvertex of the edge of $O$ labelled by $g\i$, and so the vertex $h$ has $h$-degree two on $O$). This is a contradiction. 
\end{proof}

\begin{prop}\label{subsec}
    Assume \autoref{set} with $r\ge 3\cdot 2^n$. Suppose that there is some $g\in S$ such that $g$ traverses $\X$ at $\unit$ and the word $gh^2$ traverses $\X$ strongly.
    Then there is a natural number~$k$ with $2\leq k\leq 4$ such that 
    $g\equiv h^k$,  
    or else $g^{-1}h^2g^{-1}h^{-1}$ and $ghgh$ are morphemes in~$\Gamma$.
\end{prop}

\begin{proof}[Proof of \autoref{non-involution} assuming \autoref{subsec}]
Assume \autoref{set} with $r\ge 2^{n+2}$, and suppose that $h$ is not an involution.
We are required to show that $S\subseteq\{h\pmo, h^{\pm 2}\}$, or else that there is some $g\in S$ such that $g$ traverses $\X$ at~$\unit$, the word $gh^2$ strongly traverses $\X$, and $g^{-1}h^2g^{-1}h^{-1}$ and $ghgh$ are morphemes in $G$.
Since we are done otherwise, assume that $s\nequiv h,h^2$ for some~$s\in S$.
Then we apply \autoref{trav_exists} to find an element $g\in S$ that traverses $\X$ at~$\unit$. 
By interchanging $g$ and $g\i$ if necessary, we may assume that the vertex $g\i$ lies in a different local component at $\X$ than~$h^2$. Thus $gh^2$ traverses $\X$ strongly.

It remains to show that $g^{-1}h^2g^{-1}h^{-1}$ and $ghgh$ are morphemes. 
To this end, we apply \autoref{subsec}, and assume for a contradiction that it returns a~$k$ with $2\leq k\le4$ such that $g\equiv h^k$ (as otherwise we are done). 
Since $h$ is not an involution, we can apply \autoref{power_to_traverse} to deduce that $h^3$ traverses strongly. Hence $S\subseteq \{h^{\pm 1}, h^{\pm 2}\}$ by \autoref{finish-hcube}, contradicting $s\nequiv h, h^2$.
\end{proof}

\subsection{Proof of \autoref{subsec}}

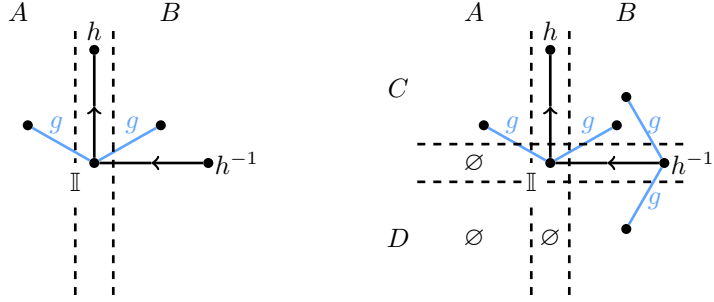
\begin{figure}[ht]
    \centering
\begin{tikzpicture}
	\begin{pgfonlayer}{nodelayer}
		\node [style=black dot] (0) at (0, 0) {};
		\node [style=none] (1) at (0, 1.5) {};
		\node [style=black dot] (2) at (0, 3) {};
		\node [style=none] (5) at (1.5, 0) {};
		\node [style=black dot] (6) at (3, 0) {};
		\node [style=black dot] (8) at (-1.75, 1) {};
		\node [style=black dot] (9) at (1.75, 1) {};
		\node [style=none] (10) at (0.5, 0.5) {};
		\node [style=none] (11) at (-0.75, -0.5) {};
		\node [style=none] (12) at (-0.5, 0) {};
		\node [style=none] (13) at (-0.25, -0.5) {};
		\node [style=none] (14) at (0.5, 3.5) {};
		\node [style=none] (15) at (-0.5, 3.5) {};
		\node [style=none] (16) at (0.5, -3.5) {};
		\node [style=none] (17) at (-0.5, -3.5) {};
		\node [style=none] (22) at (0, 3.5) {$h$};
		\node [style=none] (23) at (3.75, 0) {$h\i$};
		\node [style=none] (24) at (-1, 1) {$\colorOne{g}$};
		\node [style=none] (25) at (1, 1) {$\colorOne{g}$};
		\node [style=none] (26) at (-2, 4) {$A$};
		\node [style=none] (27) at (2, 4) {$B$};
		\node [style=black dot] (33) at (12, 0) {};
		\node [style=none] (34) at (12, 1.5) {};
		\node [style=black dot] (35) at (12, 3) {};
		\node [style=none] (36) at (13.5, 0) {};
		\node [style=black dot] (37) at (15, 0) {};
		\node [style=black dot] (38) at (10.25, 1) {};
		\node [style=black dot] (39) at (13.75, 1) {};
		\node [style=none] (40) at (12.5, 0.5) {};
		\node [style=none] (41) at (11.5, 0.5) {};
		\node [style=none] (43) at (12.5, -0.5) {};
		\node [style=none] (44) at (12.5, 3.5) {};
		\node [style=none] (46) at (12.5, -3.5) {};
		\node [style=none] (49) at (15.5, 0.5) {};
		\node [style=none] (50) at (8.5, 0.5) {};
		\node [style=none] (52) at (12, 3.5) {$h$};
		\node [style=none] (53) at (15.75, 0) {$h\i$};
		\node [style=none] (54) at (11, 1) {$\colorOne{g}$};
		\node [style=none] (55) at (13, 1) {$\colorOne{g}$};
		\node [style=none] (56) at (10, 4) {$A$};
		\node [style=none] (57) at (14, 4) {$B$};
		\node [style=none] (58) at (8, 2) {$C$};
		\node [style=none] (59) at (8, -2) {$D$};
		\node [style=none] (60) at (10, 0) {$\varnothing$};
		\node [style=none] (61) at (10, -2) {$\varnothing$};
		\node [style=none] (62) at (12, -2) {$\varnothing$};
		\node [style=black dot] (63) at (14, 1.75) {};
		\node [style=black dot] (64) at (14, -1.75) {};
		\node [style=none] (65) at (14.75, 1) {$\colorOne{g}$};
		\node [style=none] (66) at (14.75, -1) {$\colorOne{g}$};
		\node [style=none] (67) at (-0.5, -1) {};
		\node [style=none] (68) at (-0.5, -0.5) {$\unit$};
		\node [style=none] (69) at (11.25, -0.5) {};
		\node [style=none] (70) at (11.5, 0) {};
		\node [style=none] (71) at (11.75, -0.5) {};
		\node [style=none] (72) at (11.5, 3.5) {};
		\node [style=none] (73) at (11.5, -3.5) {};
		\node [style=none] (74) at (15.5, -0.5) {};
		\node [style=none] (75) at (8.5, -0.5) {};
		\node [style=none] (76) at (11.5, -1) {};
		\node [style=none] (77) at (11.5, -0.5) {$\unit$};
	\end{pgfonlayer}
	\begin{pgfonlayer}{edgelayer}
		\draw [style=c0 arrow] (0) to (1.center);
		\draw [style=c0 arrow] (6) to (5.center);
		\draw [style=c0] (5.center) to (0);
		\draw [style=c0] (1.center) to (2);
		\draw [style=c1] (8) to (0);
		\draw [style=c1] (0) to (9);
		\draw [style=c0 arrow] (33) to (34.center);
		\draw [style=c0 arrow] (37) to (36.center);
		\draw [style=c0] (36.center) to (33);
		\draw [style=c0] (34.center) to (35);
		\draw [style=c1] (38) to (33);
		\draw [style=c1] (33) to (39);
		\draw [style=c1] (63) to (37);
		\draw [style=c1] (37) to (64);
		\draw [style=d1] (17.center) to (67.center);
		\draw [style=d1] (16.center) to (10.center);
		\draw [style=d1] (14.center) to (10.center);
		\draw [style=d1] (15.center) to (12.center);
		\draw [style=d1] (73.center) to (76.center);
		\draw [style=d1] (46.center) to (43.center);
		\draw [style=d1] (74.center) to (71.center);
		\draw [style=d1] (49.center) to (40.center);
		\draw [style=d1] (44.center) to (43.center);
		\draw [style=d1] (72.center) to (70.center);
		\draw [style=d1] (50.center) to (41.center);
		\draw [style=d1] (75.center) to (69.center);
		\draw [style=d1] (41.center) to (40.center);
	\end{pgfonlayer}
\end{tikzpicture}
\caption{The diagram for the proof of \autoref{lem:crisscross}. To produce the sides $C$ and $D$, we left-multiply the arrangement by $h\i$
\label{fig:crisscross}}
\end{figure}

\begin{lem}[Crisscross Lemma]\label{lem:crisscross}
    Assume \autoref{set} with $r\ge 3\cdot2^n$. 
    Suppose that $h$ is not an involution and let $g\in S\sm\{h\pmo\}$. 
    Let $\{A,B\}$ be an $r$-local $2$-separation of~$G$ with separator $\X$ such that $g$ traverses $\{A,B\}$ at~$\unit$.
    Then $g$ traverses the $r$-local $2$-separation $\{h\i A,h\i B\}$ at~$\unit$.
\end{lem}

\begin{proof}
We abbreviate $C:=h\i A$ and $D:=h\i B$.
     
\begin{sublem}\label{xlem}
    The vertex $h$ lies on the same side of $\{C,D\}$ as one of the vertices~$g^{\pm 1}$.
\end{sublem}
\begin{cproof}
    We assume without loss of generality that $h\i\in B\sm A$ and $h\in C\sm D$ (all other cases follow by an analogue argument).
    Then $A\cap D=\{\unit\}$ by \autoref{empty_corner}, see \autoref{fig:crisscross}.
    By the choice of $\{A,B\}$, one of $g$ or $g\i$ lies in $A\sm B$.
    Since $A\cap D=\{\unit\}$, one of $g$ or $g\i$ lies in~$C\sm D$.
\end{cproof}\medskip

    Now we consider the iterated commutator $m_n:=[g^2,h]_n=\unit$. Note that since $g$ traverses, $g^2\nequiv\unit$ and $g^2\nequiv h$. 
    Using the \nameref{commutatorContainsMorpheme} (\ref{commutatorContainsMorpheme}), we find a morpheme~$m$ of length at least three that is a cyclic subword of~$m_n$. 
    So $m$ contains $g^2$ as a cyclic subword.
    Hence we find a cycle $O$ in~$G$, labelled by the morpheme~$m$, such that~$O$ includes the path on the vertices $h\i g\i$, $h\i$ and~$h\i g$ (in this order). The $h\i$-images of $g^{\pm 1}$ are $h\i g\i$ and $h\i g$, and as such lie on opposite sides of~$\{C,D\}$. So the cycle $O$ traverses $\{C,D\}$ at $h\i$. 
    As $O$ has length~$|m|\le r$ by \autoref{obs:commlen3}, it follows by \autoref{CycleWeakTraversal} that $O$ traverses $\{C,D\}$ at~$\unit$ as well. 
    Thus two of the vertices $g$, $g\i$ and $h$ are on opposite sides of $\{C,D\}$. 
    By \autoref{xlem}, $g$ traverses $\{C,D\}$ at~$\unit$.
\end{proof}

\begin{dfn}[$h$-chord]
    Given a group $\Gamma$, a Cayley graph $G$ of $\Gamma$, and an element $h\in \Gamma$, an $h$-\emph{chord} of a cycle $O$ of $G$ is a chord of $O$ labelled by~$h$. 
    An $h$-\emph{chord} of a morpheme $m$ is an $h$-chord of a cycle in~$G$ labelled by~$m$.
\end{dfn}

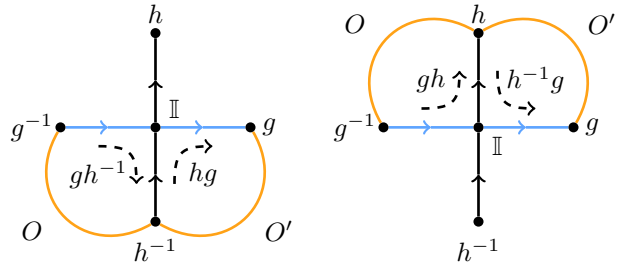
\begin{figure}[ht]
    \centering
\begin{tikzpicture}
	\begin{pgfonlayer}{nodelayer}
		\node [style=black dot] (0) at (0, 0) {};
		\node [style=black dot] (1) at (0, 2.5) {};
		\node [style=black dot] (2) at (2.5, 0) {};
		\node [style=black dot] (3) at (-2.5, 0) {};
		\node [style=black dot] (4) at (0, -2.5) {};
		\node [style=none] (5) at (0, 1.25) {};
		\node [style=none] (6) at (1.25, 0) {};
		\node [style=none] (7) at (0, -1.25) {};
		\node [style=none] (8) at (-1.25, 0) {};
		\node [style=none] (9) at (0, 3) {$h$};
		\node [style=none] (10) at (0, -3.25) {$h\i$};
		\node [style=none] (11) at (3, 0) {$g$};
		\node [style=none] (12) at (-3.25, 0) {$g\i$};
		\node [style=none] (13) at (0.5, -1.5) {};
		\node [style=none] (14) at (1.5, -0.5) {};
		\node [style=none] (15) at (-1.5, -0.5) {};
		\node [style=none] (16) at (-0.5, -1.5) {};
		\node [style=none] (17) at (1.5, -0.5) {};
		\node [style=none] (18) at (0.5, 0.5) {$\unit$};
		\node [style=none] (19) at (1.25, -1.25) {$hg$};
		\node [style=none] (20) at (-1.5, -1.25) {$gh\i$};
		\node [style=black dot] (21) at (8.5, 0) {};
		\node [style=black dot] (22) at (8.5, -2.5) {};
		\node [style=black dot] (23) at (11, 0) {};
		\node [style=black dot] (24) at (6, 0) {};
		\node [style=black dot] (25) at (8.5, 2.5) {};
		\node [style=none] (27) at (9.75, 0) {};
		\node [style=none] (29) at (7.25, 0) {};
		\node [style=none] (30) at (8.5, 3) {$h$};
		\node [style=none] (31) at (8.5, -3.25) {$h\i$};
		\node [style=none] (32) at (11.5, 0) {$g$};
		\node [style=none] (33) at (5.25, 0) {$g\i$};
		\node [style=none] (34) at (9, 1.5) {};
		\node [style=none] (36) at (7, 0.5) {};
		\node [style=none] (37) at (8, 1.5) {};
		\node [style=none] (38) at (10, 0.5) {};
		\node [style=none] (39) at (9, -0.5) {$\unit$};
		\node [style=none] (40) at (10, 1.25) {$h\i g$};
		\node [style=none] (41) at (7.25, 1.25) {$gh$};
		\node [style=none] (42) at (8.5, 1.25) {};
		\node [style=none] (43) at (8.5, -1.25) {};
		\node [style=none] (44) at (-3.25, -2.75) {$O$};
		\node [style=none] (45) at (3.25, -2.75) {$O'$};
		\node [style=none] (46) at (5.25, 2.75) {$O$};
		\node [style=none] (47) at (11.75, 2.75) {$O'$};
	\end{pgfonlayer}
	\begin{pgfonlayer}{edgelayer}
		\draw [style=c0 arrow] (0) to (5.center);
		\draw [style=c0 arrow] (4) to (7.center);
		\draw [style=c0] (5.center) to (1);
		\draw [style=c0] (7.center) to (0);
		\draw [style=c1] (6.center) to (2);
		\draw [style=c1] (8.center) to (0);
		\draw [style=c1 arrow] (0) to (6.center);
		\draw [style=c1 arrow] (3) to (8.center);
		\draw [style=c2, bend left=75, looseness=1.50] (2) to (4);
		\draw [style=c2, bend left=75, looseness=1.50] (4) to (3);
		\draw [style=d1 arrow, bend left=45, looseness=1.50] (13.center) to (17.center);
		\draw [style=d1 arrow, bend left=45, looseness=1.50] (15.center) to (16.center);
		\draw [style=c1] (27.center) to (23);
		\draw [style=c1] (29.center) to (21);
		\draw [style=c1 arrow] (21) to (27.center);
		\draw [style=c1 arrow] (24) to (29.center);
		\draw [style=c2, bend right=75, looseness=1.50] (23) to (25);
		\draw [style=c2, bend left=285, looseness=1.50] (25) to (24);
		\draw [style=d1 arrow, bend right=45, looseness=1.50] (34.center) to (38.center);
		\draw [style=d1 arrow, bend right=45, looseness=1.50] (36.center) to (37.center);
		\draw [style=c0] (42.center) to (25);
		\draw [style=c0] (21) to (43.center);
		\draw [style=c0 arrow] (21) to (42.center);
		\draw [style=c0 arrow] (22) to (43.center);
	\end{pgfonlayer}
\end{tikzpicture}
\caption{Case~1 (left) and case~2 (right) in the proof of \autoref{lem:twin}}
\label{fig:twin1}
\end{figure}

\begin{lem}\label{lem:twin}
Assume \autoref{set} with $r\ge 3\cdot 2^n$.
Suppose that $h$ is not an involution. 
If~some $g\in S\sm\{h\pmo\}$ traverses $\X$ at $\unit$ and there is a magic morpheme $m$ of length at most $r$, then $m$ has an $h$-chord.
\end{lem}

\begin{proof}
The magic property of $m$ implies that $m$ either contains both of $gh\i$ and $hg$ or both of $gh$ and $h\i g$ as cyclic subwords.
So it suffices to consider the following two cases.

\textbf{Case~1:} $gh\i$ and $hg$ are cyclic subwords of~$m$.
For a picture of the following, see \autoref{fig:twin1}.
The path on the vertices $g\i$, $\unit$ and~$h\i$ is labelled by~$gh\i$; let $O$ be a cycle in~$G$ that is labelled by~$m$ and contains this path.
Similarly, the path on vertices $h\i$, $\unit$ and $g$ is labelled by~$hg$; let $O'$ be a cycle in~$G$ that is labelled by~$m$ and contains this path.
Then $O$ and $O'$ share the edge that joins $h\i$ to~$\unit$.
Since $g$ traverses the $r$-local $2$-separator $\X$, one of the vertices $g$ and $g\i$ lies in a different local component at $\X$ than~$h\i$.
Hence at least one of the cycles $O$ or~$O'$ weakly traverses $\X$ at~$\unit$; let~$O''$ be such a cycle.
Since $O''$ has length at most~$r$, \autoref{CycleWeakTraversal} implies that $O''$ weakly traverses~$\X$ also at~$h$.
Therefore, the edge joining $\unit$ to~$h$ is an $h$-chord of~$O''$, and in particular of~$m$.

\textbf{Case~2:} $gh$ and $h\i g$ are cyclic subwords of~$m$.
Since $g$ traverses $\X$ at~$\unit$, \autoref{lem:crisscross} yields that $g$ traverses $\{\unit, h\i\}$ at $\unit$ as well. 
This allows us to argue similarly as in the first case, but with $\X$ replaced by~$\{\unit,h\i\}$, see the right-hand side of \autoref{fig:twin1}.
\end{proof}

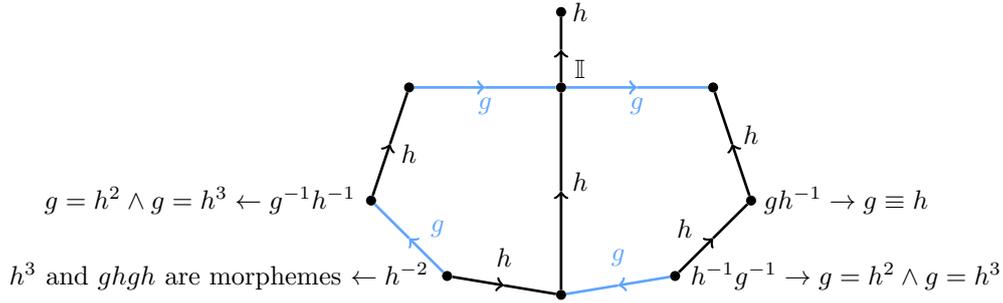
\begin{figure}[ht]
    \centering
\begin{tikzpicture}
	\begin{pgfonlayer}{nodelayer}
		\node [style=black dot] (0) at (0, 0) {};
		\node [style=black dot] (1) at (0, 2) {};
		\node [style=black dot] (2) at (4, 0) {};
		\node [style=black dot] (3) at (0, -5.5) {};
		\node [style=black dot] (4) at (-4, 0) {};
		\node [style=black dot] (5) at (3, -5) {};
		\node [style=black dot] (6) at (5, -3) {};
		\node [style=black dot] (7) at (-3, -5) {};
		\node [style=black dot] (8) at (-5, -3) {};
		\node [style=none] (9) at (2, 0) {};
		\node [style=none] (10) at (-2, 0) {};
		\node [style=none] (11) at (0, 1) {};
		\node [style=none] (12) at (0, -2.75) {};
		\node [style=none] (13) at (-4.5, -1.5) {};
		\node [style=none] (14) at (4.5, -1.5) {};
		\node [style=none] (15) at (4, -4) {};
		\node [style=none] (16) at (-4, -4) {};
		\node [style=none] (17) at (-1.5, -5.25) {};
		\node [style=none] (18) at (1.5, -5.25) {};
		\node [style=none] (19) at (0.5, 0.5) {$\unit$};
		\node [style=none] (20) at (0.5, 2) {$h$};
		\node [style=none] (23) at (7.5, -5) {$h\i g\i\to g=h^2 \land g=h^3$};
		\node [style=none] (24) at (-9, -5) {$h^3$ and $ghgh$ are morphemes $\leftarrow h^{-2}$};
		\node [style=none] (25) at (-9.5, -3) {$g=h^2 \land g=h^3\leftarrow g\i h\i$};
		\node [style=none] (26) at (7.5, -3) {$gh\i\to g\equiv h$};
		\node [style=none] (28) at (2, -0.5) {$\colorOne{g}$};
		\node [style=none] (29) at (-2, -0.5) {$\colorOne{g}$};
		\node [style=none] (30) at (1.5, -4.5) {$\colorOne{g}$};
		\node [style=none] (31) at (-3.25, -3.75) {$\colorOne{g}$};
		\node [style=none] (32) at (0.5, -2.5) {$h$};
		\node [style=none] (33) at (5, -1.25) {$h$};
		\node [style=none] (34) at (-4, -1.75) {$h$};
		\node [style=none] (35) at (3.25, -3.75) {$h$};
		\node [style=none] (36) at (-1.5, -4.5) {$h$};
	\end{pgfonlayer}
	\begin{pgfonlayer}{edgelayer}
		\draw [style=c0 arrow] (0) to (11.center);
		\draw [style=c0 arrow] (3) to (12.center);
		\draw [style=c1 arrow] (0) to (9.center);
		\draw [style=c1 arrow] (4) to (10.center);
		\draw [style=c1 arrow] (7) to (16.center);
		\draw [style=c1 arrow] (5) to (18.center);
		\draw [style=c0 arrow] (8) to (13.center);
		\draw [style=c0 arrow] (5) to (15.center);
		\draw [style=c0 arrow] (6) to (14.center);
		\draw [style=c0] (2) to (14.center);
		\draw [style=c0] (6) to (15.center);
		\draw [style=c0] (17.center) to (3);
		\draw [style=c0 arrow] (7) to (17.center);
		\draw [style=c0] (13.center) to (4);
		\draw [style=c0] (11.center) to (1);
		\draw [style=c0] (12.center) to (0);
		\draw [style=c1] (10.center) to (0);
		\draw [style=c1] (9.center) to (2);
		\draw [style=c1] (3) to (18.center);
		\draw [style=c1] (16.center) to (8);
	\end{pgfonlayer}
\end{tikzpicture}
\caption{A depiction for \autoref{cor:twin}. The implications show the result acquired when the vertex $h$ coincides with the labelled vertex}
\label{fig:twin}
\end{figure}

\begin{cor}\label{cor:twin}
    Assume \autoref{set} with~$r\ge 3\cdot 2^n$. Suppose that $h$ is not an involution, that some $g\in S\sm\{h\pmo\}$ traverses $\X$ at~$\unit$, and that $g\i h^2g\i h\i$ is a morpheme. 
    Then $ghgh$ is a morpheme or~$g\equiv h^2$.
\end{cor}
\begin{proof}
    Since the morpheme $m:=g\i h^2g\i h\i$ is magic, it has an $h$-chord by \autoref{lem:twin}.
    Let $O$ be a cycle in~$G$ labelled by~$m$, let $e$ be a chord of $O$ labelled by~$h$, and let $P,Q$ be the two edge-disjoint paths in~$O$ linking the ends of~$e$.
    Note that $P$ and $Q$ have length at least two.
    Since $O$ has length five, this means that $P$ has length two and $Q$ has length three, say.
    If $P$ is labelled by~$h^2$ in some direction, then with the $h$-chord we infer that $h^3=\unit$, and from $Q$ plus the $h$-chord we infer that $ghgh$ is a morpheme.
    Otherwise $P$ is labelled $gh$ or $hg$ in some direction, and with the $h$-chord we infer $g\equiv h^2$. For a depiction, see~\autoref{fig:twin}.
\end{proof}

\begin{lem}\label{lem:pent1}
    Assume \autoref{set}. Let $g\in S\sm\{h\pmo\}$ and suppose that $g$ traverses $\X$ at~$\unit$. Then no morpheme in $g\pmo$ and $h\pmo$ is of the form $g^2h^k$  with $|k|\le r-2$.
\end{lem}
\begin{proof}
    Suppose for a contradiction that $m=g^2 h^k$ is a morpheme for some $k$ with $|k|\le r-2$.
    Since $g$ traverses, we have $g^2\neq\unit$, so~$k\neq 0$.
    Let $O$ be a cycle in $G$ that is labelled by~$m$, and that contains the vertex~$\unit$ so that $\unit$ is incident with the two edges of~$O$ that are labelled by~$g$. 
    Since $g$ traverses~$\X$ at~$\unit$, the cycle $O$ weakly traverses~$\X$ at~$\unit$. 
    Using that $O$ has length $|m|\le r$, \autoref{CycleWeakTraversal} implies that $O$ weakly traverses~$\X$ at~$h$ as well.
    But then the vertex~$h$ is incident to two $h$-labelled edges in~$O$, both of which run between the vertex~$h$ and distinct local components at~$X$, and the vertex~$h$ additionally is incident to the $h$-labelled edge that joins~$\unit$ to~$h$.
    So the vertex $h$ is incident with three $h$-labelled edges in~$G$, a contradiction.
\end{proof}

\begin{lem}\label{short}
    Assume \autoref{set} with~$r\ge3\cdot2^n$.
    Let $g\in S\sm\{h\pmo\}$ and suppose that $w$ is a morpheme that is a cyclic subword of $w_n=[g,h^2]_n$. 
    If $w$ has at most three segments, then $g\equiv h^k$ for some $k$ with $2\le k\le 4$ or $h$ is an involution.
\end{lem}
\begin{proof}
    \textbf{Case~1:} $w$ has exactly one segment.
    Since $g\neq\unit$, $w$ must be an $h$-segment. Since $h\neq\unit$, we get that $h$ must be an involution.
    
    \textbf{Case~2:} $w$ has exactly two segments.
    Then $g\equiv h^k$ for some~$k\in\{1,2\}$, and in fact $k=2$ since $g\nequiv h$.
    
    \textbf{Case~3:} $w$ has exactly three segments.
    If $w$ is of the form $g\pmo h^i g\pmo$, then we can cyclically permute it and obtain a contradiction to \autoref{lem:pent1}, so $w$ cannot have this form. 
    If $w$ is of the form $g\pmo h^i g^{\mp 1}$, then $h$ is an involution (because $h$-segments of~$w$ have length at most two, so~$1\le |i|\le 2$, and $h\neq\unit$ gives $|i|=2$). 
    Otherwise, $w=h^i g\pmo h^j$ for some $i,j$ with $|i|,|j|\in\{1,2\}$. 
    Then cyclically permuting yields $g=h^{i+j}$, and $k:=|i+j|$ satisfies $2\le k\le 4$ (since $k\in\{0,1\}$ would contradict $g\nequiv \unit,h$).
\end{proof}

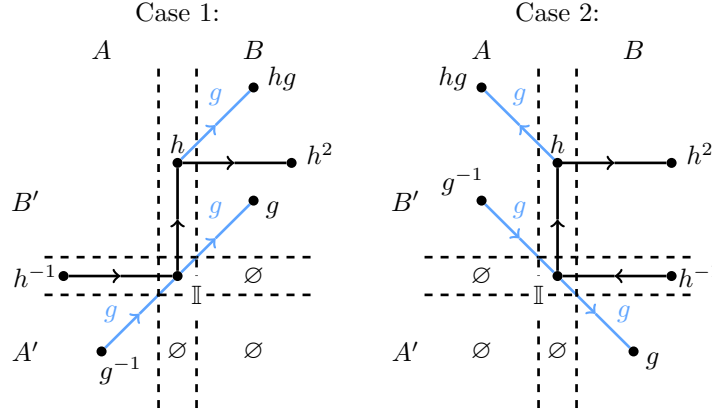
\begin{figure}[ht]
    \centering
\begin{tikzpicture}
	\begin{pgfonlayer}{nodelayer}
		\node [style=black dot] (0) at (9.5, -10.5) {};
		\node [style=none] (1) at (9.5, -9) {};
		\node [style=black dot] (2) at (9.5, -7.5) {};
		\node [style=none] (5) at (11, -10.5) {};
		\node [style=black dot] (6) at (12.5, -10.5) {};
		\node [style=none] (10) at (10, -10) {};
		\node [style=none] (11) at (9, -10) {};
		\node [style=none] (13) at (10, -11) {};
		\node [style=none] (14) at (10, -5) {};
		\node [style=none] (16) at (10, -14) {};
		\node [style=none] (19) at (13, -10) {};
		\node [style=none] (20) at (6, -10) {};
		\node [style=none] (22) at (9.5, -7) {$h$};
		\node [style=none] (23) at (13.25, -10.5) {$h\i$};
		\node [style=none] (30) at (7.5, -10.5) {$\varnothing$};
		\node [style=none] (31) at (7.5, -12.5) {$\varnothing$};
		\node [style=none] (32) at (9.5, -12.5) {$\varnothing$};
		\node [style=none] (56) at (7.5, -4.5) {$A$};
		\node [style=none] (57) at (11.5, -4.5) {$B$};
		\node [style=none] (82) at (-2.5, -4.5) {$A$};
		\node [style=none] (83) at (1.5, -4.5) {$B$};
		\node [style=none] (89) at (9.5, -3.5) {Case 2:};
		\node [style=none] (90) at (-0.5, -3.5) {Case 1:};
		\node [style=black dot] (93) at (12.5, -7.5) {};
		\node [style=black dot] (95) at (-0.5, -10.5) {};
		\node [style=none] (96) at (-0.5, -9) {};
		\node [style=black dot] (97) at (-0.5, -7.5) {};
		\node [style=none] (98) at (-2, -10.5) {};
		\node [style=black dot] (99) at (-3.5, -10.5) {};
		\node [style=none] (100) at (0, -10) {};
		\node [style=none] (101) at (-1, -10) {};
		\node [style=none] (102) at (-1, -11) {};
		\node [style=none] (105) at (-1, -5) {};
		\node [style=none] (107) at (-1, -14) {};
		\node [style=none] (109) at (3, -10) {};
		\node [style=none] (110) at (-4, -10) {};
		\node [style=none] (111) at (-4, -11) {};
		\node [style=none] (112) at (-0.5, -7) {$h$};
		\node [style=none] (113) at (-4.25, -10.5) {$h\i$};
		\node [style=none] (118) at (1.5, -12.5) {$\varnothing$};
		\node [style=none] (119) at (1.5, -10.5) {$\varnothing$};
		\node [style=none] (120) at (-0.5, -12.5) {$\varnothing$};
		\node [style=black dot] (121) at (2.5, -7.5) {};
		\node [style=none] (124) at (-4.5, -8.5) {$B'$};
		\node [style=none] (125) at (-4.5, -12.5) {$A'$};
		\node [style=none] (126) at (5.5, -8.5) {$B'$};
		\node [style=none] (127) at (5.5, -12.5) {$A'$};
		\node [style=none] (130) at (11, -7.5) {};
		\node [style=none] (133) at (1, -7.5) {};
		\node [style=none] (135) at (9.5, -7.5) {};
		\node [style=none] (140) at (8.5, -9.5) {};
		\node [style=none] (141) at (10.5, -11.5) {};
		\node [style=none] (142) at (-1.5, -11.5) {};
		\node [style=none] (143) at (0.5, -9.5) {};
		\node [style=black dot] (148) at (1.5, -8.5) {};
		\node [style=black dot] (149) at (-2.5, -12.5) {};
		\node [style=black dot] (150) at (7.5, -8.5) {};
		\node [style=black dot] (151) at (11.5, -12.5) {};
		\node [style=none] (154) at (13.25, -7.25) {$h^2$};
		\node [style=none] (155) at (3.25, -7.25) {$h^2$};
		\node [style=none] (158) at (8.5, -8.75) {$\colorOne{g}$};
		\node [style=none] (159) at (11.25, -11.5) {$\colorOne{g}$};
		\node [style=none] (160) at (-2.25, -11.5) {$\colorOne{g}$};
		\node [style=none] (161) at (0.5, -8.75) {$\colorOne{g}$};
		\node [style=none] (184) at (9, -11) {$\unit$};
		\node [style=none] (185) at (9, -10.5) {};
		\node [style=none] (186) at (9.25, -11) {};
		\node [style=none] (187) at (9, -5) {};
		\node [style=none] (188) at (9, -14) {};
		\node [style=none] (189) at (13, -11) {};
		\node [style=none] (190) at (6, -11) {};
		\node [style=none] (191) at (9, -11.5) {};
		\node [style=none] (192) at (8.75, -11) {};
		\node [style=none] (193) at (0, -11) {$\unit$};
		\node [style=none] (194) at (0, -14) {};
		\node [style=none] (195) at (0, -11.5) {};
		\node [style=none] (196) at (0, -10.5) {};
		\node [style=none] (197) at (0, -5) {};
		\node [style=none] (198) at (3, -11) {};
		\node [style=none] (199) at (0.25, -11) {};
		\node [style=none] (200) at (-4, -11) {};
		\node [style=none] (201) at (-0.25, -11) {};
		\node [style=none] (211) at (0.5, -6.5) {};
		\node [style=black dot] (212) at (1.5, -5.5) {};
		\node [style=none] (213) at (2, -8.75) {$g$};
		\node [style=none] (214) at (-2, -13) {$g\i$};
		\node [style=none] (218) at (12, -12.75) {$g$};
		\node [style=none] (219) at (7, -8) {$g\i$};
		\node [style=none] (221) at (2.25, -5.25) {$hg$};
		\node [style=none] (222) at (8.5, -6.5) {};
		\node [style=black dot] (223) at (7.5, -5.5) {};
		\node [style=none] (224) at (6.75, -5.25) {$hg$};
		\node [style=none] (225) at (8.5, -5.75) {$\colorOne{g}$};
		\node [style=none] (226) at (0.5, -5.75) {$\colorOne{g}$};
	\end{pgfonlayer}
	\begin{pgfonlayer}{edgelayer}
		\draw [style=c0 arrow] (0) to (1.center);
		\draw [style=c0 arrow] (6) to (5.center);
		\draw [style=c0] (5.center) to (0);
		\draw [style=c0] (1.center) to (2);
		\draw [style=c0 arrow] (95) to (96.center);
		\draw [style=c0 arrow] (99) to (98.center);
		\draw [style=c0] (98.center) to (95);
		\draw [style=c0] (96.center) to (97);
		\draw [style=c0 arrow] (135.center) to (130.center);
		\draw [style=c0 arrow] (97) to (133.center);
		\draw [style=c0] (133.center) to (121);
		\draw [style=c0] (93) to (130.center);
		\draw [style=c1 arrow] (149) to (142.center);
		\draw [style=c1 arrow] (95) to (143.center);
		\draw [style=c1 arrow] (150) to (140.center);
		\draw [style=c1 arrow] (0) to (141.center);
		\draw [style=c1] (142.center) to (95);
		\draw [style=c1] (143.center) to (148);
		\draw [style=c1] (140.center) to (0);
		\draw [style=c1] (141.center) to (151);
		\draw [style=d1] (188.center) to (191.center);
		\draw [style=d1] (16.center) to (13.center);
		\draw [style=d1] (107.center) to (102.center);
		\draw [style=d1] (194.center) to (195.center);
		\draw [style=d1] (109.center) to (100.center);
		\draw [style=d1] (198.center) to (199.center);
		\draw [style=d1] (189.center) to (13.center);
		\draw [style=d1] (19.center) to (10.center);
		\draw [style=d1] (190.center) to (192.center);
		\draw [style=d1] (20.center) to (11.center);
		\draw [style=d1] (200.center) to (102.center);
		\draw [style=d1] (110.center) to (101.center);
		\draw [style=d1] (105.center) to (101.center);
		\draw [style=d1] (14.center) to (10.center);
		\draw [style=d1] (197.center) to (196.center);
		\draw [style=d1] (187.center) to (185.center);
		\draw [style=d1] (101.center) to (100.center);
		\draw [style=d1] (101.center) to (102.center);
		\draw [style=d1] (102.center) to (201.center);
		\draw [style=d1] (11.center) to (10.center);
		\draw [style=d1] (10.center) to (13.center);
		\draw [style=d1] (186.center) to (13.center);
		\draw [style=c1 arrow] (97) to (211.center);
		\draw [style=c1] (211.center) to (212);
		\draw [style=c1 arrow] (135.center) to (222.center);
		\draw [style=c1] (222.center) to (223);
	\end{pgfonlayer}
\end{tikzpicture}
\caption{The diagrams for the cases in \autoref{lem:seagull}}
\label{fig:seagull}
\end{figure}

\begin{lem}[Seagull Lemma]\label{lem:seagull}
    Assume \autoref{set} with $r\ge4$.
    Let $g\in S\sm\{h\pmo\}$ and suppose that $g$~traverses $\X$ at $\unit$ and that $gh^2$ traverses $\X$ strongly. Then $h^2g$ traverses $\X$ strongly.
\end{lem}

\begin{proof}
Let $\{A,B\}$ be the $r$-local $2$-separation whose separator is $\X$ such that $A\sm B$ is the local component containing $g\i$.  
Since $g$ traverses~$\X$, and since $gh^2$ traverses $\X$ strongly, we deduce that 
both vertices $g$ and $h^2$ are in different local components to $g^{-1}$ and thus lie in~$B\sm A$.
The automorphism of~$G$ defined by left-multiplication by~$h\i$ takes $\{A,B\}$ to an $r$-local $2$-separation $\{A',B'\}$ with separator~$\{\unit,h\i\}$, where $A':=h\i A$ and $B':=h\i B$.
Since $h^2\in B\sm A$, applying the automorphism defined by left-multiplication by $h\i$ yields~$h\in B'\sm A'$.
We consider two cases, see \autoref{fig:seagull}, which we finish in the same way.

\textbf{Case~1:} $h\i\in A\sm B$. We are to show that $hg\in B\sm A$. 
By \autoref{empty_corner}, we have that $B\cap A'=\{\unit\}$.
Since $g\i$ lies in $A\sm B$ and $B\cap A'=\{\unit\}$, the vertex $g\i$ must lie in~$A\cap A'$.
Since $g$ traverses $\{A',B'\}$ at~$\unit$ by the \nameref{lem:crisscross} (\ref{lem:crisscross}), we deduce that $g\in B'\sm A'$. Applying the automorphism $h$ to the $2$-separation $\{A',B'\}$, yields that $hg\in B\sm A$, completing this case. 

\textbf{Case~2:} not Case 1. As $h$ is in $A\cap B$ and $h$ is not an involution, we get $h\i\in B\sm A$. We are to show that $hg\in A\sm B$.
By \autoref{empty_corner}, we have that $A\cap A'=\{\unit\}$.
Since $g\i$ lies in $A$ and $B\cap A'=\{\unit\}$, the vertex $g\i$ must lie in~$A\cap B'$.
Since $g$ traverses $\{A',B'\}$ at~$\unit$ by the \nameref{lem:crisscross} (\ref{lem:crisscross}), we deduce that $g\in A'\sm B'$. 
 Applying the automorphism $h$ to the $2$-separation $\{A',B'\}$, yields that $hg\in A\sm B$, completing this case. 
\end{proof}

\begin{proof}[Proof of \autoref{subsec}]
    Assume \autoref{set} with $r\ge 3\cdot 2^n$.
    We have to show, for every $g\in S$ which traverses $\X$ at $\unit$ so that $gh^2$ traverses $\X$ strongly, that we either have $g\equiv h^k$ for some~$k$ with $2\le k\le 4$, or else that $g^{-1}h^2g^{-1}h^{-1}$ and $ghgh$ are morphemes.
    Since $g$ traverses~$\X$ at~$\unit$, we have~$g\nequiv h$.
    Using the \nameref{commutatorContainsMorpheme} (\ref{commutatorContainsMorpheme}), we find a morpheme~$w$ as a cyclic subword of~$w_n=[g,h^2]_n$.
    Note that $w$ has length at most~$r$ by \autoref{obs:commlen3}.
    
    \textbf{Case~1:} $w$ has at most three segments.
    Since $gh^2$ traverses strongly, $h$ is not an involution, so \autoref{short} gives that $g\equiv h^k$ for some~$k$ with~$2\le k\le 4$.
    
    \textbf{Case~2:} $w$ has at least four segments. The word $gh^2$ traverses $\X$ strongly by assumption, and $h^2g$ traverses $\X$ strongly by \autoref{lem:seagull}.
    Hence neither $gh^2$ nor $h^2 g$ occur as cyclic subwords in~$w$, by the \nameref{traverser} (\ref{traverser}).
    But then $w=g\i h^2g\i h\i$, say, by \autoref{Ramsey}.
    By \autoref{cor:twin}, either $ghgh$ is a morpheme as well, or $g\equiv h^k$ for~$k=2$, and we are done either way.
\end{proof}

\section{Involution case}
\label{sec:invo}

In this section we prove \autoref{lem:involution_case}, which states that assuming \autoref{set} with $r\ge \max\{2^{n+2},10\}$, if $h$ is an involution then the subgroup of~$\Gamma$ generated by~$h$ is normal.
When $h$ is an involution, we cannot use the previous technique of finding paths labelled by words like $gh^2$ in short cycles. We instead look for paths labelled by words like $ghk$ where the vertices $g\i$ and~$hk$ lie in distinct local components at the local $2$-separator~$\X$. This requires some preparation.

\begin{lem}
\label{lem:invo_implies_double_trav}
   Assume \autoref{set}. Suppose that $h$ is an involution.
    If some $g\in S$ traverses $\X$ at some vertex, then $g$ traverses $\X$ at both vertices $\unit$ and~$h$.
\end{lem}
\begin{proof}
The automorphism of~$G$ defined by left-muliplication with~$h$ maps $\X$ to itself.
\end{proof}

\begin{figure}[ht]
    \centering
\begin{tikzpicture}
	\begin{pgfonlayer}{nodelayer}
		\node [style=none] (0) at (0, 1.5) {};
		\node [style=none] (1) at (0, -1.5) {};
		\node [style=black dot] (2) at (0, 1.5) {};
		\node [style=black dot] (3) at (0, -1.5) {};
		\node [style=black dot] (4) at (2.5, 2) {};
		\node [style=black dot] (5) at (2.5, -2) {};
		\node [style=black dot] (6) at (-2.5, -2) {};
		\node [style=none] (8) at (1.25, 1.75) {};
		\node [style=none] (9) at (1.25, -1.75) {};
		\node [style=none] (10) at (-1.25, -1.75) {};
		\node [style=none] (13) at (0, -2.5) {};
		\node [style=none] (14) at (0, 2.5) {};
		\node [style=none] (15) at (2.5, 1.5) {};
		\node [style=none] (16) at (2.5, 1.5) {};
		\node [style=none] (17) at (2.5, 2.5) {};
		\node [style=none] (18) at (-2.5, -1.5) {};
		\node [style=none] (19) at (-2.5, -2.5) {};
		\node [style=none] (20) at (-0.5, 0) {$h$};
		\node [style=none] (21) at (0.5, -1) {$\unit$};
		\node [style=none] (22) at (-2.5, -1) {$g\i$};
		\node [style=none] (23) at (2.5, 1) {$hg$};
		\node [style=none] (25) at (2.5, -1.5) {$g$};
		\node [style=none] (26) at (7, 1.5) {};
		\node [style=none] (27) at (7, -1.5) {};
		\node [style=black dot] (28) at (7, 1.5) {};
		\node [style=black dot] (29) at (7, -1.5) {};
		\node [style=black dot] (30) at (9.5, 2) {};
		\node [style=black dot] (31) at (9.5, -2) {};
		\node [style=black dot] (32) at (4.5, -2) {};
		\node [style=none] (33) at (8.25, 1.75) {};
		\node [style=none] (34) at (8.25, -1.75) {};
		\node [style=none] (35) at (5.75, -1.75) {};
		\node [style=none] (37) at (7, -2.5) {};
		\node [style=none] (38) at (7, 2.5) {};
		\node [style=none] (39) at (9.5, 1.5) {};
		\node [style=none] (40) at (9.5, 1.5) {};
		\node [style=none] (41) at (9.5, 2.5) {};
		\node [style=none] (42) at (4.5, -1.5) {};
		\node [style=none] (43) at (4.5, -2.5) {};
		\node [style=none] (44) at (6.5, 0) {$h$};
		\node [style=none] (45) at (7.5, -1) {$\unit$};
		\node [style=none] (46) at (4.5, -1) {$g\i$};
		\node [style=none] (47) at (9.5, 1) {$hg\i$};
		\node [style=none] (48) at (9.5, -1.5) {$g$};
		\node [style=none] (49) at (0, 3.5) {Symmetric:};
		\node [style=none] (50) at (7, 3.5) {Antisymmetric:};
		\node [style=none] (51) at (2.5, -0.25) {};
	\end{pgfonlayer}
	\begin{pgfonlayer}{edgelayer}
		\draw [style=d1] (2) to (14.center);
		\draw [style=d1] (3) to (13.center);
		\draw (3) to (2);
		\draw [style=c1 arrow] (2) to (8.center);
		\draw [style=c1 arrow] (3) to (9.center);
		\draw [style=c1 arrow] (6) to (10.center);
		\draw [style=c1] (8.center) to (4);
		\draw [style=c1] (9.center) to (5);
		\draw [style=c1] (10.center) to (3);
		\draw [style=d1] (19.center)
			 to [bend right=90, looseness=1.75] (18.center)
			 to [bend right=90, looseness=1.75] cycle;
		\draw [style=d1] (17.center)
			 to [bend left=90, looseness=1.75] (16.center)
			 to [bend left=90, looseness=1.75] cycle;
		\draw [style=d1] (28) to (38.center);
		\draw [style=d1] (29) to (37.center);
		\draw (29) to (28);
		\draw [style=c1 arrow] (29) to (34.center);
		\draw [style=c1 arrow] (32) to (35.center);
		\draw [style=c1] (34.center) to (31);
		\draw [style=c1] (35.center) to (29);
		\draw [style=d1] (43.center)
			 to [bend right=90, looseness=1.75] (42.center)
			 to [bend right=90, looseness=1.75] cycle;
		\draw [style=d1] (41.center)
			 to [bend left=90, looseness=1.75] (40.center)
			 to [bend left=90, looseness=1.75] cycle;
		\draw [style=c1 arrow] (30) to (33.center);
		\draw [style=c1] (33.center) to (28);
	\end{pgfonlayer}
\end{tikzpicture}
\caption{Symmetric and antisymmetric traversals}
\label{fig:trav_types}
\end{figure}
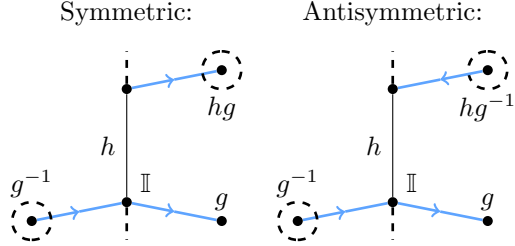

The following definition is supported by \autoref{fig:trav_types}.

\begin{dfn}[Symmetric and antisymmetric traversals]\label{symAndAntisym}
    Assume \autoref{set}. 
    Further suppose that some $g\in S$ traverses $\X$ and $h$ is an involution.
    Then $g$ traverses $\X$ at both $\unit$ and $h$ by \autoref{lem:invo_implies_double_trav}.
    We say that $g$ traverses $\X$ 
    \begin{enumerate}
        \item[\bf{--}]\emph{symmetrically} if the vertices $g\i$ and $hg$ lie in distinct local components at~$\X$, and
        \item[\bf{--}] \emph{antisymmetrically} if the vertices $g\i$ and $hg\i$ lie in distinct local components at~$\X$.
    \end{enumerate}
\end{dfn}

Note that it is possible for an element $g\in S$ to traverse $\X$ both symmetrically and antisymmetrically; for example, if all four vertices $g\i$, $g$, $hg\i$ and~$hg$ lie in distinct local components at~$\X$.
It is important to note that, in this context, if an element of~$S$ traverses~$\X$, then it must do so either symmetrically or antisymmetrically (possibly both).

\begin{figure}[ht]
    \centering
\begin{tikzpicture}
	\begin{pgfonlayer}{nodelayer}
		\node [style=none] (0) at (0, 1.5) {};
		\node [style=none] (1) at (0, -1.5) {};
		\node [style=black dot] (2) at (0, 1.5) {};
		\node [style=black dot] (3) at (0, -1.5) {};
		\node [style=black dot] (4) at (2.5, 2) {};
		\node [style=black dot] (5) at (2.5, -2) {};
		\node [style=black dot] (6) at (-2.5, -2) {};
		\node [style=none] (8) at (1.25, 1.75) {};
		\node [style=none] (9) at (1.25, -1.75) {};
		\node [style=none] (10) at (-1.25, -1.75) {};
		\node [style=none] (20) at (-0.5, 0) {$h$};
		\node [style=none] (21) at (0.5, -1) {$\unit$};
		\node [style=none] (22) at (-3.75, -2) {$g\i$};
		\node [style=none] (23) at (3.75, 2) {$hg$};
		\node [style=none] (25) at (3.5, -2) {$g$};
		\node [style=none] (26) at (9.5, 1.5) {};
		\node [style=none] (27) at (9.5, -1.5) {};
		\node [style=black dot] (28) at (9.5, 1.5) {};
		\node [style=black dot] (29) at (9.5, -1.5) {};
		\node [style=black dot] (30) at (12, 2) {};
		\node [style=black dot] (31) at (12, -2) {};
		\node [style=black dot] (32) at (7, -2) {};
		\node [style=none] (33) at (10.75, 1.75) {};
		\node [style=none] (34) at (10.75, -1.75) {};
		\node [style=none] (35) at (8.25, -1.75) {};
		\node [style=none] (40) at (12, -2.5) {};
		\node [style=none] (41) at (12, 2.5) {};
		\node [style=none] (42) at (7, 2.5) {};
		\node [style=none] (43) at (7, -2.5) {};
		\node [style=none] (44) at (9, 0) {$h$};
		\node [style=none] (45) at (10, -1) {$\unit$};
		\node [style=none] (46) at (5.75, -2) {$g\i$};
		\node [style=none] (47) at (13.5, 2) {$hg\i$};
		\node [style=none] (48) at (13, -2) {$g$};
		\node [style=none] (49) at (0, 3.5) {Symmetric:};
		\node [style=none] (50) at (9.5, 3.75) {Antisymmetric};
		\node [style=none] (52) at (9.5, 3.25) {but not symmetric:};
		\node [style=black dot] (53) at (-2.5, 2) {};
		\node [style=black dot] (54) at (7, 2) {};
		\node [style=none] (55) at (-1.25, 1.75) {};
		\node [style=none] (56) at (8.25, 1.75) {};
		\node [style=none] (57) at (-4, 2) {$hg\i$};
		\node [style=none] (58) at (5.75, 2) {$hg$};
		\node [style=none] (59) at (-2.5, 0) {$A$};
		\node [style=none] (60) at (2.5, 0) {$B$};
		\node [style=none] (61) at (7, 0) {$C$};
		\node [style=none] (62) at (12, 0) {$D$};
		\node [style=none] (63) at (2.5, 2.5) {};
		\node [style=none] (64) at (2.5, -2.5) {};
		\node [style=none] (65) at (-2.5, 2.5) {};
		\node [style=none] (66) at (-2.5, -2.5) {};
	\end{pgfonlayer}
	\begin{pgfonlayer}{edgelayer}
		\draw [style=c1 arrow] (2) to (8.center);
		\draw [style=c1 arrow] (3) to (9.center);
		\draw [style=c1 arrow] (6) to (10.center);
		\draw [style=c1] (8.center) to (4);
		\draw [style=c1] (9.center) to (5);
		\draw [style=c1] (10.center) to (3);
		\draw [style=c1 arrow] (29) to (34.center);
		\draw [style=c1 arrow] (32) to (35.center);
		\draw [style=c1] (34.center) to (31);
		\draw [style=c1] (35.center) to (29);
		\draw [style=d1] (43.center)
			 to [bend right=90, looseness=0.50] (42.center)
			 to [bend right=90, looseness=0.50] cycle;
		\draw [style=d1] (41.center)
			 to [bend left=90, looseness=0.50] (40.center)
			 to [bend left=90, looseness=0.50] cycle;
		\draw [style=c1 arrow] (30) to (33.center);
		\draw [style=c1] (33.center) to (28);
		\draw [style=c1 arrow] (28) to (56.center);
		\draw [style=c1 arrow] (53) to (55.center);
		\draw [style=c1] (56.center) to (54);
		\draw [style=c1] (55.center) to (2);
		\draw [style=d1] (64.center)
			 to [bend right=90, looseness=0.50] (63.center)
			 to [bend right=90, looseness=0.50] cycle;
		\draw [style=d1, bend right=90, looseness=0.50] (66.center) to (65.center);
		\draw [style=d1, bend right=90, looseness=0.50] (65.center) to (66.center);
		\draw [style=c0] (2) to (3);
		\draw [style=c0] (28) to (29);
	\end{pgfonlayer}
\end{tikzpicture}
\caption{The two outcomes (i) (left) and (ii) (right) of \autoref{obs:what_traverse}. The blue edges are labelled by~$g$.}
\label{fig:trav_obs}
\end{figure}
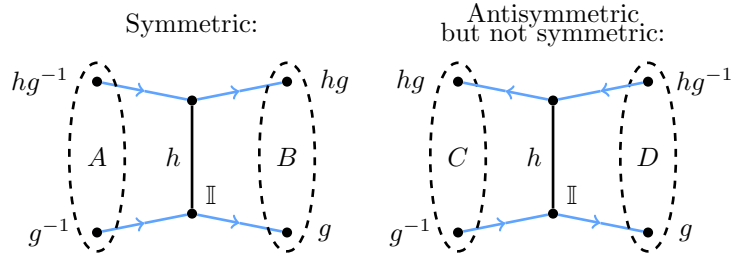

The following lemma is supported by \autoref{fig:trav_obs}.

\begin{lem}
\label{obs:what_traverse}
    Assume \autoref{set}. 
    Further suppose that some $g\in S$ traverses $\X$ and $h$ is an involution, so $g$ traverses $\X$ at both $\unit$ and $h$ by \autoref{lem:invo_implies_double_trav}.
    \begin{enumerate}[label=\textnormal{(\roman*)}]
        \item If $g$ traverses $\X$ symmetrically, then there is an $r$-local $2$-separation $\{A,B\}$ of~$G$ with separator $\X$ such that
        \[ g\i,hg\i\in A \quad\text{and}\quad g,hg\in B.\]
        \item If $g$ traverses $\X$ antisymmetrically but not symmetrically, then there are distinct $r$-local components $C$ and $D$ at $\X$ such that
        \[ g\i,hg\in C \quad\text{and}\quad g,hg\i\in D.\]
        In particular, both words $g\i hg$ and $ghg\i$ strongly traverse~$\X$.
    \end{enumerate}
\end{lem}
\begin{proof}
(i).
Let $C$ denote the local component at~$\X$ that contains the vertex~$g$, and let $D$ denote the local component at~$\X$ that contains the vertex~$hg$.
Since $g$ traverses~$\X$ at~$\unit$ and $g\in C$, the vertex $g\i$ is not in~$C$.
Since $g$ traverses $\X$ symmetrically and $hg\in D$, the vertex $g\i$ is not in~$D$.
In total, $g\i\notin C\cup D$.

Since left-multiplication by $h$ defines an automorphism of~$G$, and since this automorphism takes $\X$ to~$\X$, the two vertex sets $hC$ and $hD$ are $r$-local components at~$\X$.
From 
\begin{align*}
\begin{array}{lllll}
     & \phantom{h}g\in \phantom{h}C, & hg\in \phantom{h}D, & \text{and} & \phantom{h}g\i\notin \phantom{h}C\cup \phantom{h}D  \\
     \text{we obtain} & hg\in hC, & \phantom{h}g\in hD, & \text{and} & hg\i\notin hC\cup hD.
\end{array}
\end{align*}
Since $g$ traverses~$\X$ symmetrically and since $hg$ lies in~$hC$, we have that $g\i$ does not lie in~$hC$.
Since $g$ traverses $\X$ at~$\unit$ and since the vertex $g$ lies in~$hD$, we have that $g\i$ does not lie in~$hD$.
In total, $hg\i,g\i\notin hC\cup hD$ while $hg,g\in hC\cup hD$.
Hence taking $B:=\X\cup hC\cup hD$ and taking $A$ to be the union of~$\X$ with all $r$-local components at~$\X$ other than $hC$ and $hD$ yields an $r$-local $2$-separation $\{A,B\}$ of~$G$ with $g\i,hg\i\in A$ and $g,hg\in B$.

(ii).
Let $C$ denote the $r$-local component at~$\X$ that contains~$g\i$, and let $D$ denote the $r$-local component at~$\X$ that contains~$hg\i$.
Since $g$ traverses~$\X$ antisymmetrically, $C$ and $D$ are distinct.
Since $g$ does not traverse~$\X$ symmetrically, the vertex $hg$ lies in the same local component as~$g\i$, so in total $g\i,hg\in C$.
As in~(i), left-multiplication by~$h$ defines an automorphism of~$G$ that takes local components at~$\X$ to local components at~$\X$.
The automorphism takes the local component $C$ containing $g\i$ and $hg$ to the local component $hC$ containing $hg\i$ and~$g$.
Since $hg\i$ lies in both $hC$ and~$D$, it follows that $hC$ and $D$ coincide, so $g,hg\i\in hC=D$ which completes the proof.
\end{proof}

We now state three propositions without proof and show how we can derive \autoref{lem:involution_case} from them. 
We will prove the three propositions afterwards.

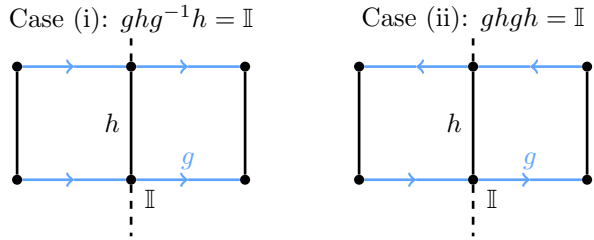
\begin{figure}[ht]
    \centering
\begin{tikzpicture}
	\begin{pgfonlayer}{nodelayer}
		\node [style=none] (0) at (0, 1.5) {};
		\node [style=none] (1) at (0, -1.5) {};
		\node [style=black dot] (2) at (0, 1.5) {};
		\node [style=black dot] (3) at (0, -1.5) {};
		\node [style=black dot] (4) at (3, 1.5) {};
		\node [style=black dot] (5) at (3, -1.5) {};
		\node [style=black dot] (6) at (-3, -1.5) {};
		\node [style=none] (8) at (1.5, 1.5) {};
		\node [style=none] (9) at (1.5, -1.5) {};
		\node [style=none] (10) at (-1.5, -1.5) {};
		\node [style=none] (13) at (0, -3) {};
		\node [style=none] (14) at (0, 2.25) {};
		\node [style=none] (20) at (-0.5, 0) {$h$};
		\node [style=none] (21) at (0.5, -2) {$\unit$};
		\node [style=none] (49) at (0, 2.75) {Case (i): $ghg\i h=\unit$};
		\node [style=black dot] (53) at (-3, 1.5) {};
		\node [style=none] (55) at (-1.5, 1.5) {};
		\node [style=none] (56) at (10.5, 1.5) {};
		\node [style=none] (57) at (9, -1.5) {};
		\node [style=black dot] (58) at (9, 1.5) {};
		\node [style=black dot] (59) at (9, -1.5) {};
		\node [style=black dot] (60) at (12, 1.5) {};
		\node [style=black dot] (61) at (12, -1.5) {};
		\node [style=black dot] (62) at (6, -1.5) {};
		\node [style=none] (64) at (10.5, -1.5) {};
		\node [style=none] (65) at (7.5, -1.5) {};
		\node [style=none] (66) at (9, -3) {};
		\node [style=none] (67) at (9, 2.25) {};
		\node [style=none] (68) at (8.5, 0) {$h$};
		\node [style=none] (69) at (9.5, -2) {$\unit$};
		\node [style=none] (70) at (9, 2.75) {Case (ii): $ghgh=\unit$};
		\node [style=black dot] (71) at (6, 1.5) {};
		\node [style=none] (72) at (7.5, 1.5) {};
		\node [style=none] (73) at (1.5, -1) {$\colorOne{g}$};
		\node [style=none] (74) at (10.5, -1) {$\colorOne{g}$};
	\end{pgfonlayer}
	\begin{pgfonlayer}{edgelayer}
		\draw [style=d1] (2) to (14.center);
		\draw [style=d1] (3) to (13.center);
		\draw [style=c1 arrow] (2) to (8.center);
		\draw [style=c1 arrow] (3) to (9.center);
		\draw [style=c1 arrow] (6) to (10.center);
		\draw [style=c1] (8.center) to (4);
		\draw [style=c1] (9.center) to (5);
		\draw [style=c1] (10.center) to (3);
		\draw [style=c1 arrow] (53) to (55.center);
		\draw [style=c1] (55.center) to (2);
		\draw [style=c0] (4) to (5);
		\draw [style=c0] (53) to (6);
		\draw [style=d1] (58) to (67.center);
		\draw [style=d1] (59) to (66.center);
		\draw [style=c1 arrow] (59) to (64.center);
		\draw [style=c1 arrow] (62) to (65.center);
		\draw [style=c1] (64.center) to (61);
		\draw [style=c1] (65.center) to (59);
		\draw [style=c0] (60) to (61);
		\draw [style=c0] (71) to (62);
		\draw [style=c1 arrow] (58) to (72.center);
		\draw [style=c1 arrow] (60) to (56.center);
		\draw [style=c1] (58) to (56.center);
		\draw [style=c1] (71) to (72.center);
		\draw [style=c0] (3) to (2);
		\draw [style=c0] (59) to (58);
	\end{pgfonlayer}
\end{tikzpicture}
\caption{The two outcome cases in \autoref{twosym}}
\label{fig:twosym}
\end{figure}

The following proposition is supported by \autoref{fig:twosym}.

\begin{prop}\label{twosym}
    Assume \autoref{set} with $r\ge\max\{2^{n+1},5\}$.
    Suppose that $h$ is an involution.
    \begin{enumerate}
        \item If $g\in S$ traverses $\X$ symmetrically, then $g$ and $h$ commute.
        \item If $g\in S$ traverses $\X$ antisymmetrically, but not symmetrically, then $gh$ is an involution.
    \end{enumerate}
\end{prop}

\begin{prop}
\label{lem:inv_all_traverses}
    Assume \autoref{set} with $r\ge2^{n+2}$.
    Suppose that $h$ is an involution. 
    Then every element of $S\setminus\{h\}$ traverses $\X$.
\end{prop}

\begin{prop}
\label{lem:no_antisymmetry_for_anyone}
    Assume \autoref{set} with $r\ge\max\{2^{n+2},10\}$.
    Suppose that $h$ is an involution. 
    Every element of~$S$ that traverses $\X$ antisymmetrically also traverse $\X$ symmetrically.
\end{prop}

\begin{proof}[Proof of \autoref{lem:involution_case} assuming \autoref{twosym}, \autoref{lem:inv_all_traverses} and \autoref{lem:no_antisymmetry_for_anyone}]
    Assume \autoref{set} with $r\ge\max\{2^{n+2},10\}$ and suppose that $h$ is an involution. 
    We have to show that the subgroup of~$\Gamma$ generated by~$h$ is normal. 
    For this, it suffices to show that every element of the generating set $S$ commutes with~$h$.
    So let $g\in S$.
    If $g=h$, this is trivial, so let us assume that $g\neq h$.
    By \autoref{lem:inv_all_traverses}, $g$ traverses~$\X$.
    Then $g$ traverses $\X$ at both $\unit$ and $h$ by \autoref{lem:invo_implies_double_trav}.
    Hence $g$ traverses~$\X$ symmetrically or antisymmetrically (as we have already noted below \autoref{symAndAntisym}).
    By \autoref{lem:no_antisymmetry_for_anyone}, $g$~traverses~$\X$ symmetrically.
    Therefore, $g$ commutes with~$h$ by \autoref{twosym}.
\end{proof}

\subsection{Proof of \autoref{twosym}}

\begin{rem}
    When a letter $h$ represents a group element that is an involution, we do not differentiate between the letters $h$ and~$h\i$. This means that, when reducing words, we remove subwords like $hh$ as well as $hh\i$. In the remainder of the paper, we will only use~$h$.
\end{rem}

\begin{eg}\label{eg:invo_comm}
When $h$ is an involution, the first three iterated commutator words in $g$ and $h$ are:
\begin{align*}
    u_1&=ghg\i h,\\
    u_2&=ghghg\i hg\i h\text{, and}\\
    u_3&=ghghghg\i hg\i hg hg\i hg\i h.
\end{align*}
\end{eg}

\begin{lem}\label{lem:pentagon}
    Assume \autoref{set}. Suppose that $g\in S$ traverses $\X$ at $\unit$ and that $h$ is an involution. 
    Then there is no morpheme of~$\Gamma$ alternating between $g\pmo$ and $h$ whose length is odd and at most~$r$.
\end{lem}

\begin{proof}
Suppose for a contradiction that there is a morpheme $m$ of $\Gamma$ alternating between $g\pmo$ and $h$ whose length is odd and at most~$r$.
Then the first and last letter must be of the same type: either both are equal to $h$ or both are elements of~$\{g\pmo\}$.
Since $m$ is a morpheme, the first and last letter cannot be inverses of each other. 
Since $h$ is an involution, it follows that either both of them are equal to $g$ or both of them are equal to~$g\i$. 
So by applying a cyclic permutation, we obtain from $m$ a morpheme $m'$ that alternates between $h$-segments and $g\pmo$-segments, and all segments have length one except for one that is equal to~$g^2$. 

    Since $g$ traverses~$\X$ at~$\unit$, we find a cycle $O$ in $G$ labelled by $m'$ containing the vertex $\unit$ such that $\unit$ is incident with two edges labelled $g$ in~$O$.
    Hence $O$ weakly traverses~$\X$ at~$\unit$.
    By \autoref{CycleWeakTraversal}, $O$ weakly traverses~$\X$ at~$h$ as well.
    But since $g^2$ occurs only once on~$O$, at least one of the edges on the weak traversal in~$O$ through the vertex~$h$ must be labelled by~$h$, contradicting that $h$ is an involution joining the vertices $\unit$ and $h$ in the separator~$\X$.
\end{proof}

\begin{lem}[No-Rectangles Lemma]\label{lem:no-rectangles}
    Let $g\nequiv h$ be letters and $h$ an involution. Then the iterated commutator $[g,h]_n$ contains none of the following words as a cyclic subword:
    \[
        ghghghg, \quad ghg\i hghg\i, \quad g\i hghg\i hg
    \]
\end{lem}
\begin{proof}
    We proceed by induction. For $n\in\{1,2,3\}$ we verify this by hand using \autoref{eg1}. So suppose that $n\ge 4$.
    Recall that \autoref{lem:what_words_look_like} tells us that
    \begin{equation*}
        g(u_{n-1}\i)( {}^-u_{n-1}) \quad =\quad g\cdot h ghgh g\i\ldots hg\i h g\i h g\i \cdot h g h g h g\i\ldots gh g\i h g\i h.
    \end{equation*}
    By induction, $ghghghg$ or its inverse cannot be contained linearly in either $u_{n-1}\i$ or ${}^-u_{n-1}$. 
    Similarly, $ghg\i hghg\i$, $g\i hg hg\i hg$ and their inverses are not contained linearly in either $u_{n-1}\i$ or ${}^-u_{n-1}$. 
    So the result follows from inspecting up to six letters either side of the places where we concatenate $g$, $u_{n-1}\i$ and ${}^-u_{n-1}$ to form $u_n$, in the above equation.
\end{proof}

\begin{lem}[Zigzag~I]
\label{lem:ghg}
    Let $\Gamma$ be a group that is nilpotent of class~$\le n$ with generating set~$S$. Let $g\nequiv h$ be elements of $S$ and suppose that $h$ is an involution.
    Then, for every cyclic subword $u$ of $u_n=[g,h]_n$ that is a morpheme, at least one of the following is true:
    \begin{enumerate}
        \item $u$~contains $ghg$ as a cyclic subword;
        \item $u$ has the form $ghg\i h$ up to cyclic permutation and inversion, so $g$ and $h$ commute;
        \item $u$ has odd length.
    \end{enumerate}
\end{lem}

\begin{proof}
Let $u$ be a cyclic subword of~$u_n$ that is a morpheme, and suppose that (i) and (iii) fail; that is, $u$ does not contain $ghg$ as a cyclic subword and has even length.We have to derive~(ii). Note that $|u|\geq 3$. 
We also have $|u|\le 6$, since otherwise we find either $g h g\i h g h g\i$ or $g\i h g h g\i h g$ as a cyclic subword of~$u$ which contradicts the \nameref{lem:no-rectangles}~(\ref{lem:no-rectangles}).

Since $u$ does not contain $ghg$ as a cyclic subword, the word obtained from $u$ by deleting all occurrences of the letter~$h$ alternates between $g$ and~$g\i$. If  $|u|=6$, then by possibly applying a cyclic permutation or inversion to~$u$, we obtain one of the words $ghg\i hgh$ or $g\i hghg\i h$, which both contain $ghg$ as a cyclic subword. Since this is not possible, $|u|=4$ is the only possibility, and here we get (ii).
\end{proof}

\begin{lem}
\label{sublem:gh_cubic}
    Let $\Gamma$ be a group that is nilpotent of class~$\le n$ with generating set~$S$, and let $r\ge 2^{n+1}$.
    Let $g\nequiv h$ be elements of~$S$ and suppose that $h$ is an involution.
    Suppose that $(gh)^3$ is a morpheme of $\Gamma$ in~$S$.
    Then there is a morpheme $m$ of~$\Gamma$ in~$S$ of length at most $r$ such that $m$ either contains at most one of $ghg\i$ and $g\i hg$ as a subword, or $m=ghgh$, or $m$ has odd length and alternates between $h$ and $g\pmo$.
\end{lem}
\begin{proof}
    Let $u_n'$ be the word obtained from $u_n$ by repeatedly deleting sets of letters representing occurrences of $ghghgh$ as cyclic subwords, until none are left. That is, we remove strings of the form $ghghgh$ and $hg\i hg\i hg\i$ as well as pairs of strings at the start and end of $u_n$ that correspond to instances of $ghghgh$ in some cyclic permutation of $u_n$ or its inverse.
    Observe that we remove letters in sextuples and, since $|u_n|=2^{n+1}$ is not a multiple of six, $u_n'$ is not the empty word.
    Furthermore, since we have removed only morphemes, we get that $u_n'$ evaluates to~$\unit$ in~$\Gamma$. Note that $u_n'$ is reduced since it still alternates between $g\pmo$ and~$h$.
    Using the \nameref{commutatorContainsMorpheme}~(\ref{commutatorContainsMorpheme}), we find a subword~$m$ of~$u_n'$ that is a morpheme.
    Then $m$ alternates between $g\pmo$ and $h$ as well.
    In particular, the \nameref{commutatorContainsMorpheme} ensures that $m$ has length at least three.

    Since $m$ is a morpheme and $h$ is an involution, $m$ cannot start and end with~$h$.
    So $m$ starts with $g\pmo$ and ends with~$h$, or vice versa.
    If $|m|\ge 6$, then $m$ contains a cyclic subword $m'$ of length six starting with~$g$ such that $m'$ alternates between $g\pmo$ and~$h$.
    Since $u_n'$ does not contain $ghghgh$ as a cyclic subword, and $m$ is a subword of~$u_n'$, it follows that the cyclic subword $m'$ of~$m$ of length six contains either $ghg\i$ or $g\i hg$ as a subword and, by extension, so does~$m$.
    Since we are done otherwise, assume that $m$ has even length. So $m$ has length exactly four. 
    If $m$ contains $ghg\i$ or $g\i hg$ as a subword, then we are done.
    Otherwise $m$ is one of the two words $ghgh$ or $hghg$.
    In either case, $ghgh$ is a morpheme as desired.
\end{proof}

\begin{lem}[Zigzag II]
\label{lem:ghg_in}
    Let $\Gamma$ be a group that is nilpotent of class~$\le n$ and $S\se\Gamma$ a generating set.
    Let $g\nequiv h$ be elements of $S$ and suppose that $h$ is an involution. Then at least one of the following holds:
    \begin{enumerate}
        \item there is a morpheme $u$ of $\Gamma$ in $S$ of length at most $2^{n+1}$  that contains $ghg\i$ or $g\i hg$ as a subword;
        \item the word $ghgh$ is a morpheme;
        \item there is a morpheme of $\Gamma$ alternating between $h$ and $g\pmo$ whose length is odd and at most~$2^{n+1}$.
    \end{enumerate}
\end{lem}

\begin{proof}
    Using the \nameref{commutatorContainsMorpheme}~(\ref{commutatorContainsMorpheme}), we find a subword~$u$ of~$u_n$ that is a morpheme.
    Since $u_n$ alternates between $g\pmo$ and~$h$, so does~$u$.
    In particular, the \nameref{commutatorContainsMorpheme} ensures that $u$ has length at least three.
    Moreover, $u$ has length at most~$2^{n+1}$ by \autoref{commutatorLength}.  
    Then, by replacing $u$ by a cyclic permutation if necessary, we assume that $u$ starts with~$g$.
    Since we are done otherwise by (iii), we assume that $u$ has even length.
    So either $u$ has the form $(gh)^k$ for some $k\geq 2$ or contains $ghg\i$ as a subword.
In the second case we obtain (i).
In the first case we obtain (ii) for $k=2$ or we obtain~$k\geq 3$. If $k\geq 3$, then by the 
\nameref{lem:no-rectangles}~(\ref{lem:no-rectangles}) we deduce that $k=3$. 
Hence \autoref{sublem:gh_cubic} completes the proof. 
\end{proof}

\begin{proof}[Proof of \autoref{twosym}]
We assume \autoref{set} with $r\ge\max\{2^{n+1},5\}$ and suppose that $h$ is an involution. The following cases are illustrated in \autoref{fig:twosym}.

(i).
Suppose that $g\in S$ traverses~$\X$ symmetrically. We have to show that $g$ and~$h$ commute.
    Since~$g$ traverses~$\X$ symmetrically, the word $ghg$ traverses~$\X$ strongly, and $g\neq h$.
    By the \nameref{commutatorContainsMorpheme} (\ref{commutatorContainsMorpheme}), we find a morpheme $u$ of $\Gamma$ in~$S$ as a cyclic subword of~$u_n=[g,h]_n$.
    We recall that~$u_n$, and in particular~$u$, have length at most~$r$ by \autoref{commutatorLength}.
    By the \nameref{traverser} (\ref{traverser}), the morpheme $u$ cannot contain $ghg$ as a cyclic subword.
    By \autoref{lem:pentagon}, $u$ cannot have odd length.
    Hence $g$ and $h$ commute by \nameref{lem:ghg} (\ref{lem:ghg}).

(ii).
Suppose that $g\in S$ traverses~$\X$ antisymmetrically, but not symmetrically.
We have to show that $gh$ is an involution.
    Since $g$ traverses~$\X$ antisymmetrically, but not symmetrically, we may apply \autoref{obs:what_traverse}~(ii) to get that the words $ghg\i$ and $g\i hg$ strongly traverse~$\X$ (also see \autoref{fig:trav_obs}). 
    By the \nameref{commutatorContainsMorpheme} (\ref{commutatorContainsMorpheme}), we find a morpheme $u$ of $\Gamma$ in~$S$ as a subword of~$u_n=[g,h]_n$.
    As usual, $u$ has length~$\le r$ by \autoref{commutatorLength}.
    By the \nameref{traverser} (\ref{traverser}), $u$~contains neither $ghg\i$ nor $g\i hg$ as a cyclic subword. 
    Hence \nameref{lem:ghg_in} (\ref{lem:ghg_in}) with \autoref{lem:pentagon} yields that $gh$ is an involution.
\end{proof}

\subsection{Proof of \autoref{lem:inv_all_traverses}}

\begin{lem}
\label{lem:c_shape}
    Assume \autoref{set} with $r\ge2^{n+1}$. 
    Suppose that $h$ is an involution. 
    If $g\in S\sm\{h\}$ does not traverse $\X$, then all four vertices $g$, $g\i$, $hg$ and $hg\i$ lie in the same $r$-local component at $\X$.
\end{lem}
\begin{proof}
Since $g$ does not traverse~$\X$ at~$\unit$, there is a local component $C$ at~$\X$ that contains both $g\i$ and~$g$.
Similarly, since $g$ does not traverse~$\X$ at~$h$, there is a local component $D$ at~$\X$ that contains both $hg\i$ and~$hg$.
If $C=D$ we are done, so suppose for a contradiction that $C\neq D$.
Then the words $ghg$, $ghg\i$ and $g\i hg$ strongly traverse~$\X$ (the traversals witnessing this start at the vertices $g\i$, $g\i$ and $g$, respectively). 
Using the \nameref{commutatorContainsMorpheme}~(\ref{commutatorContainsMorpheme}), we find a subword $u$ of $u_n=[g,h]_n$ that is a morpheme, and it is of length at most~$r$ by \autoref{commutatorLength}. 
By the \nameref{traverser}~(\ref{traverser}), none of $ghg$, $ghg\i$ or $g\i hg$ is a cyclic subword of~$u$. 
As a subword of~$u_n$, the morpheme $u$ alternates between $g\pmo$ and $h$, and so the three excluded cyclic subwords force $u$ to contain no more than one instance of~$g\pmo$.
Since $h$ is an involution and $u$ is a morpheme, $u$ does not start and end with~$h$. 
Thus $u$ is of the form $gh$ up to cyclic permutation or inversion, contradicting that $u$ -- as a morpheme of length two provided by the the \nameref{commutatorContainsMorpheme}~(\ref{commutatorContainsMorpheme}) -- should have the form $g^{\pm 2}$ or~$h^2$.
\end{proof}

\begin{lem}[Zigzag III]
\label{lem:ghkh}
    Let $\Gamma$ be a group that is nilpotent of class~$\le n$ and $S\se\Gamma$ a generating set.
    Let $g,h,k\neq\unit$ be pairwise inequivalent elements of~$S$. 
    Suppose that $h$ is an involution. 
    Let $\ell$ be a morpheme of~$\Gamma$ in $S$ with $|\ell|\ge 3$ that is a linear subword of the iterated commutator word $\ell_n=[g,hkh]_n$. 
    Then either $\ell$ contains one of $ghk$ or $g\i hk$ as a cyclic subword, or $\ell\in\{khghk\i h,\; gkh,\; g\i kh\}$ up to cyclic permutation and inversion.
\end{lem}

\begin{proof}
    \textbf{Case~1:} $\ell$ contains a cyclic subword of the form $g^{\alpha}hk^{\beta}hg^{\gamma}$ for $\alpha,\beta,\gamma\in\{\pm1\}$.
    By applying a cyclic permutation if necessary, assume that $\beta=1$. 
    So $\ell$ contains $ghk$ or $g\i hk$ as a cyclic subword.
    
    \textbf{Case~2:} not Case~1. Then $\ell$ contains at most one instance of~$g\pmo$. 
    If $\ell$ does not contain~$g\pmo$, then by $g,h,k\neq\unit$ the word $\ell$ is one of $hkh$ or $hk\i h$, which is a contradiction since a morpheme cannot start and end with an involution.
    Hence $\ell$ contains exactly one instance of $g\pmo$.

\begin{sublem}\label{boring-calc}
$\ell$ or $\ell\i$ is a linear subword of the word $hkhghk\i h$.
\end{sublem}

    \begin{cproof}
    By symmetry, we assume that $\ell$ contains an instance of~$g$.
    Then letters adjacent to $g$ in $\ell$ can only be the involution $h$. Adjacent to them we can only have $k$ or $k^{-1}$.
    Since $ghk$ is excluded, the later one can only be a $k^{-1}$. Since  $g\i hk$ is excluded as a cyclic subword, its inverse $k\i hg$ is excluded, so the first one can only be~$k$. 
    The next two letters can only be the involution $h$, and after that a~$g\pmo$, which is impossible as $\ell$ has only one occurrence of $g\pmo$. Thus $\ell$ is a linear subword of $hkhghk\i h$.
    \end{cproof}\medskip

    By inverting $\ell$ if necessary, assume via \autoref{boring-calc} that $\ell$ is a linear subword of the word $hkhghk\i h$.
    An exhaustive list of the possibilities for $\ell$ is provided below. Instances of $ghk$ or $g\i hk$ as cyclic subwords have been coloured in blue. Possibilities that cannot be morphemes, because they end with the inverse of the letter they start with, have also been highlighted, where the start and end are coloured in red.
    Possibilities that are too short are coloured grey.
    Cyclic permutations of $khghk\i h$ are coloured teal.
    \[
    \begin{array}{rrrr}
        \r{h}khghk\i \r{h} & \g{hkhghk\i} & \r{h}khg\r{h} & \b{hk}h\b{g} \\
        \g{khghk\i h} & \r{k}hgh\r{k\i} & \b{k}h\b{gh} & khg \\
        \r{h}ghk\i \r{h} & \b{hg}h\b{k\i} & \r{h}g\r{h} & \grey{hg} \\
        \b{g}h\b{k\i h} & ghk\i & \grey{gh} & \grey{g}
    \end{array}
    \]
    By inspection, we obtain our required result.
\end{proof}

\begin{figure}[ht]
    \centering
\begin{tikzpicture}
	\begin{pgfonlayer}{nodelayer}
		\node [style=black dot] (0) at (0, 1.5) {};
		\node [style=black dot] (1) at (0, -1.5) {};
		\node [style=black dot] (2) at (2, 1.5) {};
		\node [style=black dot] (3) at (-2, 2.5) {};
		\node [style=black dot] (4) at (-2, 0.5) {};
		\node [style=black dot] (5) at (2, -1.5) {};
		\node [style=black dot] (6) at (-2, -2.5) {};
		\node [style=none] (7) at (1, 1.5) {};
		\node [style=none] (11) at (1, -1.5) {};
		\node [style=none] (12) at (0, 3) {};
		\node [style=none] (13) at (0, -3) {};
		\node [style=black dot] (14) at (-2, -0.5) {};
		\node [style=none] (15) at (-2, 1.5) {$hg\pmo$};
		\node [style=none] (16) at (-2, -1.5) {$g\pmo$};
		\node [style=none] (17) at (2.5, -1.5) {$k$};
		\node [style=none] (18) at (2.75, 1.5) {$hk$};
		\node [style=none] (19) at (0.5, 0) {$h$};
		\node [style=none] (21) at (0.5, -2) {$\unit$};
		\node [style=none] (22) at (-2, 3) {};
		\node [style=none] (23) at (-2, -3) {};
		\node [style=none] (24) at (-4, 0) {$C$};
	\end{pgfonlayer}
	\begin{pgfonlayer}{edgelayer}
		\draw [style=d1] (12.center) to (0);
		\draw [style=d1] (1) to (13.center);
		\draw [style=c1] (3) to (0);
		\draw [style=c1] (0) to (4);
		\draw [style=c1] (6) to (1);
		\draw [style=c2 arrow] (0) to (7.center);
		\draw [style=c2 arrow] (1) to (11.center);
		\draw [style=c2] (11.center) to (5);
		\draw [style=c2] (7.center) to (2);
		\draw [style=c1] (14) to (1);
		\draw [style=c0] (0) to (1);
		\draw [style=d1] (23.center)
			 to [bend right=90, looseness=0.75] (22.center)
			 to [bend left=270, looseness=0.75] cycle;
	\end{pgfonlayer}
\end{tikzpicture}
\caption{The situation in the proof of \autoref{lem:inv_all_traverses}. Blue edges are labelled by~$g$, orange edges by~$k$}
\label{fig:ghk}
\end{figure}
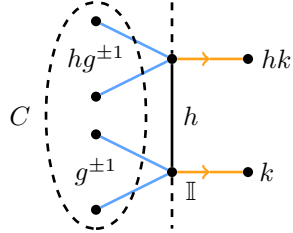
\begin{proof}[Proof of \autoref{lem:inv_all_traverses}]
    Assume \autoref{set} with $r\ge \max\{2^{n+2},10\}$.
    Suppose that $h$ is an involution.
    We have to show that every element of~$S\sm\{h\}$ traverses~$\X$, and assume for a contradiction that some $g\in S\sm\{h\}$ does not traverse~$\X$.
    
    By \autoref{lem:c_shape}, all four vertices $g$, $g\i$, $hg$ and $hg\i$ lie in the same $r$-local component at~$\X$. 
    Call this local component~$C$.
    For a depiction of the situation, see \autoref{fig:ghk}.
    Since $\X$ is an $r$-local $2$-separator, there is another local component at~$\X$ besides~$C$, so there
    exists $k\in S\sm\{h,g\pmo\}$ such that either the vertex $k$ or $hk$ does not lie in~$C$. 
    Since the automorphism of~$G$ given by left-multiplication with $h$ preserves~$C$ but swaps the vertices $k$ and~$hk$, both vertices $k$ and $hk$ lie outside of~$C$.
    We conclude that both words $ghk$ and $g\i hk$ strongly traverse~$\X$.
    Using the \nameref{commutatorContainsMorpheme}~(\ref{commutatorContainsMorpheme}), we find a linear subword $\ell$ of the iterated commutator word $\ell_n=[g,hkh]_n$ such that~$\ell$ is a morpheme, and this lemma further ensures~$|\ell|\ge 3$. 
    By the \nameref{traverser}~(\ref{traverser}) and since $|\ell|\le|\ell_n|=2^{n+2}\le r$, the morpheme $\ell$ cannot contain either of $ghk$ or $g\i hk$ as a cylic subword. 
    Thus, by \nameref{lem:ghkh} (\ref{lem:ghkh}), without loss of generality we get $\ell\in\{khghk\i h,\;gkh,\; g\i kh\}$. 
    Suppose first that $\ell\in\{gkh,\; g\i kh\}$. Then, since $\ell$ labels a triangle, either the vertices $g\i$ and $k$ are adjacent or the vertices $g$ and $k$ are adjacent, contradicting \autoref{locSepSeps} as they lie in distinct $r$-local components at~$\X$. 
    So we have $\ell=khghk\i h$.
    
    If $k$ does not traverse $\X$, then we additionally get that the word $ghk\i$ traverses strongly. Since $ghk\i$ is a subword of~$\ell$, this contradicts the \nameref{traverser} (\ref{traverser}).
    
    Similarly, if $k$ traverses $\X$ antisymmetrically, but not symmetrically, we apply \autoref{obs:what_traverse} to get that $k\i hk$ strongly traverses, again a contradiction.
    So we may assume that $k$ traverses $\X$ symmetrically and thus, by \autoref{twosym}, $k$ and $h$ commute. Applying this fact to the morpheme $\ell=khghk\i h$ yields that $g=h$, a contradiction.
\end{proof}

\subsection{Proof of \autoref{lem:no_antisymmetry_for_anyone}}

\begin{lem}\label{lem:add_z_invo}
    Assume \autoref{set} with~$r\ge 4$.
    Suppose that $h$ is an involution. Let $\{A,B\}$ be an $r$-local $2$-separation of~$G$ with separator~$\X$.
    Let $g\in S\sm\{h\}$ such that $z:=gh$ is an involution.
    Then $\{A,B\}$ is an $\frac{r}{2}$-local $2$-separation of $G':=\cay{\Gamma}{S\cup\{z\}}$ with separator~$\X$.
\end{lem}

\begin{proof}
    If $z$ is in~$S$, there is nothing to show, so we may assume that this is not the case.
    We claim that 
    \begin{equation}\label{eq:zNbhd}
        N_{G'}(\X)=N_G(\X).
    \end{equation}
    Indeed, since $z$ is an involution, both vertices $\unit$ and $h$ get at most one possibly new neighbour in~$G'$, namely $z$ and~$hz$, respectively. 
    We have
    \[
        z=gh=hg\i
        \quad\text{and}\quad
        hz=hgh=hhg\i=g\i
    \]
    so neither $z$ nor $hz$ is a new neighbour.
    
    Let us show now that $\{A,B\}$ is an $\frac{r}{2}$-local $2$-separation of~$G'$ with separator~$\X$.
    Since $G$ has no $r$-local $2$-separator by assumption, for each $x\in X$ we may let $\{A_x,B_x\}$ denote the $2$-separation of $B_r(x,G)$ induced by~$\{A,B\}$.
    Let $A'_x,B'_x$ be obtained from $A_x,B_x$ by taking the intersection with the vertex set of $B_{r/2}(x,G')$.
    Then the sets $\{A'_x,B'_x\}$ are $2$-separations of $B_{r/2}(x,G')$, and these $2$-separations are compatible, by \autoref{CompatibleSeparationsAddingEdges} applied with $r:=r$, $d:=2$, $F$ the set of edges in~$G'$ labelled by~$gh$, and $r':=r/2$, using $r'\cdot d=r$ and that the vertices in $\X$ have distance $1\le r'/2=r/4$ in~$G'$.
    Hence the $2$-separations $\{A'_x,B'_x\}$ make $\X$ into the separator of an $\frac{r}{2}$-local $2$-separation $\{A,B\}$ of~$G'$ by \autoref{LocalSepnCompatibleII}.
    It follows from (\ref{eq:zNbhd}) that $A'=A$ and $B'=B$.
\end{proof}

\begin{proof}[Proof of \autoref{lem:no_antisymmetry_for_anyone}]
Assume \autoref{set} with $r\ge\max\{2^{n+2},10\}$.
Suppose that $h$ is an involution.
We have to show that every element of~$S$ that traverses~$\X$ antisymmetrically also traverses $\X$ symmetrically.
Suppose for a contradiction that there is some $g\in S$ which traverses $\X$ antisymmetrically but not symmetrically. 
Let $\{A,B\}$ be an $r$-local $2$-separation with separator $\X$ such that $g\i$ and $hg\i$ lie on opposite sides of~$\{A,B\}$. 
Then $z:=gh$ is an involution by \autoref{twosym}. 
Using \autoref{lem:add_z_invo}, we obtain a graph $G':=\cay{\Gamma}{S\cup\{z\}}$ where $\{A,B\}$ is an $\frac{r}{2}$-local $2$-separation with separator~$\X$. 
As $hz=g\i$ and $z=hg\i$ (see (\ref{eq:zNbhd})) lie on opposite sides of~$\{A,B\}$, the word $zhz$ strongly traverses~$\X$.
So by the \nameref{traverser} (\ref{traverser}), no morpheme of~$\Gamma$ in~$S$ contains $zhz$ as a cyclic subword. 
Using the \nameref{commutatorContainsMorpheme} (\ref{commutatorContainsMorpheme}), we find a subword $m$ of $[z,h]_n$ that is a morpheme of length at least three.
Since $m$ does not contain $zhz$ as a cyclic subword and $|m|\ge 3$, the only possibility for~$m$ is $m=hzh$.
Hence $h=g$, a contradiction.
\end{proof}

This complete the proof of \autoref{3-con}.

\begin{rem}
    In the proof of \autoref{3-con}, we only use that the group is finite when we have already found a specific generating set and wish to infer statements on the global structure of the group from this generating set.
    In fact, the proof would still work if we allow the group to be infinite or $r=\infty$, provided that we allow the cyclic groups in the statement of \autoref{3-con} to be infinite as well.
    We state \autoref{3-con} for finite groups and $r<\infty$ to emphasise what is new.
\end{rem}

\bibliographystyle{amsplain}
\bibliography{literatur}
\end{document}